\documentclass{amsart}
\usepackage[utf8]{inputenc}
\usepackage{amsfonts, amssymb}
\usepackage{amssymb,amsmath,amscd,euscript,verbatim,array}
\usepackage{geometry}
\usepackage{anysize}
\usepackage{fancyhdr}
\usepackage{indentfirst}
\usepackage{graphicx}
\usepackage{color}
\usepackage{ifpdf}
\usepackage{marginnote}
\usepackage{enumitem}
\usepackage{mathrsfs}
\usepackage{amsthm}
\IfFileExists{CJK.sty}{\usepackage{CJK}}{} 
\usepackage{bm}
\setlist[itemize]{leftmargin=2em}

\newcommand{\N}{\mathbb{N}}
\newcommand{\Z}{\mathbb{Z}}
\newcommand{\R}{\mathbb{R}}

\newcommand{\Lip}{\mathrm{Lip}}
\newcommand{\Id}{\mathrm{Id}}

\newcommand{\la}{\lambda}
\newcommand{\al}{\alpha}

\newcommand{\T}{\mathbb{T}}

\newtheorem{Theorem}{Theorem}

\newtheorem{Definition}{Definition}[section]
\newtheorem{Proposition}[Definition]{Proposition}
\newtheorem{Corollary}[Theorem]{Corollary}
\newtheorem{Lemma}[Definition]{Lemma}
\newtheorem{Remark}[Theorem]{Remark}

\newcommand{\bthm}{\begin{Theorem}}
	\newcommand{\ethm}{\end{Theorem}}
\newcommand{\bpr}{\begin{Proposition}}
	\newcommand{\epr}{\end{Proposition}}
\newcommand{\blm}{\begin{Lemma}}
	\newcommand{\elm}{\end{Lemma}}
\newcommand{\bex}{\begin{Exercise}}
	\newcommand{\eex}{\end{Exercise}}
\newcommand{\be}{\begin{equation}}
	\newcommand{\ee}{\end{equation}}
\newcommand{\beal}{\begin{aligned}}
	\newcommand{\enal}{\end{aligned}}
\newcommand{\brm}{\begin{Remark}}
	\newcommand{\erm}{\end{Remark}}

\newcommand{\SharpOriginalFigure}{%
  \leavevmode\hbox{%
    \pdfliteral{q 155.905511811 0 0 148.223211229 0 0 cm
      BI /W 690 /H 656 /BPC 8 /CS /RGB
      /F [/ASCIIHexDecode /FlateDecode] ID
78DAECBD77705457F6EFEBBF5ED5AB5BAFEABDAA5773DF1D7B82C7BFB93F8F67EC190FC1040B8CC902842410207212199405
882424404242121202A18472CE39E79C73EE9C73753A9DBBF5764B1E8C898ADD2D697DEA582537AD73F6597B9FB3BE6B87B5
2726000000000000000000000000000000000000000000000000000000000000000000000000000000000000000000000000
000000000000000000000000000000000000000000000000000000E6018D562353616285902FE372A44CA6984A1791A942E2
D441139218220A4B4CE34A5942194FA21029D4729D4E07760300000080A5814AA3423280236190F9B846427972FB8BA0CA5B
F7F3ED9C53AD2FC46C3A15B6FA48C83F0E3FFBFBE167DF1C0FFDFE4CC49A4BB15BAE671E7A547435B4D633B737BE8BD28CD4
020F63634A8956A7057B02000000C0A243A95108645C121F573E921D5C75C721C9E278E8BFAEC46FF72B754AED08AB1E2D18
66F572A52CB556FDE65F29D472A688DA436D291DCA8C6D09F4C8B33B1BB9FE54F8EA9B99475E353D6925D6304414A431B43A
0D5818000000004C1C9D4E87BC36454028E84FF6CC3F7FE2E50AC764CB574D7EFDF4F6599F53AA14D7E14A022B6E9C7BB5E1
7CF4C6A0CA9B8DF872A6982A5349C1E0000000006082A0A89F2FE376521A032AAEA3E8DE39755F764F2CA614CFEF55F0DC91
C8461FBB28B32BF13BE25A8270DC61915C00C60700000000134133A9072A47F3DCB38E9D8958F7B2EE214B425BE88BF6D05A
3C0B2E9C0E5FF3A4CC6580D109DA00000000008C8E50CEAF192B72CB38783E7A63564F8C81270472A5AC909A3B275FAEF42F
771D64744BE7BB6B020000000080E98029A5BDB4D6DB3927ECA2CC32BA228D58123EC6D10F5B84AD7E59F7802A20A8344AA8
1D00000000300C5A9D86262487D67A9E0C5B195EFFE8AD7504C6822624DDC93971297673E1400AD209504D0000000060802E
823A5CF1C5D8CDB7B28FA3C0DCD48A578F2BB58B32F32CB830CE19526954505F00000000B04070A48C67D5774E85FF90DF9F
64B2855469948F4BECCFBD322B1E4C132B84506B0000000030BF68B49A41669743B2C58D4C5BBA886CFA05AEC3959C8E581B
5C75CB000B220000000060F9A050CB4B8732CE44AE8B68F05944698705329E73AAB54B9ACD30AB17B22503000000C0DC91A9
A4518DBE6723D7D58E172FC6F20757DD3E1BF96333B112D62600000000C05C10CAF83E25F697E2B6E2B9C38BF72E727AE34E
85AD2A1A4895AB30A8530000000098056C09E346E66197B4FD1C2973B1DF4B0BB1FA64D8AAE4F6179812764F000000008099
4117511C93F7DECB3BB364DC288E337426725D6C7380542981FA0500000080694B02B27DE26EAFC28B4B6C5F429A90743672
7D54A32F080300000000980E2C31DD2179AF4F89BD422D5F7A77C714539130886C780C1B3103000000C0C7E1CBB8AEE9071E
145E92AB654BF51EE922D29988B5F1ADC14B52F600000000C0BC80C2E75BD9C76FE79C9028444BFB4EF1DCE193612BB37B62
4C641307000000003029345AB56FA9836B9A8D40C65B0EF73BC8EC3A1EFA7D1DAE4437A183DA07000000803789690EB812B7
9D26222D9F5B2E1DCE3C15FEC300A3136A1F000000005E5332987136727D1FBD7DB9DD7874D3938B313F3344146803000000
00801861F59D8E58533B5E68F8BE749D4E2753625C099B2DA2E3B9C3E39C219A80C4113385185FAE9419A63C0F0A2FDECA3E
0A6B15010000004024175E4B308F6E7EA2D4280C795DA54AC114D29AC62A636B83AF271CBB1A6961F9F0F77BBD7E773CF807
A7D883DED98EF17521541E11C986852E894AA3BC18BBF955D3139879080000002C731E165DB99777DA908B0E90A3E749D859
6D7176A13F232560F5E8F7D63E5FA063DFE33FA263EA776BEFCF2D1FFC2EAE2644AD3184A726F1C78F877E5F8F2B83F60000
00002C5BB27A622EC56EA6080886BC6837B1C531C6C6CA1BC9803FECF3FDE37EDF2F0F077F7BECF98AD3616BD1712274D5F1
172B0E057C8DD4423BAEC1603B20EBB78A8E58CB10C30403000000603982E78E9C0A5F5D8F2B3198E745F4105B8F05AF447A
60BFDF9FD1CFA321FF728CDF97DB9AD43A5ACFE4D13802E630A5BF0BD7FAB2F4D1D9C81F457281210D32D96D7216521B0100
0000CB0DAD4EE392B6FF45ADA75C65B81C86542EF1F8B3D5FB7DFF8424C17EDF2F4F84AEAC1B28339D6C0132A5F44CE4BADC
BE78681E000000C0B222AE25C831792F4762B82D92B55AAD47DA059B277F99EA25B81CB59325A09B9A593AC80D275EFEDBC0
432A000000006044C802FC9988356DA45A432E452CEBCE39E0FFD57EDF3F23557028F0EB71C690691A27B0E2C69D9C132A8D
12DA09000000B01C70CF3AFABCD60333ECA681575FEDB1F1D3771458FB7C915C1F6EC8C90C3342A6929E0A5F5D3E9203ED04
00000058F2140EA45C8EDB4A17910D79512A8764FBF4EBFFCC28F80B4348356513558F159C8958CBC738D05A00000080250C
A6929E8FDE58369465E08C3DA55D398702BF9E9C51F0C7BB29E7152A539FE77F2BFBD8B3EADBB07112000000B08409ADF3BA
97774628E71BF8BAE1A58F0FFAFF756AF820A72DD130B989E60259803FF6E29F78EE08B4190000006049421512CF44ACEDA4
344E183C04764F3969F3E4AB295550D693A3590CB98543EBEEDFCA3E0E699001000080258967C1F990EABB52A5D8F0973E1C
FCCDD4A4022BEF2FDA71F5263BD5F04DA40AF1F1D07FB7916AA1E5000000004B8C2166CFD9A8F5639C41A35CFD48C8B7AF55
41EB78DDA25005139319A11D922D609522000000B0C4B8957DEC55A31FA6941AE5EAC79EFF7B2A5381B5F7172D63355AEDE2
50056AADFA4CC4DADAF162683F000000C092A18BDA74EED506026FD45805708ADBFF3A5941796FAE46AB592CA62B194ABF12
BFDDC0DB4C03000000C0C27123C336A1ED994C8519AB005115FE07037E5983105B1DACD2A8168BE9743A9D5DD48F55A3F9D0
8A0000008025C000A3F3DCAB0D44DE9811CB50D95B603B95AFC0F78FD75EED95C8458BC8804503A9D712CC613102000000B0
04B89777E655939F542931621946A90393B90DFF6CE3F7A5E5C3FFD945684631F86231A04AA33A15B61A1623000000008B1D
226FCC2EEAC75176BFD14B722E6A933EDDF1647AC33BA967C532E12232634657A44BDAFE45A46400000000E05D7C4B1D9ED7
788815C677C1B5FDA5B64FBF41AA60AABB20A3F59552B568A6F0C994D2A3CFBF1D630F408B0200000016291285E84CE4BA2E
4A93291446ABD5DE4C3EBEDFEFCBA91E039BC7FF95D3192F90F23EF47D9D89AD5E0CADF57A5878191A15000000B04849687B
E655708187B14DA43CC394FE33E166FB7CFF38290CFEB4D7EB770145EE4DF84A2A8F28968924BF1C62BE8443E28E97F56799
54FA209A9074F4F977B09122000000B048B914B7B56224C7A426CF370E559E8BDC8484013A6CFCBEB4F2FEDCC2F37757A32C
424ABD42CB1F85963F443FBD731C4F859859DCFF1D954F34297BDECA3E16DB1208ED0A000000587434E2CB1D92F730C41453
2BD82865F06A9CF9E1A0EFF63DFE233AF6FBFEC9FAF117568F7EFFE661EDA3FFA46EB0D4A4B220B693EBEDA2CC6089220000
00B0E8F0C83F1BDF1A84197541E287104984994D3157E2769E0E5B87E4818DDF57D63E7F404A001DFB7DBFB47DFAF7132F56
A17FAAE8CB37A92C885A9DEE54F80F26324F0300000000A68958213C13B9AE9FDE61CA8554281403A4EEB4C6C8FB5917EC13
775F4B3047C7EDF4537135CFEA07CB494C9C099639A6D9DFB3E03C3430000000601191DA11E65970812B658129E697C93987
DF9AC24A4F000080A5874EA7C394123EC66188A924FEF838676898D53BC4ECEE6774F4D1DBD081A2DD4166D708AB0FC71DA6
08F02C095D28E7CBD53230DDC77148B2281C48812D801702E754EBFCBE44B0030000C0BC2057C91822CA28BBBF83D2503192
93DE1511D5E81B5475CBBBF8DABDDCD3D7336C5DD30E38A558D927ED418763CA5ED7341BF7ACA31E79671F97E813F2C4B706
E5F4C6558DE6B793EB907E405A82276583FB7B13AA9070EE95199E3702A65808255BD09FEC92BA1F0C314D941A0552FE62B9
4028E3F130363AF832AE482E102B8432150653370160D9C2C738639C8116627556770CD200C8D15F8EDFE69862F9B8C4F155
D393BCBE84065CD920A38B2A2030C554A19C87A9A4329514BD3DD8123A9137D6476FAF1D2FCAEC8E7A59E7E55578D131D9E2
72DC36D7741BEFE2ABE1F58FF2FA125B493563EC0196840E0A21A6D9DFBFDC85072BEB1706818C7BE4D9DFF918174CF11652
A5842D6120A18E347F1FBDAD8D545B3D5690DF9F98D11599D81612DBF234AAF1310A01A29B9E24B63D4B6A7F9ED9FDAA6820
AD6EBCB893D2D84FEFC0718650BC20518820B334002C61504430CE19AA192B8C68F0464AC03E71D7A3A22B28D89FCCCC3FD7
679F2E22D7E34AA29B9FA0335F8DDF897E86D4DC4302A383DC8084047ABD2C4F9BDB279A970D67A9B52A687E0BC4CDCCC359
DDD16007A4C03912068E3BDC4D6DAE18C949E9087D517BDFA7D8FE76F60987E4BD9762373B245B78E49DF52D750AA9BE8BD4
7B5CCBD3D8E640240C9ED7780457DDF629B177CF3A722DC1FC62EC66A7546BF4CDA02A77A416D0A99048200BF0CBF6110680
250947CA6A27D7A3F7007A853AA75AA377C220A36B41AF38C8E84EE978A9D71E49BBBC0A2FC6B70637E2CB09BC11A952BC7C
CC4EE6E32E446F22F046A1052E10BA095DE9508653AAD5B2B5005FC6454AA09D5C97DB1B175AEBE9917FF65AA2F995F8EDDE
C5D792DB5F34E0CBC6D80362B97086E7E4F4505BD009032B6F208580E4845FA95362DBB3065C299E3B02D33B016051C392D0
6BC78B512C801E6D140EA087DDC005C094D2BAF11274E9AB09E677724EC6B706B510AB2902824AB3F4C3E784D6E02765CEA6
93E578492294F10E3FFB66B945B25C296B88D98D02F9C8469FBBB9A72FC76DBB91698BB437920742197F7EAF25920BAAC70A
FCCB5DAF2598DFCB3B33F508D345649D4E0BCD0F0016110219B7015FEE57E6EC9862195EFF8825A61BB73C0AB5BC7AACF07E
C139FBA4DD8F8AAFE6F6C50F303A8572FE12AE02E7D47D05FDC94A8D025AE382825A78F970F6B2E819C0B823ACBECAD1BC97
750F900CB04FDCF5BCC6A3015F66B004594D840ADF52C7AB093BFD4A1D4B873389BC3168DE0060FAA0E7B487DA125273CF25
6D7F74D313530BA35079F2FB13DD320EBAA61D40914E33B18A29A6CE7D6283A9215608CF46AE47EF7068900B8A6E4297D2F1
D223CF6E693FD104DE68ED78714483CFADECE3F689E6110DDE466C5A329534B93DD43165EF9D9C93D93DB114015E0BFD0600
60AAB0258C8CAE2814A50694BBA1DF4DB9A8E8B5F6A4CCC53179EF9332E7CA915C8A80B094DE2DC58369C855B1C43468930B
0D8A588FBFF87E492EAC13C8787DF4B68CEE28AFC28BD712CC436AEEF6D25A4DA778E5C35937326D3DF2CE960C65B2247468
8A00605268B4EA1E5A8B77B1FDF50CDB6642E56229360F63C7B6043A24EFF52C385F34984AE28F2F0D6DF0B0E87252FB0BD3
DCFB60E9712662ED00A36329DD11D2932DC4AA88066FD7F48337330F578CE498EC73513890E298B2D7B7D4A193D2A850CBA1
35028029205349D1B3E99266F3ACFAAE482E587CE5574A533A5EDA27EF417147D1402A45805FEC0BA5CF476F6C265641CB34
00A8A9F897BB4435FA2E8DDBA10B49B5E3C50115D7ED9376FB97BB4E2E1C367530A534ACFEA17DE2AED4CE7026F48F0180F1
C36DCEAB263FE754EBD2A1CC457D2328D0C8E87EE59862E95970A17C386B72BEC1A284C01BBD12BF9DC4C741E334842A98D0
558EE4221FBAE8F580885C359AEF536C7F35616758FDA345B7774637B5D92DFDC083C28B038C4E689600602C6822120A28DC
B38E0E32BB96C61DA9B42A147138265B7A175FABC3952CC6F5D1195D917E654E1C2913DAA7A18431DB36F89BC5BB2B075BC2
A8192BF42BD34FEF8F6AF45DBC6B7330A5E445ED7D24ECABC70A946A589E00008686C81B7B5874D9ABF022D2064BECD60432
EEF31A0FFBA43DA1755EBDB4D6C5B500EA5EDE998CAE28990A83266A302EC66EEE20372CBA6223D1DB42AC0EA870BB96601E
D3ECBF347204150DA4DA27EE4AE9788929A5D03201C09092C0B3E09C6FA9237FE9A6D9C77186BC0A2EB8A61F48EB0CA708F0
8BA5D8E75E6DE820D7431335183A9D2EA8D23DBAE9C9222AB35AABEAA3B74D0DC707545CE72DADBDB67BA8AD4E29D6110DDE
8B719A13002C46A842C283C28B4FCA5D444B3A11D0148DF872D7349BFBF9E7EBC68B4D7F563FAA9ACB715B916683566A3855
30A12B1DCE744B3FB0580A4CE2E332BB5FB9651CBC9D736288D9B35425FDF58C434155EE90241900161AAE94E95362EF536C
BF7CB2E9CA55B2F0FA470EC97B221B1F8F73864CB9A8C583E90F0A2F41A602034313128FBDF8A7E92F5D112B4448DCA216E2
9462D544A858EAC10BF14ECEC9A0AA5B5258A20B000B864C853DAFB9E759709E26242DB77BEFA777B8671DBD9373AA6A34DF
64872C9F56DE8C6D0E14C31E7306EE2ED0E94EBC5C41E69BF430D308AB2FA2C1C73E69376A21CBA45E2802FCEDECE3CF6BEE
436E6400589057DF842EB9FDC5ADEC63CB79ED4F52FB73C714CBA8465F127FDC048B87DEF995A3F99AA5986ACF94D1EAB477
734F170EA49866F1447241F970F6EDEC13287636CD76BB70E0B923D7330EC5B706412B058079A77AACE066E661F47399DB61
90D185ECE09167D744A8546B4D6BFBC5C93C7B5DD0560DDF5710DB12105C75DB04CB36C61E08AB7B609FB82BBB276679D64E
17B5C921794FE95006345400984788BCD1BBB9A7629AFDC11413935BC63CAFF5704AB1CEE88AE2CBB82652AA71CE9043F2DE
E5160C9A842A98D035112A9C52AC4CAA5472B5AC66BCF04EEEA93B39271922CA72AEA08A91DC6B89E67DF47668AB00305F4E
30A4FA8E4F89FD725874307DAAC70A9C53AD83AB6E8D73064DA13CC583E98F8AAEC25443A3C096D08FBDF8CE74CA43179113
5A83ED937625778442ED208114DF1AEC9671482883371800CC034583A928DCE8A2368129DE02B960CFFCF377734F37112A8D
3E9AFFA2F67E54A31FACD1360A5A9DF6E8F36F4D64F33E14143F2EB1774EDD37C4EA81AAF98F6C633C2CBA125C7D1B4C0100
738426227B165C583EF396678A4EA70BAB7BE8946295DB1B67DC8406B7B28F15F42743AE5763A902D4061AF1465EEEA7D428
AB46F35DD30EF897BBAA342AA897DF8AA53653A8230058ECC4B6043E2ABA0A79F53F4EE1408A43B2C5AB265F149218AB0C17
6236B592AAA12E8CA50AFC4A9D92DB5F18B10C021937B5336C394F2CFC382A8D529FBE29FD00244306805933C8ECBE9F6F57
D09F04A6984624D27E3DE35040B91B9E3B6CF8ABABB5AAB391EB47D97D50114641A7D325B787FA95391BAB0044DED8F31A0F
244D7BA82D501D1F8221A23C2ABEFAAAC90F4C0100B323BCFE118A8030480E363D7818C7ABE0C2FDFC73867F338F73069D53
AD4D7F01824CAEE3F03444AA6A684CD9D5AF68EE94D73463D54D58591D565A2B2DAB955636EAFFB7B15DD6DE2BEF1B56E048
2A265B2396684D5D154CE81AF0658E299646B97A37B5D933FFFCEDEC134B785F92F9A2115FEE98B2972E2283290060A6F4D2
5A918FAB18C90553CC24665707565CBF91695B8F2BD5EA0CE7CBAAC70ABD0A2E9AE00234AD7602C980E1716563872CA75412
10CE77F6641FBD46DF77966A7E84B2E500F9C7DD44B3BDC455DB096BCC09ABB6117EDC4344C7CFFB493B6CC97B4E506D2FD1
AEDC663D08E626668B905AE8195450196AB9C214730B530484E3A1DF1BDCBC9ADAB162E7D47D21D577E1E99B0E02192FACFE
217A48C114003053A21A7DFDCA9C642A18839B8DE99C522C4B86320C96E628A5E3E5B3EABBA6B339854CAE2390554D1DB2D8
74119201D667A81B2C49EBCC89DB0E92AD4E538FDBD35DBDD88F9EF142E305F199A29C12496E99A4A21ECB2BD7FF925E207E
9522F40FE3DFF6E59CBFCE3C789E86F48399855E2D981F21DBB9329E46F24B6BA503A30A9E40633A95AED4280E06FE97DA80
4B51E46A595E6FBC7DD2AEACEE6878E8A64F07B9DE31C51267DA7B9A0080A931C6EE7F5474196614CC9A8CEE48F4E6C9E989
55A8E506B81C9204712D4F8DBE439C5A3341A4AA6A5B640111FCC39769EB7611D1B1FB18E5B23B2B244650562BC591940AE5
8C237D1647DDDC298BCF12219D70E8226DCB01F25A73C29E13949B3E6CA422903C30852106AD4E73326C15554834CCE5505D
27B63DBB96685E8F2B85C76D4608E5FC88066FBF52273005004C9FE4F6D0C725F6304C39172A4672903048EB8C90A9B085BE
9647DED9BCBE04855A66B437AD58DB3DA8884C161EBB465FBF9BB8D98674E40ADDE739AFB446CAE5CF73443F825326E588AE
DD61ED3E4E596B4E443F1F0473AB9B302AC3982923902AB04FB26823D719E05A3C8C8DFC9A7DF29E51F6003C68B3A09950E9
906CB10C77790380596BE927652ED190DF781E5E3E55CEA9D6C9ED2F167A31947DA279ED78B1216732BC86C3D3D4B7CA3C02
389B6DC8661644DB4BB4901841FFB021D226B079EACC22F16577E6B683E4F5BB08E75C999945121CC938EBF491F1BD0A2EE4
F72D78F71A534C0BAEBAE592B69F21A2C22336CB462B658654DF7D59F7004C0100D3A17AACC0A7C4BE9FD101A6981BFAAEF2
7672AD73EABE84D66752A578E1AE7426626D0FADD5C0B7C7E56B6A5B64CE9EECF5BBF513069CEFB3AA1A31A3187A04A70C8C
E0EF3B4B5DBD838064494A9EC8F0DA00A982D05ACF855EF586C25BBF32A79B5947C40A486239272A47729D52AC4C6D833300
304DC2EB1FF997BB821DE64316E8854117B5090983D896A712856881FC915D94D928BBDF60F725C5B4ADDD72B787ECB5E6C4
1DB6E407C1BC619CD2E8D69648B5F19922A40AF4DAE0322DBB444267196E4C419FB2A0E3454085DBC25D82CCC77B97D8DFC9
3909192CE70E8E3BEC5578B17C38074C01001F872E22FB97B96674478129E65118F4D2DA5CD36C629AFD1762422057CAB24F
DA6DB0EC496304D593977CB33DC42D07C80F82B878B269455B4A952E29477CE03C6DD576C26577565D2B26931B6225235205
458369F772CF2C9824C079165CBC9F7FCEE83B6E2C0DE46A594AC74B644F3005007C9CCA915CDF524743069ECB44180C303A
5DD2F6C73407CCFB50C23867E87A862DF21A0B7D237CA126AF5C6271826A664174F060F50D9B6EC42A9268436204E687F5AB
150223F8A304A5016AB99950E5B830FB2993F838AF828B9E0517E0699A475A49358EC9967C8C0BA600808F10D9F0F849990B
D861E18441625BC8FCAE4A6823D77AE4DB2DF484EA119CD2FD3167F50E3C0AC38BAA1647168B3182D2D98BB566270195B9A2
1E5BE0F447BA1156EFF9E88DF37E5EAA90E05970DEABF0223C47F30B12D2E85D97DB9700A600800F219071FDCBDDE25B83C0
140B240C3A298D4E2956E95D91F338345C399AEB5362CF14D316A8E44AA5AEBC0EDB758CB2C192F4E8195728D62C2ECB6797
482C4F51576F2784440B1674A60153443DFAFCBBF93D2743447D5874053ABA1702855A9EDA19763FDF0E4C01001F02F9AC27
E52ECD844A30C5C2098366429553AA555E5FC27CE5C143A70AACB8BE40FB5A72789A97F1821F7612F69DA596D72FD6449754
86DAD993B5620BDEFECE020E7C48959243817F9DC713B2250CDF52C7BBB9A7612EC102D1882F7748B690AB64600A00782F39
BD71E82D64C4ED809785349898281BCE724CB1AC1ACDD3CD478681B4CEF0E735F71622DD3181A272BACF42513672A64CCEA2
774CE189828D5624F323949A666C21C612B43ACD01FFAFE6EB6C22B920A0E2FACDCCC386498FB93C1961F73D28BC84843A98
0200DEFFDA6CF08635898621AB3BFA7AC6A15652EDDC4F15D31C10D1E02D90F1E6B78428A63E78816666410C8AE26B344BC4
ECD54D98D569EABA5DC48C42895235CFD20045F44742FE312FCB4C30A524BCFEA16BBA0DD206F0B02C1C484B4736FA84D73F
025300C0BB4895E2A0AA5BC8C580290C434CB3FFADEC63038CCE399EE765DD03FD2608F2F95CF3D8DC25DF7E88BCF52039BB
64A9EDA34DA0A84E3B32FEFD333E2A452893CF6736488D56732A6C15533CD794832A8D32B625F05A82397761468580D7E826
74F9FD4937326DC11400F02EA3ECFEC0CA9BE523D9600A83E15FEEEA59709EC41F9FCB499E56DE48E978398F0B1E6B5B641B
AD487B8E539A3B9766DFB548A2B5BFCBFAD746DCB368BE443A6FC200A98273AF36107863733C4F6E6FDCA5D82D384325A058
E634112A1C93F7426228007897065C5940C5F501461798C260A8B5AAFBF9E79E56DE9CCBAC80C725F699DDAFE66BCFEBAA46
CCCC8268759ADA3FBC94DF934A95EEFA23F6BF36E203C2F9526C7E840152051762368DCF6D8BDE7A7CD9E5B8ADEDE47A783A
0CC320A30B29F33E7A1B980200DE8950E29F943973A0D3D2B0203DE0967E30B62560D66EFD71A9435677F4BCA882DA16BD24
D86F471DC52B97BCE575BA89BB4F38FFDC800F8AE263B279100648155C89DB3EC8EC9EF519FA191D571376960CA5C3736130
A842E2B3EABB39BD71600A00788BE8667FBF32D873DC080C31BB9D522C0BFA9366B700CDBBF85A764FCCDC3323B576CB3758
922C4F534770CA656279240CDC7D38FFDA887B112B50ABE73AF950A3DF4C79CFACB7A942EEC9256D7F7413EC556A5044727E
5CEBD3E735F7C01400F016E8B908AABA057630BC6B42FF558CE438A75AB5926A66F1F7F70BCE214531C7F56B83A38A2D0748
E647C8FD23CA65657DAD56E7F6503F94909423D2CD4D174CAA82DDBDB3EA8B16CAF90F0B2F3F2ABA02CF83A11B804E93D31B
773BE704980200DE44AA148754DF8D6D0E0453188BA846DF3BB9A76631C7EC5EDE99A281D4B94C9762B0353676D44DFB488D
1DCB319D8B5CA1BB7083B96233A1B4664EA330B356056AAD2ABCC1DB2575BF4AA38407C1F0548DE63B24EF013B00C06FFC82
881C5C753BBB27064C61441E145E0EACB8C1C73833EB2BC8B72B18489E755F01F289576F33D79913F3CA25CBD6F25CBE5E17
99ED250E8CCEDE2FCF5A15940CA65F8CDDCC1051E011300AADA41AB78C83901A0200DE6484D51754E95E399A07A63022480F
38A7EE4BED0853CE24667C587425A7376ED6595B835FF1576CC1BF881368B506BA4DB144C7E5EB86C6B403A39AEE414D61BD
26BF4E8D8EBC5A7563A7A67F4443656825129D818DDF3FA2F8C99A74F0028DCD9D65CEA6A97905BD339C57D0CFE8B818FB73
13A419371EA8CAEEE59D81A5A000F0265DD4A6A795375B8935600AE3D24EAE754CB16CC4974F6D9A301D7C4B1D3267BB06A1
AC56BA722BDEC58B2D57680D768F79E52A8F10C5614799ADBDCCF2826C8539B662E7E4B103FBF93076E09ACCF191C2375259
D6A061B0B586347E7EB9E4DF3FE3EF3EE1A835B3D12448155C4B30475E7EFA7FC29630900E8C6B790A2DDF888CB2FB7D4AEC
6737AB0700962ACD84CAC0CA1B7DF4763085D149680BBE957D0CCF1D99E6F7679DAF804C536D3D48DE6F4735F01E075EC18A
1F2CB1D57BF507FAE5F5EFE858653179ECC156ECC2D65A62AEBE8AE62E83665BF60EE17DB701975D3C9BC1148D5683A2FE31
CEC034BFAF54CB032BAFA32815DABC7121F1714155EEB0201400DEA47AAC20A0E2FA3867104C610ADCCFB77B5EEB31CD24C6
4FCA5DD23AC331E5CC1C99563B71E516F3C73DC4E64E43CF30F479A1303B88ADDD87ADDBAFFF69714176C8496E8B0E67F99E
F3B28DB6D80F7BB175FBF4FFB47237B6F394ACBEDD70A245A1D41DB94A5FBF9B8823A966A10ACE476F24F046A7DB67D29770
396EAB50CE83066F5C18224A689D677A57249802005E533A941950E14611E0C11426F29A724CB6281E4C438EE6935F7E51E3
11DF1A3CD34D79D2F2C52B36E343E3052AB5A147F07F5505FBF4FD032535EAA151CDD09866785C53D9A00E4F561D7393AFB1
FA4536A02F585F927178862BE4D09872DD2EE2C51B4CA41066AA0ACE44AEA50A89D3F9F208ABEF42F44F9D940668ED468723
6546353E86711C007893C28114FF7257C69C377601E651A7B9A6DB4C67EFA4A826DFA8465FA19C3FFD9353E8EA9FF7938E5D
A30B445AC3DFDA9BAA60E51E4CFECECC4A22597BF2A61C49827593C260E56EACA052AD316049C31204DF9AE1320A67B6B584
46AB3EF6E29FD3D9BC1249B81B998751AD413B3705D0B38324C18B5A4F300500BC26B7371EA982B964E307E61D54237E65CE
DC4F554A52FB8BD05ACF19AD6774F162A370B8BDC7389B1FBD3982805481F07DCB0D4864EDE663B27593CA618D15E6E62357
AA0C5742B566C2F632ED276B12933383590D1AADE650E05FD5DA4F145437A18B69F1774EDB3FFD09A5C08222554A923B5E04
545C075300C05BAA800BAAC0949028444E2956A86A549A8F399AEC9ED8A02AF7E9EFBA5BDF265BB995F0F8054FA1308E577A
ABAF40F8814588E7DCE5EB6D7E9960B0F9282633AC84E9EC97AFD882BFE7CF9DFE9F28D4F203FE5F7DF26B1DE48673AFCC88
D39E7E002C3498529AD2F112BD00C1140000AAC0C469C09739A7EEEBFFE83842D950A66FA9234B429FCE09B5DA09DB4BB41D
B664325D65AC9B9AA62AF00B53FC78009BEA2EF8612F8619BC63E3B61FE7DF9BF083D3CE6BC495B28F3CFBFBC7BFC3C738F6
497B909083B66D3AC855585A67387A88C01400F0962A80110413C4AFCC39A0C24D20FB60D0DA44A8F02AB8401791A773B6CC
22310A815372C56AB5D1BAAFA7A90A9C1F2A7EE92BB0C60E5D93CB0DBEB1338BA331B3205EB8C19CE6FE0878EEC89988B51F
F9826E421756F7D03DFB28B46A9342A5516675473F28BC08A60080D7E4F7272155C09630C014A606D2030EBFAC4778FF02BD
21668F7BD6D1E9AC1F5128757B4F51F7DB51393C8D11EF683AAA40A5D2ED382D9B9A6DB8CA020B4F52AAD546286A488CE03B
335CDB342660208FDF416EB04FDAFD91EF74535BECA2CC684212B46A9342A19667744579175F055300C06B4A86D251400AEF
2BD3A47C38DB35EDC018FBFDE9719862AA63B2E57496C927668B566EC5E7954B6697BB6FFE55C187671BBE4A57FD7840FF9D
3556D89663B211BC46678C228B24DA4DFB48271D3EAD96753A5DC548CEADEC631FFA825821704AB1CAEC8E82F66C6A602A69
6A47D8933217300500BC469FAFA01CF215982E1E7967231A7CA4EF4B55A450CB2EC67C3AA59E46ABB33A4D3D708EC613688C
7B2FBF5105BB31D51B9D003AED0493AD8BC951FD642B9BCA66807E49C85619B1B453AB14BB07159F5405E95D918F4B1D3EF4
85B8D6A72E6936D0924D51152825291DA181B0060100DEA066AC28A0E2FA28BB1F4C619A90F8E3F649BB5B89D5EFFDD7D311
6B3E99D920AF5CB2722B21B34862C419056FAB8249BF1F5FA8CEAC546754E87F8667A85C7D15ABF6606BACF50307FBAECAC3
93555AAD314BCB176ACCF6EA67177C4A1568231B1EBFAC7BF0DE7FC57187CE46AEC77361DD812922518892DA9E3FABBE03A6
0080D7B4926A9F56DEECA6368329A650A97528A6A6D0D5A37865DFB0A27B40D1D6236FEA94D5B7C99A3A642D5DF28E3E79FF
88024754D1986A81486B80FEEDA8465FAFC28BEF5D6B70317673F3A736DD3B768DBEEB2885C1561BDDB66FF615A0E3D7DD91
76EA7F5FB96772E0C01A3B765DDED1AB3185C6E01BCAFB7E139E44FD58978556A77D5C629FD1F59E010295567537F7744483
373C56A68940C68B6D090CAF7F04A60080D7F4D2DB822ADD9B0815CBD60262A9769CA86AEE94155448A353858F5FF09C3DD9
679C19872ED02C4F51ADCF50F71C4707C5FC2865F731FDB1F7A4FE43DBCBF4736ECCEB8F388111FCC46C515195B4B55B4EA0
A830D9FCC7B70AB5DC29C5B2600005CF6FBB27F7AC634503A94ACD077BB97B8614ABB611C212847285F133E7BCD9573095D3
F8F5F183E52F9FAFB1C23C4214F9D56A3A536BF40293A8EA955BF1DE21BC8FAA028D5BFAC17A5CD9BBFF543B5E74217A13AA
BE05F46B222D99A6EA1F51A2E657DD242BAC94661649D20AC4C939E2842C116A99A9F9E28C42715EB9A4B85A5ADB22EBEC53
8C11544C8E5AA5824C4A136C093DBCDE3BB12D044C0100AF19630F0455DDAA1CC95D5677AD54E99007AF6B95C565886E3DE6
D85EA2EFB4256FB422A103797FDB8B34A40AECEFB2DC1EB2EF3EE1780571FD5EF290F74791E3C360AE4700C7ED01FBD24DE6
497BFA8173FA0C00E8AF36DB90CC8F504ED8D3EFF97363D244558DD81841398F8EB866ACD035DDE6DD8DE0FDCBDD92DB43A5
CA0F66E8BDE1CD36B3208E1294A660F63755C15A6BCCD6597ED4557EC4557ED44DBEF7926C3DD20916FABE027DA78115E6F0
48313466FC1E83F36E0C54BF1F117B4815D84599E1B843EFC6A157E2B7A38A9BDFF288C4DA31A2B2A14D867CFDB368FEF547
EC538E0CABD3548B1354F3C3942D07C89BAC493FEF276DB044071155FDA67DFADFB71E24A326BAFD1079F771CAC1F3B42BEE
4C9FE73CA419CAEAB0EE018581B7CE341D6842E2F39A7B39BD71E00800E08DE782145C753B77793C175AED048EA82AA894A2
57E2717BA40428E8FD89DEA8E7DD989E81DCD80C514DB36C04A7E40BA7E58C50B4C5E66AFA4714657552F4B7F703B976AE8C
BD27A9C8896C39403A7C85E6F9949B5E28EE195460B279900737338F24B53D97AB7EB3D1615CCBD397755E1FCA69807C1972
0A2E5E6CE44A4C4D15ACB2C03AFA34A338FD3186D7D436AB83625567EFC8371DD60B06FDB689BBB013D7E5E3442397BCB605
FBD60C5756FBC1EDAAD55AB56DD0DF30D5DBD341635B9EDEC8B09D9732683413448A0A45FAB1E9A29BDE6CDB4BB429118B34
C07E3BEA192706AAE247CFB8CF6305F159A2BC3209F2F5150D587D9B0CA953F47B6EA924B348129D2AF20FE7DFF6E55C7667
1EBB4647F200B58D6D87C8A8C5A24F90C028AE960E8D29E7A5AD2E1688BCB1C0CA1B152339E00800E035128538A4E6DE92DF
358CC5D15437C950BC7FE40A1D854EE8A57ACE951118C9AFA8C728F4F90C948854157A2107BFE2A337ADF9518AD95E120AE2
EEF871D06B7918A79CCB24843E7ABB43B2C520B3EBCD0F4B86D27D4AEC59E2F7A737CC29D1CF33447E416B12A2E0ED7C05D8
3BE31E0C96F655BA6ACB31D9D4C48395BBB1F377E552CC987E4AA9D221EF69E7F2C1258A5C29CB36F89B7783D0C94986C373
B9346A2D63046561A5D4EF25EF843D1D7970D474F79DA55EBAC9F40DE565154B3AFA1474967A16952B936BC7894AA4192293
855332E3677D970271BF1D0D89DBBC72E9F0B871D244189851763F7A7CDA48B5E00800E04D5ED4DE7F567D77A9DE1D9EAC4A
2B10231FBDC95ADF8F7AE13A332249D0376C887C79634465528ED8D993A5970716A48317684F23F928829B75E4FEA0F05258
FD438942F4A654B89169FBA185A517AE33B6DB9291E33091BA98666EC367B1CAF536BF4C4A5C618E35756AB4468D5FBD82B8
2BB7E2F9EF5BD7A99BD075521A2FC76F7BEBC32765CEFEE56EB3BE2287A7A969D68B58E4AFCD2C885B0E90CFBA30FC42F965
B5188D39FFB52910691ADA642FE305576E31771EA6ACDBA597074FC2F8752D322E5F33B174E9A5B5DDCB3B8BE30C81170080
378968F0F62B735E927A20365D74F80AFDC73D4414613D7EC16BE932CE5E816304554CBA0809929FAC495B6C48D71FB18BAA
A4B378DF5204846B89E65D94A6D79F88E4820B319BC63983EF7E5924D16ED84B44A19F58AA35911A99A62A90C974E676B2A9
FD94D758615ECF142AA3EA9A9E41C57766B8AC62F1FBC2795D4E6F9CD76F53E68EB07ACF44AEE34A59B3895EF1AAF402B1C3
3D160ADEA7FA284262041DBD72ADA18451EF90222C5168E7CADC6C434242DAE93E2BBF42623AC2727E6925D65CCF3824960B
C00B00C09BA474BCF42B7392A9B02573472CAE26AD406C7B9986F4C0912BF457A942535897A752E96A5B64F7FC39E68729C8
5FDBDF65E5964966AA0D9E56DE0CAEBA2594F35F7F72366AFD7B1796E6954951845BD988694D66A0789AAA00E1F2E8D76D13
CDCFCA0CBF15C29B6834BA9FF793CE5F7FCF208256A77D567D37BAE9C9AF5FD6AA6FE79C98C590DCE0A8023554DB8BF4F5BB
F522F64130B7B1436EAC14135AAD7E874DCF40EEDE9314541E2468515B5D7A93126BC68BEC93F6800B0080B7281BCAF42F77
25F1714BE05E24526D591D76D99DB9D19274F8322D214B24146B4DAD9083634ABF973C8B1314B33D7A6D505A2B9562D32D24
5B42BF9A60DE496978FD8973EABEB2E12CF53B8B169DEEB3912FA3304CE84D3E7D55F038F43FDB264E0E224865462EF96D5F
CE1A73C2BB8B4A902A70CF3AF2E67435A4D0CE46FD38238D8DDA43588210298175E6C4930EF49834D1C73324189271A23228
8AAF2FDB2EE2955BCCF23ACC74BA9EE64ED1602AA49D048077E9A0340456DE6823D52DF61B191A577A0470CDF612F71CA784
250A4C7C481447520544F0771EA6FC6445BAE3C769E9924F732E6270D5EDA7953784B25F16D1234597D0FAECCDC90653ECB0
2523C931CDF514865605FB3FA10A1C1F285EF7156C3E2A93298C5CF28A06FD4A847747A0902A3813B106CF1D79DD51703DE3
50764FCC344F4B65A89170DD6F47437AE09C2B232D5F2C31499F4BA6AB0323F9E8B15ABF8B78CF9FDB33A89858FC20E596DA
19FE9154D500B06C21F2C682AADC8B065217EF2DA0583BBB4462794ABF2A1079D831D3589E3F1DFA8615F703B91B2CF5E90E
50C0489FC65C328E847935616707B97EEA7F33BB5F2161C091327F1B7E2A7ED84188CF14299526B4D06C9AAA4022D56D3F35
B96DE264AA43576F85D2D8A1B340A459B915EF1BFA763A23915C601BF43795E697F275521AEDA2CC3E9254EA357285AEA84A
7ACE8D89AAE9B83D3DA348AC509AFA9240A4BAEF3EE1FCBC8FB4D39692982DE20A16F74444F4C84434F82CF9E55700301B97
AA14A3F033B6257091961F05320F82B96616C4831768E84DAB5B84EBADEB5A65176F307FD84938EFC6AC6CF8F442C267D577
F5B30B26BB0B7A682DD7330E52048437BF10932E5AB59DD06D6231DD344710826395EB0FFC221E56EDC192F3541A13F03F36
76349B73D4373FD14DE8BA284D1763374FFDAF5AAB724DB3C9ED8DFFE4A9FA87150F82796BCD89D667A81149429E7031B9D7
8A7AEC9423038919672F76EFF022EE3420F04603CADD965B02370098262FEB1F0456DC588C25EFEA579CB067204970DB9763
5263E833452CD1BE4A11EE3CACCF41F73C56C0E17DCC534CCD2EE8A5B54DE83B42A5E75E998DB0FBDEFC8283077BA72D996A
620679ABAF40FE4E8F0E97A74BCA576D3A2C9B9A51F083957EAA21996EFC4E75245F1E85F090D07A33A2D7E974A91D613E25
F6D3EF2890C9B569056224067EDCADEF8AC791548BB1AD0AC5DA8070FE543ECFC22AA946BB28131FF5D33B3CF2EC46587DF0
FE07807749EB8CF02D75142B848BA8CC1AED445195D4FC3065EB4172429648A55E0A39D97A8714F67759ABB713D0CFCEBE8F
ADA37C5CE210D5E83B359DE042CCA62642856EE2570BEC3E4EB97883696A332B7EB367A20556DBAAE9ECD774F4E98FA64E4D
5CBEEA6600FA826CADB55E12ACB1C2903CC82B371561535829FD6E03EECD21759D4EEB5D7C2DAD337C4A214CCE2888FDC819
C689AA0741DC95DB0847AFD1D1D9167B5BAD6CC00E5DA0AD352784C60B788B7034A18950E99CBA4FAE964D0000F00E75E325
01E56E83CCEEC5526014B2C5678A365A91F6DBD11ADA96D4732DC5B42FE3051BF6122D4F51B34B24EA0FBC6E09BC51FBA4DD
C3ACDE89C9EC46A99DE198EA1747C3E16B7EDC437C1AC95F88AD9AE64B1520D76F7D49B6FFCAE47155B6C74EA6DF3671B77E
22C1DAC97CC8481244A5A93426730754A6FA9F1B71B1E9A237FB0ACE47FFD4476F9FD0E728E83B1DB146FEE1A507C8871EBB
46FF6107C12B88CBE22C91D44024AADAC183B57A07E1FA23F6FCE6083500450329AEE907E0E50F00EF6594DD1F5871B37C78
71E403174BB5CFA205667B88679C196304D592AC91EA26CCE61C0DDDE3F358C187D682DDCD3D95DAF152ABD3667445FA97BB
70FF33E1B0B15DB66A1BA1A84AAA35B1F9EC8F9E2BA63A017E1920B0FCF54032E0072BFDE689ABF6E8BF70D8459E90635A35
ABD34DACDB4574F7E1BCFE4428E31D7EF68D6232D8F42CB810D31CF0DE3F54AA7489D9A2C9DCD794F402F1126BA87285CE3F
8CBF7EB7FE611C1E5F34B37CC50A617C6B7050D52D78F903C07BC19412F480BC998CC5649148B58F5FF0D05BC8E11E6BA9A6
5CFB45AAE155F67759C8BF7BF87359DCF744975D9426E7D47D342109893AA714ABD7798F13B2442BB7E2FB4C6F26D8B318A5
C505D94FB6FA3E8195BB7E3D5698639B8FC9CCCFC88E5D975F79A8F08F560D8E9A5C348D5481ED25FAA18BB45FFE7742D742
ACBE12BF1DFD4EE6E34E87FFC0C7DEB347159BA77912C643D571E13AB37F44B154DB6A5C86E8E7FD24EB33D4EE45B2689122
200457DDCAEB4B80973F007C88A846DF27E52EA62E0930BD24F87137F1A6378727D02CF94A1189B53ECF914F2138DD679168
EF899DED93F6548EE6A9B52ABB5766535DD908AFA7DC4DFB4844AAC9F5A2F40E6B4A6BD5F1D92ACF30A5C7CB5F8F7BA1CAD4
027561A57A645C231299E8FC10A40ADC7D38A8EDFDE77F75918D8F83AB6EA3DF43AAEF0654BC67D7031255E5ECC95EB119EF
F9942B146997765B2DA8946C3F44DE7594D23DB00884C100A3F35EDED9414617BCF901E043140DA6FA95399105A69BE150AE
D00584F3F592C087C3136A974FD584250856EF205CB8C11CC6BDDD439BD9FDEA61D1258546EE9A7EA0682055A5D17F0185A5
FBED680BB18DCE7206A982D078E1BF7EC24BFF335BC339751F92646285F074F80F44DED85BDF479575CE8DB16A9B7E32DE32
31516503B6C396BCEB18A567C8D485413DAED429D5EADD8CA00000BCA69FD181E29DDAF122D32C9E563B119E284492C0F501
9B27D02EB7DA49CA11AF35279C74A00FFD76E816534A2EC76D1B64764735FABEA8F510C8F49DD8D667A8762E8C8FAF6D0466
415E99E49F1B712393DA0C534A8F3EFF8E2B652231E69C6AFDD63751BC7CE40A7DFD2E6262B6685999A8AA11DB76886C759A
8A279BAEC3556BD5D9BDB177734F439306808F2051889E56DE8C6D36D15C46D925920D96C44B37984CCE327576C802C8CB9C
72A4BFD563F0A4CC25B625B06EBCD82DFD004D44449F6CB121DFF0669BE006108B9D8E3E395205D54DFA85065D9426BB28B3
8909DDD5849D8DF8F237BFD633A8406EF1276B525E9974195AA9B85ABAC19274CA91F1DEC930A6004B420FAB7F98D4FE1C9A
34007C9CA8463FE462DE5CF66E22B476C9B71F22DB5EA29B72006200B28A25EB2685C19BD96F46D9FD0EC91638EED0F9E88D
FD8C364C36B1D192E4F79237FD1D97806942A6AB902A40F20C8981F086474195EE838CEEB391EB35DA5F076BFA8695FBCEEA
2541453DB63CADA4D54D24668BD69A139C3DD9329929CE121962F6DCCF3FDF436D81260D001FA77224D7AFD4097919D37A15
D354B697683B6DC9ADDD72A8A3CC22F10F3B0997DD996FAEBF7048DAD342AC7649DB5F309080A74ACCF690C21305EF6EF007
CC118954FBCF0DB8842CFDEAC26B89E68DF872DF52C73757EE8C115407CFD350A45C5A2B5DCE86526B744151FC955BF1E8A7
0916AF1E57E2946AFD7AF70A00003E049137E65FEE9ADF9F683A4592CB7537BD391BF6929037840A9A223947B47A07E1A60F
4720FAA58736AD33FC69E5CD17B5F783AAAF37F450D6EF22A7E48A552A5005F31F05AFD88C0F4B1008E4DC2321FFA00BC927
C356D145A4A97F65B0D5A71C186BCD8945D552B01526D3BA3D64AFDA4628A8909854C1146A797A578467C105A82300F824BA
095D48F5DD909ABBA653A4C46CD18F7B883ECF793239F8B85F094B10AED842786D161EC6BE9AB0B37428D32165674675CF5A
734A6E9944A3018BCDF7033299C8C83F9C573E92E590BCB77AACD021D962EA9F8462ADD37D16D20CE9051230D41434A6DAF6
12FDE77DA451BC096537A20A89C155B7A6B389150000888CEE289F127BEE6FB7E53516FD234AF32394E3F6F445974CD500DC
0FE4AEDC8A8FCF144D79FF1B99874B87322EC69985E4E5ACDE8E2BAEC6B430AD600154C1064B92E753AE67815D5A67B84B9A
4DD97016FA5CADD6790571FFB5111F9E28042BBD496BB7DC6C2FF19C1B4369323D575DD4A61B99B63421096A0700A6E588E9
9D7EA54E55A379462F8942A173BCCF4681C6D4946FE02D2452ED3957C64FD6A4BA167DC6DD627DBA09E70745E73D333CBFDF
D255DD28D74157C102A8822D36649787E4132FBF6F27D79F78B90253EA7B06E23284DF99E17C9EF3C0446FA152EB90525AB1
051F9562127A4937A12B1C48714B3F04550300D37D8A35CAA02AF797755EC6EFB528149B591003C2F94B6333C48500475299
1FA11C3C4FC39334984A743561476AE74B8724EBBF6FACA96D52802A581855403AE3916F17B53EABFB9567FEB989C95D277E
D841B87883091339DE0B5FA8B173616CB0241128C69FDDC79130D0CBED55A31FD40B004C9FB4CE709F127B1EC6366219E82C
F5A10BB403E7698B74037A8351518F2197E4FA90AD534FDCCE3D96D79770316EFDB7DB12AA9BA4A00A164215FC644D3EEC73
29BCE19143B24523A18CC6D0589CA0EC3C4C21D3A1A17E90F65EB9D91EA29327DBE825E96774B8671D1D62F640A500C00C1E
1C7A876FA963D568BE11CB1014C537B320A5E5C3BA834F1318C95FB1059F51A0CCE94E7C5E7BCB21FEC8BA636EA58D0C0DA4
365C00556066357CD0FFDF8D84F2132F5748E498B3277BE536020C727D1CB942F7248CB76A1BA1CAD8862A19CA704DB3811A
018019A1D22883AB6E85D67A1AAB002338E5AE63944B37996CC8D93B0D6472DD497BC6B643D49A4E9A7DB2B94F66E8EEDB3B
0A1A87D53043730154C1FA832987FCD766F5443F283A975120FBD66C3C245A0096F92444AACAFC28E5F0659A11BBB03852E6
CB3AAF88066FA80E00982939BDB1DEC5D76842A251AEFEF805CF6C2FD1D496399B32BD430A330BE2CD0732BBC85D51350587
0237BEA82C85944FF38E766262ABC3890B218F5CD36DE26BB2CDF6928EDBD395309D603AB1865A97982D5AB1059F536AB4E7
BA8FDE7623D3D6D4B2B401C0E210F6FCB1C7A50E4659D23B8253EE3E46B1BFCB5A0E1B25CF236109C2B5BB4567BC6FBD0874
3A6FB33AE098F5705000232D45D4D76BD0685A8369956CB58CA0920C28856D0A7E9D9C532E67E5CB18E9327A0A468D91525E
49492FA5A450F45342087CF398FC70F2A044619468193D59464F933173E59C5205AF56296C5589FBD4184EA360E83446F02C
52027EBCB676D7957F3C7C9A7A2262C51977E25A73E2E09812DADE34A1B3D456A7A907CED18C32B6A5D569F3FB93AE67C0EA
030098252FEBBC022BAE1B7E4F84A791FA1905F915901D6E6620E774E95089DF46ABC215FFABF36FFF63E0EBCF7AD1F1CFFF
3960B195E0FB484EA7CFAFF3D72AB96AE9A852D0825C36F2F8524AA404EF2F1EBD271C7410F49CE1751EE4B69A739A36B2EB
7E60D5FC9359F935B3EC2F8CD23F338A3EA7177E4ECBFFFD2F47CE6F8FFF7C8EBE83BEC928FD13A3F44B66C5FF66557FCBAE
5BC56934E3B6ECE075ECE7F79C140E5C138DDC96E07CA5E4302436E4EC2225BF512519D428983ADD3CCFFAD36934AC92229C
C7EDA12336759B5787AFFC7F7277EE713DF8D5A6352511190A6878D347ADD6A51788576E23E4961941D4D14594C0CA1BB023
1200CC9AEAB142EFE26B7DF436435E94C1D61C384F3BEBC260B060507C66706BCA9B7F324362001D7D7FFBACF76F9FF5A0E3
EB5F8EE13347313279B60A40AE915354A24E39BB04A3C54B084F4523B704BD76BC764BE4A95935DF31CBBED2FB7AE4D6F37E
AFF7E6655F32ABBE61D7FE9B5DBF8ED3F4339207BC0E6B7E972DBFE704FA2B41FF45A41C84838EE8100DDF108DB8A3B34D1E
37A73E4487A0FF92A0EF1CF2FEFCAEC34806705B77719AB7701A7E64D7AE6055FD1D5D8E51FCC594A8A01720F1F025FA90D3
B08EDBB69BDF734A38EC26C13FC128D172563E122D1A19612E1D0B2A9108F7C0A36FEDDF7AFE5B6FC63E64DEBF7D3665E4D8
AFCD700929D0F066048DA9DE7D8C72F42ADDF0976E23D53AA7EE678AA9500B00303B4472C19372D7E8267F435E34315B64B6
9798902502FBCF08C9F858FF96B5FDAFC5C0BBC77F7F366A7F11F9B8E945C752B57454C1AB4161B88410281C7246DE99D3BC
19C5EC8C923FD072F5E13CA3E48F2CE4FAEBD7725B76F23B0F21272E1A72158F3F90108331CA2B193353DFEDCF6F508ABAD4
92618D8CAC55B2B56A1112183A9D7A3A25D069155AB558ABE468E454541895B847296852702A64CC1C8C1A2B253D178F7B23
518134062A9B5E363498E98B57FA25D209A884489C302BBF41A285DF794038704D827B8C51E3E5DC72956440AB9AEE7E3D3A
B57AECA6CB2FCAEAB7F644761EF8E6B3FE157FA067A441F39B3E0A852E2249B06A3BA1A5CBA0D35E642A2CB9FDC5FDC9FC12
0000CC9AD4CE70EFE26B6C09C33097D36874576F33ADCF50074661AC76668C5C39D7F3CD07F4C0EBE3EBCF58C545BA0F8CE9
6AE414E47665F43409DE5F3860CF6BB764D7AD464E167958E46751788EE2746ECB0E7EF73114CB8BC71FEABBEE99D90A7E9D
5ADF6FCF40EEDE28378E34867E1A837E2CA349CECA4782448C7B8CF409BFE714B76D17BB7E0DB3E2AFFA7E8CDC491953F33D
B7D55CD077418CF3C668090A5EAD1AC32105F2A193238FFF4993F6FEF0BFE5341AB4C0E9338A576EB42239DD3768EE8231F6
C0DDDCD3F5B852B03F00CC05026FD4A7C43EB7CF40730E5BBBE5DB0F911F3EE34132C399C541346ADF9AAFDF8D67DF7561C3
E74F2A05BF7617A8654405AF06A3C68846EFF1BB8EA0701B797FBD0C284481F6D728D0E6A1407BD05EAC0FB4E3141C1468F7
6B558B23A9AF562D524B47F477474F96E00344432EFA1B6CFA9955F50F46D19448F813BBEE075E878D68F8065238724E29FA
FE9BFD183AADAE6FE38A4F9874D2AA04DF47D008A78F50AC75F7E16CDA47E20B0DB451874EA72B1E4C7749B3D142622F0098
33E1F58FFCCB5D551A4304EFCFA2053F59934A6A609EE1CC60E464F5FEFB0FD351057D6BBE918CB7ABC42D182D4E347A87D7
79905DB79A5EFC85DE4B96FE995DB78AD76EA5EF6FC73F91D1D39582668D7CE90CC26A956CA5A853C6CC91129F09879CD1BD
731AD633CBBFA2E5E95510ABE65FBC764BE1B01B468E50702B3572B2646CECD39260F2E8DBF03D34C299F8E889EA266CF50E
427CA681060A19224A40B95B7C6B10181F00E64E1BA9F641E1E5467CF9425F48A9D49D73651CBD4A378564E98B0B727868EF
BFFED774FC57EF3FFE2F46E13E6EB3997E7A40DEEF19655F721AD6F1BB6C45C337A59448E40DD5186199184D23A729F90D18
2D5E3C769FDF7392D3F413B3E2AF53B325D9B52B914D48CF8F4C5315F47CF37FE8208FE44CA031D57B4F524E3A186868B299
50E99462C595B2C0F2003077D45A7570D5ED90EABB0B7DA1F65EF90E5BB2EF0B1E74F2CD144A5478EFBF7EFFE9BE02747CFB
1939F6FF63D7AFE0751E108DDCC4A8D10A7EBD46C158E606D4AAF84A611B464F168F79F0BB8F21A5C428F9D388F3FF98B62A
F84CAB80258A334022D5EA3395599098EC055F6A2456085F35F9F994D883D90160BEA81E2B785078A987DAB2A057894C16FE
644D827C86B380D7D4D8BBE2CFD31A41D8FC577EBB8F8257AE9153C06EEF57084AB6825F27258753134F4C4B68FDED3324C9
C06E33A5AE45B67A3B212967C107117A69ADAEE9078658B01D1200CC1B329534A0C22DBC7E61A754B978B1F79D85D507B342
373160BEFDD3FEEBBF3F23F83DD242503B3D947C5EEFB7FFEF7454C1D0315B30D74C21D3D4E6472897DD990BFCEEC292DA5F
DCCB3B03060780F9A57020D9ABF0E2C2E50FA733D507CED39C3DD9228916AC3D7351C02204DAF67EFB7F7E2CB6D54F35FC9B
786404CC357D461DAFF67EFDE9E103465E08D86AA6F0851AD707ECEDB664ED423EF1C3CC9E9B99875B4935607000985F8432
9E5FA95354A3DF029DBFBA49B6C586FC3201F69E9B015AB550C1AD94E07C799D0798155F0F1EF94F32C3F7F512F47EFF053D
3D5507DB28CE048C48E85AFFCDC7B4D67F7FD66FF919AFCD423CEA2167E569147430DA3451A9744939E21F76123AFB172AD3
8552A3CCEC8EBE997904AC0D000B416E5FC2FDFCF30BD45D10992CDC6C432EAC843589D3422D1DC5A87182FE4BECBA95FA34
83C55F70DB7EE6B5BBE13C1DFBD6FDE3759663FD3199A7B77F87192D314EA7857E9819432BCEEFDFBA6ECA8CBFCA83FFFC3E
6C779059B29B59FE175AEEEF59D5FFE0F71C97925EA8449D60B7E9D035A058B393809EFD053AFF3867E856F6F18A915C3035
002C040219EF71894354A3EF429CFC9E3F67CF714A471F8C797F0C9D56A1E0D58A718F79ED568CD23F213DC0AAF917BFE794
94F45C299CDCAE4237C12A291EBFE5367878DF80F9A6018BAD43276DF1DE5E822EF053B387D7D438EA70A57FD3AADE7FFCDF
7A79F0B7CF7ABFFF1CD916FFF8A1824D53CB06316AB4A0FF0ABB7E1DBDE0737AD1E7DC96ADA2D1BB7276A1562D04EB7D0412
55B5EB28C5C16341160CAA34CAAC9E98EB19B6B0A40900168EFCFE24CFFCF323ACBE797E7E55BAF36E8C538E0C0A1DFAB7DF
8F56C991337384432EECFAB5937B10FC81DBB25D347A47CE2ED2AADE33EC82914982CE0E616F8F82056BB4E7C5FE4A6E6D0D
39328CE0EF4B0CF2A725C523DBFE56B0C927059B0FAFC35A9F1F320F09B67F0AFA2F62B4048D8C08067C2F6CAEE6C275A6E5
A905C99435CE1974CF3E5A31920376068085432C17FA9539473478CFEF690914D5BEB3D4DBBE1C991C74FDDBA865048C122D
E83DCBAAFA46EF6B2ABFE6774F76538B61A59549D61736FECBE04EEDBF517D2185C0EF3C24213E5389FBC0386F21C1B4BEA1
FC8D562481709E73404D761444DFCC3C0C4606808506696F8FBCB35DD4A6793C676BB77CCB0172D02B3E98F737FE453284BC
3FAFF320A3F4CF938305FF16F65F91D1D3B44A0E18C7C4D1693039BB48347C83D360F6CBB042ABB918F758296805E3FCDAC2
35BACC22F11A734247DF3C4F381C630FDCCC3C523356084606808546A9510455B98754DF516BE72D2F71418564D33E12EC9E
FC6BA42319901002796D7BE8C55F209F823C8B68F8A69C5DACD3C8C0388B0B05BF5E3CFE90DBBA9351F4B9BE2A9B7E168FDD
57F21BC132533477CA566CC1E796CE67EE32B94A96DC1E7A3BFB049817000C431BA9F65EEE9979D4E1B1E9A26D87C8F915B0
0061422D1D961082B9ADBBE8859FEB5716B4EE443E45C16FD0CF230416AFCC13F7A06AD52F202DFB0B2DFFF79CC60DE2312F
E837400C8F2BD7EF223E8F99CF25C97DF47697349B4E72039817000CC6CBBA074FCA5C248AF989EE9F45F3CD0F53AA9AB065
AD07B031293184D766412F427AE00FBCB6BD1262300C462FA92A9611306A0CBFFB38B3EC2B3AD2064D9B90E4538A3A96B34D
F024D5CEC394DBBEF33628865E4A910D8F1F165D81F60600866494DD7F3FDF2EA727765ECEF6E8196FF7314A4B977C791A53
2DC34B49A1BC0E6B7AD1178C923FF0DA2D913C504BC7A0992D49340A06468DE5F79CD06B8382CFB9CD5BC5381F95A87B795A
834C53DBD8D1E631EF711BA9D629C50AC71D869606000626B93DD4ABF022893F3EF753B9FB70F69EA2F60D2FBB64051A055D
4A89E475DA20313039216D975E0F60E3D0BA9641D533315A1CBFFB28A3F44FF4C2CFB92D3B24783FB574D9E5A3A6B3D427EC
E947AECE4F4E481EC60EA8B8FEACFA0E343000303C7C8CE353621FD5F878EEA772F664ED3B4B1D1E5F46FB22E934988C9989
024646E99FF54EA1798B04EFAF960C41BB5A66DA80862159D8B19F51FC056A06BC762BF4BFCB6A8D0987A7B9ECCEDC73621E
76F0D4E9B41523392E69FB913680A6050046A16A2CFF4ECEA9B9EF3CE2708F657B91364E542D13BB297835C2414756E5DFE8
F99FB31BD68BC7BC54A22E684ECB16358693129F715BCDE9059F33CBBE14F49C9631B3753ACD72B8779E40E3E8C1DA614B9E
FBA948FCF13B3927737BE3A1450180D1221DADE6798DC7937217917C4EA906AEDE661DBD4A275096BE2A504B87C5E38FD8F5
6BE8F9FAB4F9C2216705BF1E1A1230A15FA7D02B1EF7E234ACA34DB58D6157A5A079C9DF355FA8719BDC39718EE791EB774C
7E7E27F7940E56EB008051C171873CF2CEA67545CCE524576E3197BC2AD0AAF8183586D7BE57BFE4B0F44B7ECF29392B47A7
857D1F80DFA0E0550B079D98FA7E24FD220509214023272FE1FB158AB4B77D39DB0ECEF51E3BC8F54E29567DB43668420060
74727BE3EFE69EEA67CC7E8195FD5D96ED251A8EB464558182572BECBFA49F73AE9F57B6534A0E87FC84C087D069E532461A
BFEBB07EB241F117FC2E5B19336BC9AA02B1F69E3F778EAA8023657817DB3FAFF580C60300A6805C2D7B5A79E369E54DA972
9609CA5CBCD8FBEDA823B82538DB50232349F04FD8F56BF5430635DF8BC73C54E27E6833C0A75B8E9C222506731A7F9A1A50
100DBB2FC99D2F0422ADBB0F672E2308537B23BAA61F102B04D06C00C044186276DFC93999DD1333BB3F47AF05CBD3D4FE91
A5D69D2E6715A0400F457C8C923FF17B4EC8D94513CB630A19305F28054DFA0185F2BFEA331BB49A63D4E82536EAC413689C
3DD9E64766BF06A18FDEEE9C6A5D8F2B85D602002645564FF49DDC53038CCE59FCEDDD279C3DC7295D034BE775A7C670E231
2F56F577FAD4F78D1BA4C4108D1276340666834E8361B4445EBB957E8542F97F09071D94C2A5930E91C3D75CBBC3DA3BDBCD
940532EE937297800A37682700607271B10A0BAA720F287713CA7833FD5BEF10DEAEA394E6CE25B2F58F9C95C7EBB0A1177D
CE2CFB4A387045296881E601CC55674A86C5A31EACAABFD3F37FCF6DD92625472E8D4E03264763E7C23C7491361B9B685539
BD712E69FBB95290DC00608AE0B9C377734F27B7BF98E91FBE8815981FA154362CFA7D10340A8618F7F8972E82E69FA59448
9D1A368204E6ABD7402B6366F13A0FD00B3F67967F251C725C02DB645019EAC397E9E75C6793F1B887D6E2946259375E0C4D
03004C96D2A1AC5BD9C75A88D533FAABA41CFD9E8939F3BA9BAAE151F2EBF93DA7F4B308CAFE221CB8AA5ADEBBDE000BD569
A01F9CF264557D3339D36017464B5AD4B743A0A82C4E506FCD7C7724B684FEA0F0D2F31A58770000268D56A78968F0F62ABC
481512A7FF572535D29FF793E332166B58ADD360524A24A7E147F4A26637FC28258769A18B0058D04E034606B76D2F6A6FAC
CA6F44A3773572EA22BD9551BC72933529247A6669D0146A59625BC8CDCC231285189A030098386C09E361D1E5E735F7642A
E934FFA4B35FBEDD961C18C95F8CF7AB920C08879C99655F318ABFE0F79C56F0615777C0300D6F50387C439F04A3E8737ED7
1105B76A31DE45F78062F5764246D1CC9C7B03AECC21794F2FAD159A01002C0A7A68AD37326CD3BB22A7F97D2A537DF002ED
86375B8A6917D79DCA9839BCB6DDF4C2CF59D5DFFEFFECBD775494D7DAFFED5FCF5ABFF5BCBF77FD9EF5BEEB5DE7099AE49C
9C9C243E273927D134138C3156EC5DACD8C506221650A4290AA280D805449AF4DE86A1F70E436FD37B5DD3E7BEE72EBC7BC0
A849EC4198B9677FD62C97E230B3F7BDDBF7DAFBDAD7A51B0986170D2093BB676032F06264758E82BCF764B53FE83977C08F
6CA9FCE428AD520F54416DCB1B781AB3E4035E19CE8F5A6FC30E0081D810795D89A7D337D7B34A5FF3FD87BDC53BDC841CBE
CD843724CC2A1D2B7CDCB150DEE4641467C15804902901B138B4B88CBB206A06BC71446433A2DA44DE495039AE647305AF3B
F0D526E555FAA9CBC56E2409F31D40EC11A3D9A032CA657A9158C313A8D92CC5C088BC7F44D6C794F7331503E0C55331253A
A1C22033A07AAB1A261881DDABBDE497B797A5187C9DF75FBCAE58E5C26F61986CA25D305DAFBAE79888F6A1A8F87D75F711
B3AE07F655C814822342EDD00571E9C796D3848E1DA8CA36B6D62DA9912EBE416A24144752DAEE9E4ADB04263DD8E810AA62
C28C5C15B393DF5831989FD5F930BE29E26ECDC5B0D233400C5F2A3A164C73BF42F3B852E2115A7222B4C4F32AFDE4EF5EA1
74CF2B2527828BDDC19B838A8E5E2C3C1C423B7EA3D2F7417D687A7B54497F6623BB6240D2352509C79506D9E522B7F0322F
8DE9D5914813B3346072C8A3EBADBFC910598925C951BE83A4E20B1DF32A6156C26E0C997A48DCC08F97D5CD15E4BD276F58
6014A5597F910562CCF99060A7DB6B2DF1E428593D5C743C79550BB71AB636844A00EBBE915D0E04C09DEAC0E0B1A51FACFB
D74A4F4594794796FBDCACF4033F07C220AA36182804B0B2E7762514F42417F5A68125BE7C30AF72A800BC2A86F2818A281F
CCA7F76715F6A480F7A4B5DF4F6CBE11537FE55E4DD0DD9A0B372B7DAF979F052BF2B5D2D3403C00A900BEE82ADD33BAEE4A
3623AE8E45672B86CCF83BDFAE1F92F69CCFD915DF741D275EB1BB5ED36C5CB8897B3BDEBA2399938481FF505AFDADE5D4A0
61213C3580589D6455D628DBB73C7674615E2571ABDE7C1B64991D57712EDF7CADB8670312C6C9B48D191D31B09521B68E01
D575F21BB33A63232BCE5D2870BD587804D8FEE165DEB7AAFCA3EB42D2DAEED3FB339B38957DE24EBE9AAD322A50FCCF862C
2348426B520B35DC21594F3BBFA17AA438AF3B29A139F25E6DD0EDAA80C80A9FB0D2D3A00C414547434B3C1F368401693128
E9FAF3DFFB5C80043A97BD13A89A575A0D1B0F083C03255A9D95BA4B1166857638485CF6896593B673A759D30EFB36C40AC1
8D6C75FF6951F10722DA87EA1E37CCC8B2DAA256D41B672F64E597BE7A8750AC15F8E7ED0373D71496566352725523DDC2D6
3A261DD868C0164B6DBB17D7180EE6D51B95E7810936BE7F1B4273BF587818CCAE21C0DCA39F0A2B3D73A3E2FC9DEA0BB10D
D7925BEFE430E2C07CDBC8AEE811B5F1542C03AA833DD66E20FBC59DC02A07EBAF7FDEFEE062CB46FAADAA00D0858A7A535B
B935A07721D814C8783DAAE5AB590C4153D9404E4ADBDDFBB5976E56FA4594795B4442E111505AD0CF814499D813072009CE
65ED78A5E7E1C9B1CC893D83D618C715330CA97B8E8A8ADF1795FC4D3370C676AF8743EC62F6C1F57AF6754BA4A3020745DB
2654556F8DE69291088F523AAE640B25D8AB9663555899976FEE5E2DA29E0C5945E260BD6EE7D7D307B2525AEF024BEA72B1
1B58E5838BDDAF949CB8567A0AD874409F002500264FF0BF60C5BF571304B401984EC12BAAF672545DF0DD9A8BE0E7E04F60
FD8D6DFF5E00EF04EF07BF35B68B7B667C0BD7F2C98547C1673E6C0C2BE8496EE1560B353CD881A90430D2C1821BDF140156
D84B456E4028826E0016DF3A169DAB1CC170CC0ACBAC362907A5DD55C345400F803E1C59710E74D78B8547424B4E8011D1C6
AB0342E2CF3F96A4969B7EB97BFBC41D2F795B4C8A7AC1466E61B9D5B916A0AA4630B58EEDCA7EA167DF207128EF21368051
942EABFB4998F79EACE117A338DBDA8A2791E17B4E88361D7C45060414438095ED99BA9EAB62BEAB39D0A8E816B602F31F58
6D63B39FC53BCB72AA5BEE7DA3D2172CEED175218F5A6EE530124A07721A58E55DC29611591F4FC594EA846AA3D284198917
9F2482D90F27301447B426954C270646D990ACA753D008AC24F08D999D0F1E3686838917AC14D7CBCF5D2B3D3DA6138E00E5
90DE1E05DE034582EDD22968025D1774A7CB456EC0F406BDAB6AB890A518C408B36D5544A4E1012590DB9500142FE899401B
5C2C3C145EE655D0F388AB1A79FB09CA6C00830B888D977C4873A769E9165E44B47539EF99A485B28605E3A90F8DA20CE848
00B1215065ADA2759D30CF415A35CBC08BB2AAB20DB2D0796B3821B75FE654801338988BDC1297B5F32678BB034CCEE58379
60D20E2E7607F31210036096034B7354DDE5E4D63BF4FE2C60B6037349A213A2383A19D69951018A04E6DEE2DEB484A6C83B
D58191153E40245C2A3A0AD61450CE5A66091015B04B5B3F32BD38B3331634DC858243600D4D68BA5EC7A44B75226AD44EAC
E53773AA323A62EE565F00CA19689E609A3B18323DA236E2CD83A528F452CB161CCD033CB4E72B0713B9EBB868AFA7E8953B
8A9386419028ADFEDAE25BD8BC0A85410B213608A61F50751DB4443328FB5437126225B2162746734BF4DF2E65D5BD347E11
580A8124A0F5A54F8C0E91F614F6A480A51F4CD72134F7B0D2D3E0EF31F5577218F10DACB221698FCA28B79256C3495CA8E1
760A1A8148886DB80A14C2F8166E08ED784AEB9D5E71070CD7608580951118BFE773765D2A3A06E465CD088DC2D93C81621F
96F595F467C6D48746947B8F9FB2C53745F48ADECCE38EAD180AC83F78B3D2576D7AFE8640F02DC5820D9CCA062B48A94CE2
7ACE6D4BE25A8B6FE12EB3AE0FF67988AD8E5F44AC19382B2A9E31E6157396C4A63E77804A439CBA285BB0916B36932F9963
3D53D7A7B4DDFD335F04A6E58AA1BCBB351703F30F06151E01D63730C31FB5DC2A1FCC036BEB5BA47D9F12944679A7A0A9A0
FB51745D08A84260BEEB8502D7B8379F8121EF88264E6558E919DF9CDD607D0476F490CCBE22D888343C20E0E39BAE879779
016D00441118622CF9E06BFEFA80B4CB27DBE57EED65A3F93979934B6B0DF3D7736EC64EF1FD4412D369472E894B3F1615CD
50F7BA43DF4288AD0344AE8E196AC9D6513C43DD7304E884A92D0F936B9EBF9E7BDCEF85961453DEEF95B925AA2EF82D3F5F
3E90D39500ECEBC0828357E99E372BFD129AAE0325302CEB4530A34D37A540CDA91C2A886DB80AEA1558E00A6CB46C461C85
6D526BD703ECCA2B251EE7B27644569CAF18CAB79EEDA62941AA13D68CD0805E051A2920DF32F4687DE96AE3ABBD02DAF9F5
DE99DBE21AC351ECF7D70DB47A62A79BD0C55DC4134EC92182C56C21CC4A8B6145FB101856DA215FC2AC803D1F42096580E9
79F72D17132CDB5F2E53786311C3C9CC22DD378B59A5358617991EBEB97B81DDF1A617B580D1018C94A0C2A3172D37AACEDC
AFBD94DF9DD4236AD52114CC610A9E52D570E18386AB21B4E381F9076F570576C2745193C8A0B42BACECB44FCEAE9B95BE35
4C9A1E5E327D06A1864BEFCF020310744E200FC05F5E7ED7606CBFA5C22B630B18BF7FF4C6BC1EA3FA690DA7A862F26F2258
24018E08D5BDC745C5EF8B4B3FD631AF010B0BB62F844A188529D2AAAFC66E2C3A9B75BD535206B912773D2376DAC6C39EA7
FDA57A5150E191E0E2E3AF1313F5F1CE806220B5EDDEA5A263170B0F8797793F6CB8563E98CB520C1224610F6D3A22EFCB66
C445969F0BCC77059370D940CE1F0D2EC804A234C81F36849DCFD97DB3CAAF9E55FAFAD981ED0D9CC0FAC49D991D0FC63AE7
C12B251E25FD59CF3D2618A796597226C339A5ED2E46FC666AE8E841966EE1F987C9117432DD69C6248191ADEA7215154E07
F6949E7B8724E0C882501093345F563B4798EFA068598D6A5A27BF003D83E88FCBD98111F2E7CDB7B2609A7B6081EB8B7C92
9F45635216F6A45E293901DE1F5EE615DB70AD6AA8101829F6D9AC4044D532E951B5C17EB97BC10329EA4D9D929038D406E8
CCA2DE34BFDC7DD7E8A74B07B2E1FEC06BEB2859D570614C5DC8E56237D039935BEFBCE85246D570D1E974E7472DB7CDBFBD
FB732240EAB495D7D635695DDA220930C388AA73B7B0D04152F96F83207E9434C3A6845015445E266B983F9EEE7392532999
10F27E92EADB25ACAE3EE48FEBDA55FA49BFBC7DAFBCA4DFC16FB8531D08963F2009A2EA82CB0673F96A366CD671EBAC9D57
1F5D17E29BBBE742C1A1B2C11CD23E364C26018E7238ACF4F4F99C5D605113C138126FA3A9F05E71FBA3D6DBA1744FFFBCFD
51B597D98AA13FBEAD72B8F054DAA6C4E61BCF5E0ACEA5EBE6AEE6DC8A9B1C9FC33149A01F5476EC101658AE755BD2CAC0A0
0410AA83AAEAE58D4B2DC2A07131AAAC9BB4EF1548B0953B795B8F0AFF2009D4D74A4F9DCFD9CD7B71B422AD4995D7951854
7804AC77372BFDB2190F07245D041CADCF9B7E19C2E6070D577D73F68019B847D40A9FC99FA4A43FC33B736B78995727BF11
3E8D3F0918E3B95D096165677CB277DDAAF21F92F6FC71C7000883B8A6F027FB5D7A03B1DB43B4F1806088F5AEE3878C4B82
0165C7368B24A8FEC612020E4A6B889D0803758BBCC94998E7206F58802A6B26C392C5C9DC12DDEC45ACB47CEDEF7609C625
C1736D8771332DB6E11A58E3AE949C886DB8DAC02A7B7DAF03BB852489365E1D98757DB25DA2EA8261FAE9B7436D52DEADBE
E095B125BD3D4A6584CEE713864C2F2EEC49B9463FE59DB92DB2E2DC90B4FBD9FFAD1BF3310083DD647E7C6928314BE3B892
1D9BFA4EC39E3F3E3850766CB748829AEF4D92FCF11F4220F6220C346DF2E69563394017A0AA776E04C914F8B6A3C2153B78
CF7A0D819936B4C4D337673758FAFFF82B1DFC86C80A1FDFDC3D6337C11F0CC97A61ABBD117A5447EFCF0AA1B9FBE5EDA3F5
65C007F24630E5FD170A0E05E6BB3671AAE0D37817C8F592A2DE5460149CCDDC76B3D28F291F78F25F4D9CCAB359DBEF5407
8E9B004A35BEFD98D0F99080C97D47E7FBE3EE855C8B2FC1B8249016414900B143CCDACEC7C2A0718959D3F1EEBE8820464B
AA0D5F2F66DD4F523F3B275C2E76F3CFDBFFAB24783A061BD8E5A12527FC72F7DEAEF22FE9CF946805B0B1DE1A919697DC7A
FBACE55ABD8FDD3A64BE292D9CEA33195BC013E3BDB31C1C9071140669715F5A7899D7F8BE9658FB384650A7A01158044033
487496E19F9069D92E88497E57DB053822B244822DB41C1C98A405501240EC5718683A2C470996C8DECBDFDD7545850ADFE9
265CBA85ABD33F3EA41369F8E32107F9AADF044FA81E29BE54740C48857B354175AC524A461B98125AB935574A3C80D0AA67
D1E1D378A9D148960EE49C4CDD98D01CA985DD6F52B541FAD861E2AE472DB7C607FE88BC2FA8E8E8C5C2C35CF520621C7571
17ADDDC31F189978EF02C22C57F7BA8D49825946490EF8016C11883D836A5AE58D4BC6AE2BAEC50C2313FEF9383E9A5FA6FF
7AD1D38D02966210D805A174CF277601A09E5576B9C8B275105517D2CCA98277EB261CA94E14DF14E195E19CD27AC73C2989
9F6C0E9CC4F3BA933C53D7E77625405FD6C947A21564753EBC54743420FF00D066E36AC1B28D90B38BA96A2B2CC57E5ACD0D
8F5262F8441AF28459ADEE3B292A9A2EA9FCB7519C01250104621106EA2659FDFCB100479B71E3046F324BE4F8A60382953B
79A6B175BE5BD87A26DDF94685AFD2201B7F433BAF2E8476DC37674F546D701BAFD6E6B2CDDA10E0D9D2FB33BD32B746569C
931B60A8E4DF3F9CACCED813A9EB4B07B2E1D39842D88AA1C4E61BE7737601C381A518D0A3DA98FA2B67B3B794F7D33CFC94
8B370B1ADA262C743989E9C632C5BC2F299F691024C24B8810C85361A0AC91D67C0F8481B263DB04E64A40CDE4C334CDEC85
ACCC424B58B33A66A967EA0630C6C7BD8BC1F0BF5E7ECE3B73DBDD9A8BCD9C2A141AB093423BBFDE3777CFC5C2C36CE5107C
1A8F7709082CAF3BF144EABAAAE142F834A61C9224DA79F5B7AAFC81824D6AB925D4707318F17E055B2EA4DC5DB08979C25F
A5544FCCF2AD1D0E12D13E14977EACE7DE2761A82208E4B720F2124B48E4420715632F894FCC06FE08C7BC702377F7712941
8C8ECFBA191D31A363EEF13175216346AB4FCD08CDF4E250A89077014B3178A5C40368831E987B712CC8037D20EB44CADA8A
A17CF834AC071DA229E9CBB85878E462C161A0D6CA07F32E14EDDBE8EBEDB89691923D012BB89E7B4F5CFA0F11EDAF3A5618
49C0234B08E43998C439E2B2CF44C53334FD67FEFCA7198CC4B910D91C27416D9B2AA135E454DAE6CAA182D1B1B030E7B276
04D3DC8B7AD3EC3CD3DC1422D38B6F56FAFA64BBC0CC4AADBC5A20098AFBD260AFB042B8CA91B8C6703063DCAB092A1FCA09
A59D5EE4B677DDD1CAEEDE3FE55D6014A54A2ABE10154DD70EFA91188C5F0D81BC08D2C08F1397FC4D5CF2916E24E44F7D10
395A526D0492E05C78DFB53277BFDCBD9DFCA62169CFA52237BFBC7D29AD7760A4E2294763528199F65CD64E861D0B03A67C
E06CD68ED4B67BB03F582D1861AE67958EE73D49EBB873A5D0DFE9D466CF9B29B2B78D2C6592D1A435DF090B1CD43D6E040A
0D1308E4E5CBB959C78A1015CD90947DAAE7DE7FEB8FE109F1F57B258B77E71F4FD87AA3D29B21684E6EBD733A7DF3DD9AA0
6E210CC36B2DE810CD9DEA0B3ED92EBDE20E3BACBE4C270A2E76BB51711EE68CB07EC45ABEC50B3177D715DA498F68BF357E
CE17332F2BF56F14E9D4B2BD60D6B4CB1B175B7CA83A5D7013CC6A0181BCC6C8C10D9A41DFF17C6163013DDE5C8A9B467D43
15DFAC0F77895C97DC1E91D79D14907F3030FF20BD3F0BBA1058E18E4164B94F40DE8117059DA62A0654F7A03EF47CCE6E18
437B62411092CD37B7304C25D5FA943CEDFD24F5F518E5E59B8AC008B9DF3519785D889087DC5644442B6FC5A9625335D934
5D4DB3B17F1895CA71E2A5EA0C27B03A263D947EDC3DCE65B5D7C9155E3BCF651C1A90BE81A0C50C4C65BBB3E5BE55CB1AB3
BE1F361604F29A106685BAC74D98EF20AB9B8BAADF6C7B99C447538B79F35DDDB6866EC9EEBE77A3C2F754DAA6A8DACBCF46
3485581572BD248476FC4AC909F0177B91BE24593194EF96B462F0B741F8216F232C75446B97292D5F1B7C4BE1765EB2FD98
70950BEF97F59C9FD670E6AFE72C76E62E71E62D75E6396DE3ADDCC95BBB87BFD285BF6C3B6FD936DED22DBCC59BB9F3D670
C0CB692B6FC37E818BBBE8E405E9CD87AAA20A7DDF106A429EE33CC057B1135BAFB95CDFFCD3FE034E1E478F3D742EE88D33
FF1A67C0C0E36AFBFB4C62D1F3A635B9AAFB109004B23A4754550F1B0E027923701357D1BA612CE7F232CCC87AFD5F2C6CAB
DE78D17953E0E1E8FA6060885D28380CA65F338EBC626152E20323687D9B11CC061985BAC42CCDC334CDDD04F5DD4435B035
C0DF1FE568734A74A5358656866988854ED4D524C8381CE5B05FEE5E20DE3002B387FAB2148367329CB33A6361D3BF1D284A
B675996253D5A72E489D0F09166CE4020D0056F98DFB0587BCC4FE61F25B712A306673E93A7AB5A1B2C158DB626C6C3781C1
DBD183B4749AEA5B8DD54DC6F27A03AD529F59A48B495187DE557A5F96EDF314ADD9CD5FB869FCD3B85B8F0AC10FC1E7F40C
A2E43302C16836543233F6DD72F9F9D0F6B5DEE7DC930FDDCF3FD371D57F70EF8E9E558BBA9DE6F5AC593A7070B720290137
189E0841ED90BF253441C517264B00430804F2C698B59DB2FA79C2420765C70E127F75E4108C30C737DED8766DC31237AFD3
49EEE77377C4365C7B6EE6230EDF5CD96088CFD05C8C9403E362D77111984C80B1F0F33ACEDC5596D7BCB5162B034C0E8B36
71176EE4028362C106AEE57F57B37FD9C00186C666570198404E07C96E3C50E596E8BAFA9127E194216F472BAFD63B6B7B6E
5702E56BAA1B0B8CE39FB71F36FA9B028C77B09A5FB9A304463D50028E2BD9C0E43FEC2DBE1EA3CA2AD6D5B5580E02C452CC
607CE30B02386ED9731088303096CB6A0DC939DA2B77144060ACDEC5075301D01B7B3C44576E2BCBEB0C7AC3E3913E286F3B
72EFD88A737B0E1E5DF360DE5FEB67FE07E3D3699DCFBC18B33E1838BC4F376C391AD3F3A2C5A59F884AFEAAE7DE1B2531D8
9410C8DB8128CA25E5FF0304B666C0FBC90F0D5C2E2FE6FE90C7D1FEBDDB070EEE62069C1797140FCABA2F951EDF7471E78F
3BCF6EBEB43D847EA472B890187D3AFA0699685691EE62A4629FA765A4035B60EE6A8ED336DEA68302572FF1F95079D87D65
6C9A3A354F9B51A8CBA3EB0ACBF5F46A3DBDC65052ADA755E9F34B2D3605B01AEE25AA2FDF547806485CDC852B77F2C08C01
A626201BB61C169EBA287D90AA6EEA30198C5021BC0DD98C876377151BA95DCD464EA55BD28A11393C567E03C0621D11ADDC
E926048A1D8C5C30FAAEDD53E6D2F56DDD26B1EC9D6CDC61382992621D3D0818FBE0BB80EDF0CB986900467AC86D45439B25
C2809C545C083C10FFF98C9E99FFD10594C067BF7D8D6983BE9D3B150DE1F2A61F8081A319F081F71021903F076110C4898A
3F10977E6CE0DD05FFE6C546F72C9BCFF8EA7DCB88FB644C90CFFC5F0CC72FF2B62ED9E4B6F1A70367569C5D17DB7481AFB5
4CB924315ADF6ABA11AB3A704A0C6C0AC7559CA5CE3C30A8BD2ECBEE26AAC16007437B808982B10FF43FF1DAE6056A26152A
9CC33777F422A5C0ACC8D586DE55028365EC2BD8F3D771B61F135EB82E078A021E34BC112886DCA83C1F4AF77C128F9A7AC8
F4E22B2527EED504C1E67E2DBB0025813EF708902E76E682C1B5EDA8303C4A59566760F3CDF8248E2DF05D1C81B9A6C97837
51B5F7840828935F36700E7AC932534475EBB6B47F328DF1D91F24C11361F0D9FFDDEF3A9D9FF1DFAAAEDDB80926608540FE
2C2461D20E050A8B1CC4A5FF62859E627CF5C1B818F8DDD06BFB64DA83595F6E39BEA5429068C4343DFD64E403958BBB45DE
FFB49AE37C4810182E4FCBD7D6B71A473866B576826D790C1F05D2A2BDDB944BD7854529F77A5ABE178810F0BD1723E5752D
461266467D3DD8CA21BFDCBD29AD77A9D99949B27C30EF78F26A1DA2866DFD72B43AE2518E76B78768EE6A20E6B9FE61725A
959EC90562608AC7124F8855351AAFDD53AE3EA83E3637AC76E6F49E1749825F6727C6AC69BC78271C61C06685402666BF00
952A3BF68E5CFE2FC6E7FFABF3D3E70FBDAECFA6757C362D6DC9D769F99A63E7148B365B166560B3036D505A631862A16F71
D4F896A5252CF3464D33302BD4FB3C453FAFB57829B89E11A7E469152AB875F06AF2BB1FF9E7EDEF165130AC8458CB0F2A3A
96D87213B6F24B301889945C2D18BC3FAE60AFDF27B81DAF6EEB46ACED54CE8C8D7633478BD7EE697BC946C193D727D3B851
57C7E31540209009C128E8EC5EF44249F0441834CD74705B707BC56E69F04D05BDDAC0E29A89299D4B803CA86C3084472937
1E10009367D34141448C12D83BB0415F028299AED24FDDAAF2A7D87D0472D4721BD1FDD14A1834E3C58F689456A5DFEB299A
B38CBD7A173F3A593DC844096B75D22150C3E0F6752F9F949EA802E6A5200285F9D720900943D3DDF3CAD107147BDBA7D30A
7E58DAC31F7DEE45E3A902C747FB47D0E45CED815322C715EC153B7861F7956C1ED4062FA4895DE197BBB7668446A54AC9F4
E28832EFE8BA60D8BECF058C913397A48EAB38CBB6F16EC7A9FA8751C2BA6D6B542EEFDDB4EA3555C188EF594C0FFD0C2190
09D2E49859989E6A712778D5E8637C3AAD6FEE17565B11A118CB2DD1B99E117FE7643185EE26A815D01DF105DCACF4BD5E7E
964A66753BAFEE58A293FD446A7A7D8C264BC673A72D3CC795EC803039A31F9972E781D755322ECEAF332F81F7B023AE8D42
EF220864A26C6D04E1DEBFF35A9AFCD3695DDF7F66565BB52B97488A6516E9F69C107DE7C4DAE126CC2FD5C326FE236DBCBA
80FC8355C345D4A88E01D525B7DEBE507008B6ECEFE81B428FF948E62C6383B140ABD23F890660138C049CEF9CF91FAF9C94
3A67FEA738270BB6350432617B0566B3302DE5B555C14C12B701039C2FC2E2D235AB76F1BF77629F0D910DB1E081C2EFB95E
7EEE46A52F4E09EF02B672C837774F2DB30436EB1380E99C43D359A287ADE544C4287942DB6B68657353F7CFB35F3635CDB4
6C14F46D59A767B1608B43201388BAA3FDB576EA664EEB5DE76443F562F421016172C715EC552EFCAC6278ECF81B2A870B2F
151DEBE03750A02E0DACB263894E3037E2D311AD2542EF2A7E586EB96540AF36A0665BDD5D67875D61FCF3BF9E2F0C668E5D
4BFC72BA086E144020130D2293F6AC5CF08AED0230003FFF7F790F636CAB6A7A035150A6DF7A44F89D13FBF24D05BCBDF804
13660A2E768F6DB866EB15D1219A84E648A070609B3E5E49F9D8311FC90FCBD89E81D221966D7BE69B356A6650804518FC2E
8ECA589835C6ACBFF262A270A311363A0432B190382E2DA18D6BEF9708839E754EB8C9648B150473A3CF15F9774B59BB3D44
2D0C136CF1711EB5DCBA4A3F29D5096DBA163C15F362E1615A5F066CD0D1B1FD316757C14FAB39371FAA541A2A6C9E98552A
614A52DF968D8CCFFFFB71CC55A007FEFD5EFF7E174971A18DCE481088F5432008EBDA955F83883E27DE78F7624735A3D376
2BA8D511F1199A25CEDCC5CEDC8232E88268A14FDC115CEC563E9867D3B5E816B51E4D58AA34C8618336B419976FE72DDACC
CD2CD2A228A57CF2D55D85ACEBB3074EFC27F3D2664162BCACBCD42414C2168740DE29985ECF8F8DEE5E38675C8D3FCD4DF6
F97B83EE47D59D1D36BF25428E56371AC7A3BA0185005B1C7085E661D377FCC9D1B128C78F56C1A6AC6E3202C5BB6C3BAFBC
DE40BD387F5AA69FB8FC7D45CB42B3B615B6350432990BA7BAA39D1FF760D8EB64FFC19D7D3BBFEEDB3F8D15F99549C8A64C
15074650B7F3926F97B2C2EE2B11D4DEEF38C73745849579A98C0A1B2DBF0ED12436DFB8567A1A4A82459BB8AB5CF8E3D906
2986519C29A9F85C54F2378320611402814C056695CA2814E88668E292CFA5D51FEBB9E154AA9D50829D0B917DBD9875E1BA
DC68B26B61D0C02E0FA57BB6F26A6DB4FC622D3FBCCC3B87116FCF8D58DF665AE26C91042D9D149404B889AD685D27CC7750
F77A106698F70A02995A0803FFBEB0C0415A3DDBAC69A352C5142A3C3042FEF522967F987CD2523B59E373304843681E999D
B1365A7EA67CC037776F2725EE57BE1D1D3DC8B26DBC65DB790D6DD474C5D70E05888AA6CBEAE799D56DA3100864CA65012A
55B6390B0B1D94EDCE14AB9A5A4B04455A84C1C5EB72D48E8F12AE959EBE5F7BD9460BDF2BEEF0485E63B7818E595CB3F321
E1C28D5C8B2F011541E4A5D2EAAF45C5EF1B78D170368640AC0454D520A67F0C5E7ACA0D4CB5860808977FB3D8E26360B7ED
FBB031FC66A5AF01B5C9284F1DFC86830F7EB6CF869329F0E37E929F567352F3B5D434493095AA739730CF41D5B58F406186
0B08C45A200944CBBC22CC7790D539E2261EC56A2757E2DEC1161F8307A9767A2B81D6971E51EE3D22EFB7C5C2D731E987E3
16DBA35647C98868E50FCBD9E1514A9CA2311DF59CDB62FA4792CAAF1079399C872110AB023772E40D0B45C533D43D148C20
27106347CF89E72C631796DB631C83365E5D58E9E9664E95EDE955D2722DD123798D1DB65A7E99DE7125FBB89F4467A0A626
C074BDF2C625C018D18D049324CC630281581B84495A282A9A21A9F81C515450AF7ABD43A8B3ABC0690BAFB317B1B7A6652B
86AED24FD1076C2F983C4660B4BE0CAFCC2DF6D6643D83E8F21DBC757BF9231CCA2E979A81B3C2420779E352B3AE07CEBF10
8835CA024CADEA3A242C7050B4ACA56405CBEA0C0B3771F79D14A935F6956447699087967866D9E035041447F2BB1FF9E6EC
B1ABF6526B893397647357B3F3E8944DF86592D1A5D5B32D4E86FCD851080462AD9835EDE2B2CFC4F4BFEBB9F7A9573B9C20
1FA4AABF5DCABA72476157CD0AD6D6D09213492DB76C6FEDC08CD98CB80B0587ECAABDD2F2B53F2C635FBCAE20297A6F86C4
8DAAAE03C23C0765A70B8EC0B0C61088158F5602D1B1C22C6E87F5F3083305C3CEABB5B8D7251998724BAAEDCBC1E04AC909
5B4C9E8862485E77927FDE3EFB69A9FE1174CD6EBEB3AB5024C5A85A47A3205152F699A4FC3393246F1402815837989125AB
7314D13ED40EFA52B282DD03C8DA3D7CE7C302B10CB39F66B5D16C08661C2DEC49F1C976B19366425132E4B6E2A7D59C9C12
CA9E1DE08858D9BE5998F79EA6FF0C89EB46211088B5EF17E04651CA58B4C36FCCFA7E4A5631254FFB9D132BD49ECE114268
C7EF565FB0B96213245E3A907D3A7DB39D34535D8BF197F59C1301121CA76CD02D1DFBBA88F65769CD1C54590FA75B08C456
C4BCA265ADA868BAAAEB20252B2857E2EEBE929FD7715A18F692BA3D84E671BB3AD026174A16FD48BC5DC42BD06809AF4BD2
459BB94D1D94ED96987E48DEB40C181D3AD6B551D2BE9C7E21109BDE2F40E465965B8A95FF4255D48C3F5FD568FC792DC7C3
DF5EC2A95D29F188AAB3C9A0C76DBC3AD7D8F9F6D046E575869FD670FCC3E414AEA376E4329858E48D8BCCDA2E38CF422036
0461565A2291163A283BB653B28226C47280FBE37236ADCA2EDC0E812A78501F6A8B25EF16B6BA252E379829DE4C1ACB6D44
E9D22DBC8E1ECAC6D3004A00E80130ABE839B746211088AD812A6B44C51F482AFE89C84B2959C1CE5E64F976DEC1D3620A9F
E18E8311E6D012CFA4E69BB658F86159EFD9AC1D43328A07BAA96936CE5FCF391F2AA3F246C1A09F256C51D3324CD70B2758
08C406B70B54AAAE8360142B5A375073ADC4476F3E54FDB8825D5441713B5469905DA59FCC66C4D962E1F96AF6159A47C550
3E9515B899BC7C53B17013B7AED548D53A9A35EDB286F9968D02EE5D38BB4220B63A59A91A44B40F25E5334DD2420A6F171C
3D4771EF028E72E45AE9E992FE4C5B2CBCC220BD5F7BF9616318851B88D187AC72E1BB9DA7723F1CDF2850B4ACC2F483706A
85406C75BB0053AB7B8E5A62205374BB0045C9D0BB8AB9AB3914B6D12CE247D0185EE6D5C4AEB0C5C2239829B72BC12F772F
851B283E43E3B8929346D174C9631B051DB2865F4445D30DBC2838AF4220B6BD6EAA9BC5257F1BDB2E28A26405EB5AC6CF73
A9ECF85D3A901D517E76486AAB47F38DEC8A23D44DA62C95E347BC251B0F08A472CA86D5D20E058C1917EB30C3089C542110
9B86C434EA3E8FB111BD91921554A870CF40E9F21D3CAE80B2C9E9925BEF4456F8280C521B2D7FBF84E199BA5EA0E650B275
6A5B8C8B36732F4450569762BA3E79E312CB46013F06CEA8100825B60B5AC4F48F2C9711149594AC606691CE7125273E4343
D516BC5E7EF656959FED965FA8E186959E29EC49A164EB4425A9E7ADE550F886AC8E79555434037A14402094C1E25DD0EB66
895DD0B98B9215EC1DB4E4A33972564CC9DA211872C59230F1A6ED56C16836A4B7475D2C3C42BDD69129F0633EE3C7073825
BB1F6EE22B5AD7090B1CF49CDB702E8540A8B35DA0AA1715BF2FA9FACAACA36044328391BC10215FE2CC1D62A1D4ABDD90AC
F72AFD64F960AE4DD7C2E25A10BF84A45C6AE15686C9691BEFCC25CA8629D073A3C4251F8D053364C0891402A1CE768159A1
EADC252A9EA1EE3D4EC90AE694E8E6ADE1A4E451D0099CD69711567AC6D6A30031E5FDE7B2773673AA29D63A59C53AC7956C
AA9E5E119846D9B15398EFA0635E85B3280442311079B9B0D04156F33D6E1252AF768C3E64F52EFEE92029F5AA76A73AF07A
C5398CB06DFF768D49F9A03E34ACEC0C959A06C747AFDC562C71E6B67753331D92519426299F29ADF90655D6C2291402A118
382292372D1395FC553B7C917AB5932BF11301D20DFB052A0DA5F2B819CDFA2B25C7135B6E50A02EB5CC9263894E048512ED
89A4D8C133E2EDC7843A3D15B3079298AAFBB020F73DCDC05992B097E4A410881D41624651AA30CF41563F8FC429E82F1D93
AC9EB78653DB42A970466DBCBAD092130DEC720AD485AD1CF2CBDB5B3D524C99D6E9EC4556EEE47906BEC10E950D395620F2
3269F53792B2CFA81A191502816046A6B4E63B71E9C706FE03EAD58E5665F8791D37368D5227BC0F1BC3AFD14FA94D0A0AD4
C56436A6B6DFF3C971A14CEB94D618166CE05C8F51BECE9B0D2642ABB3A52D05CDC0596044A8BA5C0994CA21C22010BBDE2D
C04D3AF67561BE83BC7905F56AD7DA655AB68DE773853ADEE00866BA52E241A50C021DFC06F747AB841A2E35AA939CABFD69
352797FE5A3B6F2A0D2155D8CCED454CDF276F5C2A2A9E611024C0991302A130666D97B8F46349C5E788926ADEE06CBED9C5
5DB8FBB88832356AE65485D08E37B1A9137B4AAE974456F8DCA90EA446756EC5A9166EE4B6305EEBCCDD6822F5069BD92BD0
716E8A681F2A9A97C3A4C91008B521CC4A75F76151D17455F7418A554DAB277CAFCA56B9F005628A84A3BF5B7D31ACF4B409
A394A744ED48895BD272A3D960EB154150323042BE66379F2B78757F2349CB8505C246440181A9951D3B0579EFE99857E09C
0981501E445E262C7090D67C8F23028A55ED6E82EA97F59CB66E840275116B05978BDCD23BA229D646420DF752D1B1E4569B
0F94A750E1C7FD24DB8F096D6807E03531490BA4555F8117A2A880132604427970135FDEB8484CFF48CFB941B1AAA5E56B1D
575224227D56672C583DD98A218AB511394A960DE6B825AD4031DB166F1CCB8995E8E0690AC6D9D60C9C13E6BDA7EA394698
5570C28440280F49207AEE5D611E057D0E8B2BF53FAFE32464D9FC3584B130051E0FEA4329D903F96AF68582C329AD776CBA
16834C74C37E81DB7909C55A07330CCB9B965B32240AE2E16C0981D809662D434CFFBBA4E20B54D540A57AD5361B9DB670C3
A394B65E91B2819C0B05871882266AEA52920415744B5AA1436C384275EF10BA7E1FFFE48557072B3021A4C16833A70C065E
8C98FE91BC71314C7C0081D80F042A5731F68A8ADF57F79DA052BD5A19A6953B7981E1B67DBD1AC5D150FAC93BD51728DC03
C55A4108EDF8BD9A4BB65B85CE5E4B906DEFE057DF8445CD2482DA4AF422623C9EA1762860942446211088BD409A24F9C2FC
B13887148A650A26EAB57BF8E76C3C6441C5504160FEC1365E1DB5BB601D8BEEFE68255FCDB6D1F2B70009EAC2F30F7BB504
25499B096988AA5B65F53F8BE97F374A72E02C0981D81598615852F595A47CA6499247994AF50CA2EBF709DE2802ADB58160
C62B251EB7A972A3FF25684CAADB550181F9AE365AFE568669F52E7E6084FC959280B09D28C73A76A4A8F87D45EB3A4C3F08
274908C4AEB0042EE83B61095CD0B59F32951A1841371E101CF3B16107B0A2DE54BFBCBD5DC2167BE884DDC2D6D3E99BAB87
8B6CB1F01D3D0890A001E1D489064CE2461563AFE5F8603808CE9010881D6292D12C8708B53F126605356A34C8B2B885BB9E
B1D5CB620A83F452D1B1D8067BC9658F11E62C46AC47F21A04B3BD63ACEE0164E37EC1B910EA44D84695B5603610977E42A5
FD430804F20673B261445A355B52F6A951984C8D1AF50D5B4E106CF7B25852CB2DFFBCFD5CE588FD7442919677B9D8ED4685
8FED495026BAE9A0E0448094326DA163858B8AA62BDA36608661383D42207608615669FA4F0B0B1D949DBBA851A3AE7E6495
0BCFC3DF2627EA7E090348826C469CBDF5C3166EF58994B58DB696EE81C9356F3D2A3CEC4D9128462481A818FB04B9EFE946
82E1DC0881D82D88BCF4D743042A04316BEFB1E4BB3F1B6C7B9BBAE42879BDE2DC95921306546F6F9D10C5D1D4B6FB1EC96B
F5A82D852F104AB083A7C5BB8E8B50334985565035C8EAE68ACB3E3549F3E1C40881D82D988129ADF95652F69951944181EA
347598966DE75D8CB43D07B0E2BEB473D93B5A79B5F6D90FA53A6168C98910DA711B2AB346479C09926E3E28B0A1FCC82F41
CFB969B97DD0B619CC0970628440EC1602536B06CE8EDD44384081EA54351A9738736FC4DAD8BE075B3114907720B1F9863D
77C56E61CBE90CE75F0F506CC0FA2688D190DB0AA7ADBCBE21D4F61F3FF93878113C3E8040EC1E4B0AC57C0759DD5C92B0F9
54834515FAF9EB3849D9B6B4114D90446485CFA522378DC9AE33D1902451D29FE996B4A257DC6E2B658E4DD3CC5BCBA9A8B7
F9ACD0665DB7BC71B1257891381B4E8910889D63096754F9A5A4E29F401ED87A5D52F3B58EABD8F46A5B9AA57318F1DE99DB
3AF80DB02B1ACD86B8C670CFD4F55AC43604521E5DFFCB066E62A6CD67E3320812C5A59FC89B96017900FB210462E710A84C
D5ED2AA27DA819F0B1ED8A90A33762950B3771197D36B3E9D1256CF6C9DE99D5190BFBE138529DE82AFD6460FE419B286D43
9BC9692B2F28D2E6C37D6806BC05B9EF69FABD4802859D1002B17748DC284AB124566E72B2E97AA834B8D725E9BABD7CB9D2
36BCBF14066950D1D1EB15E7501C81DDF00943D21EDF9CDD776B6C2039D430DBBCE588CD5F4EC411B1A26DA3A8D0C1C07F08
BB1F040219B52456EE12D33F9256CDC2F4FDB65B0B26D7BCED9870FF29916D683192B8531DE89FB74FA8E1C21EF83B5AB835
2752D6A5B6DDB7F272EAF484F765D9DA3D7C9E10B3DDA78DC8CBA4D5DF486BBE459535B0EF4120108BB160E22B9A578A4BFF
61E0C5D8F052D2695AB9931F74C336B673333A63CE6438B7F3EA61F77B9E6422CB07738F27AFAE1A2A1C1D0BE660B5458D4E
56CF5BCB29ABB56187C3B18C4833949D3B714400FB1E0402B1CCBA984E3B7C4958E8A0EADA67BBB52828D3CF5BC3492FD059
7F516B468A4FA56D2AEC49817DEF45600496D79DE8FE68652B6FDC80B5526140AF31FCB29E13F94069BB8F5ADD736C2CA461
08EC751008E40988BCDC723FB17E1E49DAE45E28418CDE88552DDAC4EDECB5F633FA1E519B77E6D684A6EB0449C08EF7124C
98F151EBEDE38F560D4818565BC80126BAD95560BBAE05986144DEBC4254F2376AC431834020133639E80725159F4B2AFF85
AA9B6CB1FC4209E67A46ECEC2A3021561D0087A31CF6CDD97DA3D2D76436C25EF74AF4A8F661639847F26A9662D04AE5344A
FA87C99DB6F2FA876DD27BDF242D94547D09CC0154DD02FB1B040279028E8894ED5BC5F48F749C9BB658FED62ECB1DB1B3D6
9DD656AEB75C3A08A6B9AB8C72D8E55E138D49F5A03EF464DA4620A8ACB3841985BA9F5673E2336C326A818E152EB48436DD
47A052D8D92010C81348DCA0635DB3E44F64D864FE4430333BAE60A717586F54432DA2BE4A3F19907740A8E1C0FEF6A6C2E0
6143D8E9F4CDBF0A03EBDA0E62F421AB77F18FFBDB64FE6E75AF9BC5A980790576330804F23B1045E5B86BC1E8A88D9D77AB
B5C4D960D9F21D3CAEC06CA5CF16334556F8F864EF622B86604F7B1B9316D1C4375D3F95B67158D66B6DC2C084907ED7644B
B770BB076C2CEE046EE4285A568B681F1A45A9B08F412090DF81E9FB25E53325955FA29A76DB2A3930D656ECE0B9FB5AA9B1
0624C1AD2A7FEFACED03922ED8CDFE8430D02634459E4859DB2D6AB53661904BD7CF5DCDB91D6763992C1079A9B4FA6B59ED
0FA812DE908540207F301C1081A275BDB8F463032FDAB64A9E56A0FD71053B35CF1A8F0F501CB95D1DE09DB9AD4FDC09FBD8
9F448FEA929A6FBA25AD6864975B55C1585CF30E37A18BBB48A3B5A57D3630D2C5257F53B66FC18DF0540B0281FC1E02D368
87FC444533D4BD1E36546C990277F795AC72E14BE55617E8D884196F56FA7A656CED1577C00E362118CDFAF48E68B7A4E545
BD695655B0BB092AC715EC5CBACE861EA666C0C792FE60D067942461D78240207F803489B30579EFD9564284FA56D3BC351C
FF30ABF3EAD723DAF0322FEFCCEDFD6206EC5B13088A23C57D69EE8F5626B558D17D99B62ED3B2ED3CB7F3121CB791D18EE9
949D2EC20207032F0A762A0804F25CCC9A0E51F1FBD29AEF70C4361CAA1194BC764F397F3DA7BECD645505539B14574A3C7C
B2770DCBFA60BF9A700812AF63D23D53378497795B4991801808B9AD705CC9A6D7D846F463B3B65BDEB858523ED3242D843D
0A02813C17CCC094D6FE20A9F827222FB38902770F20C0403BE4655D91E5841AAE7FDE81C07C57ABBD624F01C851B24BD8E2
9FB7DF37778FCAA878FCB329A5AED5B86023C7C35F6213FBF196F845955FCA1A16025B0076270804F27C7B0791283B778DC5
32BA6103EB02391AF5483D7715A7A8426F3DA51A94767B676EBB5A7252A285B966DE391CE5487899D7F147AB7B446D535E18
D44C5E8C94830E595E6703DB057AEE5D11ED0330DE7144043B12040279FE3A8BEB7523C1A2A2E9EA9EA3D65FDAFE1174ED1E
FEBE9322EBB9A3D6C8AEF048597BA73A506D52C2EE3439280DB2C4E61B6E894ED98CB8A9EF00EDA6A55B2CDE05286AEDFB05
9AC173829CF7B483BEB00B412090976092E60B726DC0E1902046EF27A91D57B2F3E8D6B25190D315EF96B422ADFDBED16C80
1D6932413053F9609E67EAFA6BA5A7CDF854E62320C9D1C818D5774EAC943CAD353F319230A9187B85F90E7AEE3DD87F2010
C84B306BDA8545D3A5B57308B3558764E919449DB6F2F69C10115670411CC5907B35411E296B2A86F231C20C7BD114A84492
E811B55D2C3CEC99B27E6A732C0EB3D12D47049B5D053CA1F5A61FC50C4C45F32A31FD63A3381B761E0804F2D2E962485A35
4B52F96F6BCEA166349157EE287F5AC329AB9D7AAB5CA0E6F8E7ED3F97BDB34BD84AC27BDF538A44278C6FBA7E2CD129AD7D
2A6FDBA5E469BF776285DDB7DE5324545927ABFD415A3307FC05761B0804F21270D35884C3B24F8DC214AB2D647DAB094882
3397A63ECB5B13A7D223796D4499B750C3859DC71A3061C69A91E233E9CEBEB97BC553E4F0A954E31EFE929FD771EA5AAC34
5FB6519426299BA9685983E907619F8140202F81302B357D2745B40FB5C341566A0FCAF0C3DE62A7ADDC61F6549E20E3041E
DF14712CC129A323468F6A61CFB11E489264CA076E55F91F4D585AD0933C2565686C372DDBC63B704AACD258630C643DFBBA
A868BAAAFBB0951F14422090A99F5109939E735B58E8A062ECB1CA097F34314BFBBD13FBD6C3A99CCDF86AB65FEEBE3319CE
2DDC6A330E1D09AC111DA2291DC8F14CDDE097B74FACE54F7E01E2D235DF3BB122A2ADF11C4133E863B980301408FB090402
79252619CD6AAF2174F420BF6CE0EEF110194D5376824FEB4B774B5A7EA3F2BC400D73CA583504898FC8FB6E57071C4D589A
3EE99E065A3DE11D2CFB6139DBAAE2698C8B6B55F76141DE7B40FFC34E0281405E09AA6913163AC8EA1C49C2BAF2C58B65F8
212FF1E2CDDCCEBEA9291830392F141C727FB4923E906530EB6057B109F4A8B69649F3CEDC76326D638F6852B384F70E229B
0E0A56B9F07B0751EB7920042A51B66D16D13EB466DF210804623D60BA3E49F94C69F56CCC60459E48A899BC1DA79AE3C48A
4E564F49010A7A920FC72D0A2DF104E6274EE0B09FD8102449F0D5ECB4B6FBEE8F5685969C541B276F579F56A5FF653DE7B0
B71539183CCE8050F985495602FB06647C88C0470079992A30B264F53F5B5B36047AB5618E13FB988F04C727BB038F7B111C
4B702AEA4DD5211AD8436C141443FAC49D372ACFBB25AD486F8F9EB46BA40F2D0E066CFF30B919B38AB917515448ABBF91D5
CFB3E6DBC790494652B357407314D07E7CFA2AFA1A550F43C10019B5644310293BB68B4B3F36F063ADA4485DFDC81267DEC6
030281785223C3102401968F23F18B232BCEB3E4831881C1EE61EB005DD7C2ADBE5878D8FDD1CAAAE1C9481788A264C86DC5
EC85AC5B7156E1F06F146748CA672A5AD7620698BD0B323EE7EBB8D99F71521D38A9D39FBCD8C97FD10CA793705F14F2F872
E22911EDAFDA91CBD6501E9E10DB7A44B8D899D7D63DA9EE046DBCDA93691B4FA73BD78CD00CA81E760CCA408E9232BDB874
200734EEC9D40D1DFC8677FD8D0A157EEA82F4EB45ACF88CA9DF6BD2F3A244B40F558CBD042A859D016291CA2C1A37F3234E
DA0C4EFAFB4F5F69D3C515FB481C81CF0742E2061D2BCC9223A9D77DCA0B2357E2EEBE12C795ECCCA2C9F3EEE3AA9801F907
8EC42F4A6B8F126B050449C05E4141FB88C0B8AA91B4F6FBEE8F567A676E7BD759170562EC988FE4BBA5ACD4A94E91A06386
080B1CD4FDA7ACCD9D183255481BBC39197F7DAC0AD2A6FFAA0A6670D23EC011057C3E108B2C103E12E63B28DB9DA778FA32
1017AECBBF5DC28A7C3049EE613A447BBB2AE0C0839F222B7C06A5DD28D4C95407C14C4CF9C0A3D65B6E49CBCE666DEF11B5
BEBBEFE208B0435EE2EF9C58E90553798145D3EF2DC87E4F3B120C5B1F3236DF63BCBC2F7F950433B859FFE4A47F302E0CD8
29FFADE3D046A15904B1F823555A4216342E9AC232184DE4D5BB4A20092EDD50E0EFFE688B1C25F3BA138F262CF1CFDBDFC2
ADD1A3F0E2A11D61C28CC3B2BEA416A00D969F4EDFDCCAAB7D475FC4E498F79F147DBB94F52867CA760CD43DC7C0E886C10A
20E3E805F5DCACBF3F5605A9D3A50D67B8197F7BF24F59F3051247E15382981F872C983B552AD1849011D1CA6F16B3CE5F95
4D82E7367D20DBFDD1CA33195BAA860BD5260509DD6EED531B988D6CC56066E7038FE4D5A03F940DE6BC8B6F19E1980F9E16
7FBD98153315776C495CA7ECDC252A7430F01FC2168700E46D579FC800F6A3BF208ADEA72201FC99F111891BE153829875DD
62FA47D29AEF704432F9DF6E3012C1B7145F2F629D099221E8BB5DA02B870A3C92D778A66E28E9CF90E9C530100104C14C3C
15ABB837CD3B73EBE1F885492D37B5C8042FDF7C31E6192899B59079F59E7292AF2BE28848D1B64944FB2B0C610479DC1BF3
6771D21EBB170AE9EB490213D0963F7B886010D6C1430408A61F90547E29AD9A65D6F54EF257AB3484D725D9EC854CAF4B52
83E15D4D98044914F7A5037BF044CABAA2DE54894E8811309D01E419618C9B41AFA863955E2E76738D9D1F5478A47742E322
2A547840987CD602A667A0542AC7277368CB9B9649CA679AA4F9B095214609E399E30307796BF0E828A9E8087FE610C141D1
194EC2E911AA02035356374F52F105A2AC99CCEF1D8F693C7B21CB3F4CFE8E321D004B30BD23FA58A2D39974677A7F96442B
30C35333C88BD5A3CAA8E815B73F6C0C774B5AEE96B82C8711A79BA014993A03713DC6724CB671BFA085619A9C1AA19A5659
FDCFD2AAAFAC2A461964AA500F243FB993C84E794F339265E9999CD2A71715D3A60B8A57C1FB8910DCC457B4ACB1181492DC
49FBD261B6D9F990F03B27D6CD5815FA0E0E0E647AC9AD2A7FD787BF04E4EFAF1A2E92EB25707F00F29A18503D4FC52A1BC8
B990EF7AE0C1BCF339BB1BD9E513B0238191F9A57AA72D3CC795ECC9F13F441455D29AEF64B53FA2AA06D8AC1061E9BAA787
058FFE826A46C6AC423937EB49F882199CD419886A100639B477558088951D3BC4A5FF3008E227E71B6B9A8C0B37727F5AC3
492FD04D784CE34E41935FDEBE830F7EBE5EEED3C6AB5319E5388C52087973808C541A640392AEB4F6A893691B813CB8567A
BA4BD8FC673E932447197DC89E13A259BF30CF85C8648A777B9A609216492BBF94372E316B19B041ED1CB35ECCCB99F9644F
805FB088F835E9DB6F5C0B52DFB3043924A1CF955D43A07255F75131FD233DE7E63BEF9966F25E92FADBA5AC953BF94D1D26
62E2BC5A748826B5ED9E5BD232B7C465C9ADB787A43D5A440DEF174026606D351BA53A6107BF21B6E1EAF1E4D507637F8E28
F3EEE0D7BFF5074AE4F8A51BF2AF17B196EFE01596BFC3409A4671E658B8E37530DC3144CB2AFEAD5341D0934B888A8EB0A7
AE0569D345157B49029EB4DAB72A30AB34FD5E9310F498C337BB9E11CF5AC83A7256C2134E98FDDEC4A9F4CBDDB73FC6F152
D1B1AAE142B1966F341B60B342261A528F6AC55A411BAFF6417DA867EAFA7DD18E01F907CB06730D6F1EEF0241495A957EF5
2EDE57F399A72E4A85EF26E587419020A6FF5DD9B11D370960FBD939E26A574EFA87BF3A15FC453392F924EB815ED0F09BFB
89691F607A3E7C62763DD9613AEDF06551F1FB9A81B3EF4A7810A38F72343FADE1FCB8827D2741A5374CC01EC180B4EB7AF9
59D787F33D5337A4B6DD1B9476AB8C0A78580079E7E3654C1E4875A21E515B5AFB7DDF9C3DFB63E67A24AF8E6B8A189074BD
C1A020473902F3783CCF859BB80F52D5137E3357CF8BB62441E87625CC4AD870F60C8E6879B99F3F5DF7D3FF81199FDE4327
30849BF9B7279911C60E11D2E021825DCF728449CFB939960AE1F8BBF87CF6D81601B089361DB4785F9BCD7F6AEA1B96F5DE
AE0E3814B7F068C2D2A8BAE0565EAD5C2F31C1CD01C854003A9EC22065CA07CA07F36E570578A4ACDD17E3782E6B474AEBDD
5E7107F91AF7BE1184AC6D36EE3A2E020364CD6E7E1E7D220F148CC27BA2E2E9EA1E37128709BFEC1A3DBF969BF5F1931886
E2CA0324FE9B8B30B2C6679223A44E973505C0FB89F6AD0A503D2F4658E8A0EA3A38B19F0C6C9F3B09EA1F57707E58C68E7C
A052AADFD2B39024C9664E7558E9E983B13F1F895F7CB7E64213BB42A215E8102DF82FD88290290723CC5A934AAA130E48BA
0A7B5222CABC3D52D6ECB9FFBD5BD2F29B557EE54379622DFFC5DDDB12B8238FAE5BB7970FB4C1F663C2AAC6B751B9DADE9E
619F33DD4BE7767DF769F74F5FF5EFDEC38E5C2BCC9BAE1D3E095323D939D2C6739C27210D53DF53F5C5FE6ED1D70C673D73
3F71067833CC9464D7AA80C40C824461015005BB27F063F34A75AB5CF8B317B10E9E1677F622E89B6F11C874A2FCEE477EB9
7BF744CD01736C745D4833A74A6C11031A98D91062B5A038A23629814260CA07AA878B631BAE05E41F747DF8CBAE7BDFB83F
5A1556E695D315DF256C01DDF877BF28578C2665E99DB6F0BFFA99B5E588B0B05CFF9AEEB8B8C1C00C0AE89AFD37C6FFFCEF
CECFA6757E3ACDF2E7CCFF64FCF3FFEA9A374D5CE0031BC5CEE1E77FF9F480C07227F1F7DEA79841FACCFD444B9043BDA006
0639B4675D80C8B245F4E9928A55F2AA1651469AACBC0C91CBDEFAE3CA6A0D5B8E0866FDC25CB9939F57AAD7BD89170130FF
CB07F3AED24FBAC6CEDF17ED18547434AB33B64FDC21D38BF5A8168A01886D8113981ED5290C5291863F24EDA91A2E7AD472
2BB4E4E4A9B44DA07BBBDCFDFA5882D385824340F1E6F7243572CA7A85DDBD2C495AAE71ABABEEBBC5B2D53B94B71EA062F1
4BBFC2840E9F3DD539F33F1E8B81DFBD3E9DD6F5F317F2EA72D81676CB5848C38F7F7B27F139E13204B4654FEE277252A72B
3A2348E8A665C7A8DAB2BA9CA631BEFADF8C2F6730FEFDDF8C2F1DBABEFDA47FFF4E4D6FCFF86EC26B7E0EB06E9C0F59F4C0
12676E42A646AEC4B1D738341891F703EBC9126CF6E1FCBD513FF8E6EE896B0C6F645708355CA5518E6026D84010CA8804A3
59AF362AA43A11E8DE23B2BE06567936231EA802A08441CF3F91BAC6F5E1CF5B2367AD0DFA6A959FA3D3E9254B4FAE5BEDE3
E272F5F8D947E76E54F8C6350727B644C43545C4354680DFBAD37429E1CC86CE7FFE17E3B992E05761D0B773B39EC584CFDF
3E5176DF7F7A3A6089691CF1DCC488AA9EE867EE27CEE0E5FE00831CDAE936018E8B73B319B36674CE9CF67B5B63E6B42EC7
7FF1135E9D6D4DA325EE27A957ECE0CD5EC85AE5C27F94A391C8F1171D19E0040E6CFFBCEEC4EBE5E73C92D7EEBEF7AD5BE2
B2D012CFF4F6A8565EAD40CD01869501D5C16D01883D4080991733E911ADDAA4541A6472BD040806B196C7510EF64BDA5B79
35E90D051E37E216B986FDB0F5F2FC7D57B6F8879C4B0CB95B199ED41291D01EF920CDAB70E1E7DD2F9104E3070A33FF83F7
201A3E6DFB44485F3DB60930E65490FC1783A8FEB947030661D353F160713F988EA8876090433B445E5BCDF8D7FFD7F9E90B
E793AE1FFE29C97B6124E4DA66D43340E9B842F0ED12C12E374569156E78C6D9992449819A0DACFE8C8E18A001BC33B7ED8F
99EB7277B6FBA3954145C780B153355C382CEB93EA84C07A32998D245402103BC63496C0B19D575FDA9F0D864C6CC3B51B95
E72F161CF54CDE7E306AFD966B4EABFCE63B9D99B7DCE7FBB517E73887CEDF17EB74E6E8B7759FFF57F7672F5505E0F5C9B4
41B74366950A3E647B0351329F0D69C8CB73245F9C2B565038EFF109C2D8AEC29853223C44B02F8C7C7EEF5AA7CE971A1A8C
CFA6B53A39B205DD2A44A1352935888223176454F47885B72DDF57EDE89C3B7F67DABEA0C4CBE93177CA6EDEAEBC74B5D4E3
6CD6B6C3718B5CEECC76B93BEB68C2529F6C97F032AFC4E6C892FE8C4E7E235FCD023240655418501D06BB1CC42EF707D88A
A15A664972EB9D6BA5A74EA73BEF8B76DC71EBDF7BA3E6B8252D3F9BB5FD72915B44F9D97B3541F14DD7333AA20B7B934B07
B22A870BAA4768E503C5C955F4D337F2971E4CFC7A7DF2DEAF3676CFB40CD257AA82DE0D2BB4FD7DF0E1DB1BCF6644B26C17
64FC8D5FF015BF70D6F35EB32D27084F5401CC94649748E9252F9704E0D5F5D9B49ACFFF4FA0C78F6E694E076216BADC5AB0
2964D18A734B9D4E2F5EE9E3B4EDDA26B7B89D7ED947C0E476A7C6FF4143687A7B74E940762BB77648D623D470C45A81C220
D5985406B3DE8CA3300431C40E0143A09E454F68BA1E987FF0C083791627C3C465BEB97B232B7C8030281FCCEB12B6709423
220D4FA213CAF552A099C190D1211A03AA3761461443C0D8C108334E60E0859A718D1E93AB50A672B4DACD9BF1AA213CBEE9
D7B36281BAB303B685BD21AE3AF024A4E1AFA10BA7BFF835E3A92AB06C17CC80410EED0A4CAB6506050023E2E5F3093043DA
664E0B9A3D6BEEFEB38EBBDC171E39B52BF8E2C5B4B087F537D3DB1FD0FA331AD9657D927681860D2631B8E84320A3635E85
0C417362F38DF339BBF646FFB03FC6D127DBE556955F4E577C07BF81AF664B7522A541A645D446B3012CF76F31700862B469
804CDF7BB1ED93D7D82BF8745AAFF31ADDF0106C1AFBEA8726CDD390866FFE62A73AC04C49760522950EB8EE791D55D0FDD9
B4F8D91FEFBFEFE5FDE8CCC55CAF809CE341056EC1C56E170A5CCF65EFF44C5D7F34C1C9E2307067F68EDB5F1D8A5B7026C3
39B8D8FD7EEDA5ACCE87CDDC6A6001C1A70DA1FE80C24CF5ACB25B55FE6038ECBAF7CD89947591153E053DC9BDE20E895600
340090CDE03D7F32EE165F84C56768767B881C57B1672D12AD9975AFE8938F7B66BEFA0461C8EB246186D1EAEC8BDF84341C
3B146027FFE5E52F4EAAC3B38708925A0F18E4D08E26319974E0F0BECED7DB7E6C5ABE588C01EB446B1ED5E8CD2AB549A936
2A5446B9C22095E9C5129D50ACE50B355C9662B09D5F5F3E989BD9F9E0417DE8B5D2D367B3B61F7AB860E79D59FBA27FF4CE
DC06E64C5A5F064B0EF3774328C288ACEF6163D8F1E455BBEE7D7D3275C3FDDACBB54C3A4FC5044303C80074828E65A50A3C
FA917AC37EC1D78B59B37E613A6DE5F95C91E5979B467AC43D3B9C19AFBA83C0F8D241949E0A1BCBDE90369CE53C930C5150
E4A46365EB79C57A6ED1F35F7CBAA86C8BE5C4E1E971C30738029368D80B984EC70ABEF8CABD82AE99D39A3FFD5F7E9F6C77
5CA7DAE126BA97A81964BE4C3A9224811166601319509DD6A4521864529D90A76275095B4AFA33631BAE06151E3996E8E472
7736500BC1C5EE395D714C793F6C0E880D81937813BBF22AFDE4FE98B987E316DDA8385F3690CB560C290D32D0ED7162C276
5C11942CAAD01F3825FADE89357B2173C37EFECD87AA8E5E44A1C4F506021BFB1E716E16E3DB7FBC4CDE7F326DF0841B18EF
B0E1EC0A1237F3F2FEFD6CCE2305E33A811949122789E7BF462D510B629E462D18FB2D1DBB080639B41FE495E5AF3E41F874
5AFB971F14FA3FF4BFA15FB39B3F6B0173F6221630558222159DBD6F600791A3A419470DA85E6D52CAF462604CB5706AD2DA
A3AE94781C4D580A448267EA86A8DACB2DDC1A0C6E5841AC1260F8970FE5F9E4B8ECBEFFDDC9B40DF14DD73B054D72BD4487
6826BCD3CA557844B472F166EE57F3994BB7F0AEDD5332FA1095863021CFD964E33D8862FCFBBDE78FE54FA6F5ED5C676041
8F02BBC328E9E066FFE3E9FAFE3852C12B36695135FBD9A8059620871D61F07EA2FD60120AFBB66F7C8530F8745ACFAA4598
19379846152A7C60047D986639D6FC66310B2884953BF9B71EAA78A2B7E93304499830A3C6A402F32A4B61C93A77AF26E864
DAC6DDF7BFF5485E13D714DE27862ED3102BD81920F09A119A6FEE9E3D5173BC33B765753E1C91F5A98C0AA3D9F02E927371
0466FF30D9DCD59CD90B997B3C4425557A99027FAE18782AB909425659DEBB668965C7E093B1B045E37FF9625A9FEB3445F3
3930DA603BDA1BCAEEE8DF243C4AFB078EC85FE71705B4154F431F5B821CCE2109143E4F3BEA394D0D8C59EFBF2C8AD1B79F
289B1A7F23265152A3238659E6A847EAF5FB045FFECC9CB38CEDEA25AE6F35FE9992006B4B8FEA94463947395C36901B5971
DE2D71F9EEFBDF07E41F281BC801FF051B0B32C98CC8FBAE959EDA1FE3783275437A7BF4B0AC178858147B5737B87942B3F7
65D9B74B59E0E5E12F69EE30E9F404FE7A8946491C37ABD5AA9666EEBDDB237EE7989702254515F21A2F41CE7B9AFED32409
F7DFEC0E7EC1AC674C7E0749BDE7EFB227BF084547F8D3ACCA63DB0506412D3C44B02308525E5DD1E538F3B17DF16C94D44F
A675CD9BA5A8AF7DD1AF2228A954E39DBD4860841C98365FCD676E3CC0CF2F9D80E51B28041DA25518A483D26E301BFB64EF
B2A44D4C5E93DE1E25D10A60A341DE29A0EFA5B5471D4B74728D9D1F5D17D22D6A0362E09DE6E3D0EA88B0FBCA1F57B0BF77
629F0F95F70CA24613F136FB10244920086E30E0468B44D773224545D3D5BDC749DC009BD5AE4014C3DCEC4F9ECD81A819CE
7CCD8300BDA0819BF9B7671585AA279A24E0FD443B02084845F3CDFE43D3BAE6FC3FE30791E0CFEE857338D7C34C22E1EBA4
46329A480EDFB275B0C499FBAFB923AB77F3734A26CCB407B3319893B9CA117A7F5630CDFDC0839FC0740DA482DC20816D07
995806A55D418547F646FD70A1C0B5668426D74B2CC704EFF2BE0C418CA61768976EE1CDFA8579E4AC25E73810DB13F5E17A
EE7D51F107EA9E2304A6818D6B57A80752B8597FFF6DF6E49137D867C89FF3FB2087D0DDCBAE540189E879F70439EFC96A0F
2012AD81CB310985A8428E9BDECC3842CDA4588A6716E9D6ECE6FF6B2E73E3014165C3445A28186ED6211AA1865B31947FA9
E8D8BE18478F94D5B4BE0C04833139217F720810E58379C71FAD768D9D1FD714312CEBD5A35AFCDD7B580D8CA0FB3C455FFE
CC743E2428AB35BCE5FEC08B31F01F8A4B3E5275EEC211316C65BB42485FF324231258D685F44D84F90D2C355993DF6F0F11
66202A789DDCAEF60A0C3AE65551F10CCDC0993FFF69383E2A96E109999A051B2DEED387BCC460EA9BD802630406E4015FCD
A6F5A59FCDDAB1276A0E1009DDA256D8949037B6A6515D62F30DD787BF9C48590714A64C2F4630E3247C2F8693B1A9EA1F57
B0E7ADE53C4CD7A8B504F10E8E6D8DA25471D9A78AB68D989105DBDA7EC0112D2FF77F7E7B27F1C673B327BF081DBBE4B737
111C34436930C8A1FD40601AEDC07911ED43EDF085099CF4B8422C225A3567191B4C7DE1D14A3336F13AD38C9BB5887A50DA
1DD7187E2C61E9A1B805599DB1C0CA836D0A79256A93F26ECDC5FD317303F20E3473AA4047C2F049DA231D62A17B3D454033
1FF3F9FFD97BCFE8A88E2DEF9B4FCF5ACF9AF930CFCCBC77CD58BE8ED736E06B1B1060632372146072103983002144C6E49C
731202912424010A28E79CB35A3975EE3EA7F3EA1C4ED25B2DF9FA7231410285EED3FBB7CED212A2FB9C0ABBAAFE55A76A6F
BC896BA37A6C1B9745162BCBFC5E593C8D30D4418DBB0E7A5ED2CB2E0DF961FF635156B675651D8A262CFFB2B5E0F9C758C6
6A7889E042AA805069EBB6E1295F187857BA79D42698BA269BCF01FCFB51AD0B374A8A2A7B6ABF9695B4280C58564BFCFEA8
E5AB037E3C9FB29DA76A849A055E8BCA28BF9AB17F5DE0C873C9DB6BA46566C248F7E2242829CB387E9E60D42C415084CE6C
E9D92559AB325D9E334C5130DAA62D817A771DB0F4A5FCB0FFE587FE4FC7258C747F0FFF8478CEE6976FC20FF91F42CF87B2
7511289B4CC35983A77D6512DDEF89FBEBF47478BC61CC1CC14F53799702D44C8F7584144DA2195F8DB4F46AFA6FEBEE7BEC
7EBEA0549803F50BBCBC3E7033FBC8BA408FCB697B1BE5D5969E7138F0266C0473F391DA7D02D76BB3A4BCDAD20B4FB6698A
14F91E48185855D950FB2ED4A55BD59405FBE7657753D0F5F05B8481B2E07FDC84344BE10D822B9990445DBE104FEF6FC623
7A6A3982B6AF9AEE3C2E1F34A675CD0E8C2FEAC19528D4CFA3DE5EA4E18657DCDB1AECB92568725AE30BA86517C7643304E6
9F5B7FDFE37CCAF60619C71EA5A877B74E4965E4B6C3B221E3B8BF9D55C895BDE4298E30D42A8BA6C832FF6E5124810D0000
D04948335F513809751D56657A8F3E486FA043A274233CF993170993B37BFC003541D95426794673ECAEE7F3BD1F8D8FAD09
A1C06FA70B8A5E9A8CE63CDEFC78C2F1B88DB55879EFEB01447D8BCD6B93F4274F9E7FB0C666EBBDA75316A1AA6C2E9EF2A5
198B004B0000A0B3AAC0D82CCF192ECF1A4CE87ADCC3B03D287CA579EE5AF1B0C9BCDB41DADE1914B466753E37756FC4E28D
0FC7C5563F61C0FBABCB50C4CFF40D99B1E3E99C227E86D51E17A60F8E5615575A3C178BC6CE114425F5B67F4E9AD4A82B97
61891F9B24C1600C000074764E6DA893A57F8384016516F6CE13792262FB51D9A031DC135795BDF3449AA1F4566D012F7D4F
B8D7A6471392EB61EAC472C45AFE91D8759B1E8D4FAE0F37DAF4741FBD12CD29328F9B27F05C2ACA2936F7C1E3195A53E32D
8D73330AEF82490000D05955A0AFC2923E56E48DEC4DB7A80A1575FA86EA8751ADBB4EC8DF1EF9A55BB501ADB368725A9276
3E9BEFFBC4138904A87DF66121CD0F0A2FACBFEF7137EFB4C280517DE7A935ABD03C6A9660D62A5129C7D25769D0D6EF92BC
F8C8C0BB048601004027B1690A25311F290AC6F5F273F546FAD663CDA031DC8D7B70A5BAF7BA6EA40D34666542EDD3AD4F3C
F7457835CB6BC0065843B928CF3764FAFEA865F57855DF86E4CE2E348F9E2D98B346DCA580E3DDDFCA5A4E4A63DD748D0720
BA0D00009D8331E351D2383755D9DCDE7FB6D5C63C89D2BB4FE4ADF2C37045AFCEE9D01452A69704975CDFF870ECE5B43D3A
8B1A4CC1A9D19A5597D2766F78303AB6E689A5878317BC93E24ACBC485C2D9ABC595B57DEC8EDB28B88525FD5553B3892634
60240000BC5B13D016A3E0B634C14D53E3DD2709A0A8B68804FDB0495C240C648ADE5EECA56892AB6C389FBC1D6983E8EA20
B00727A58097B62568F2E9441FB1964FF7F5A96A4EBD75C67291E712D107C616EF164CD2503CED6B5585176916809D0000F0
4EECEE8E9B0E61C99FEA9A0FF7551A28BA2DBC5D186CD88DE90C7DB0CE69252D45FCCC9DCFE6A1AB51560D56E144E8ADDAAB
E9BF6D7C3026B93EBCD75C16BF05818458B31D1B37579094657484F2B12892E4D943944553083D07AC05008077AB029B5C53
BD1E4FFDD228B8DD87C92029E649946EF058AEEF21596F1EE8FE03866154267950F1D50D0F4605E49DB2511087D109A89214
FB86CC3816B751ACE5D10EF0DE5CA3A3F79D51FCE4C90B8E7494C8C5366D99A260AC3CC7DDAACA02830100E0DDF3748B5059
3C4D9631D0228BEDE3EE8B60EE8769878CE31EBEA8ECAB34D00CC59196EC09F7F209F6AC10158079382C24458496DE5A7F7F
5464D50382B239429218A6ED4EB066E844DE853B2A476AE00255E91C3CF56F663C0ACC06008077F7AEC62679F6107411FACA
3E4F8CDE489FBDA51A3C967BEB715FEE8CD29A55ED23CE68FF9C130E32E2002F2333488EC6AEDB16FA6B1D5ED1278E895E4B
4286F1C7A9BC6D4764668B03EDF667688B86B3561AEF66140582E50000F04E083D074BFA449EF7334D38C4267C4C4E6E3F26
FF791A2F2EDDD087C9A0191A8D38FB23976D0B9D598755809D380E15E2029FE0A9175377E92C0EB4A9BEA6D13673A578E146
094FE870D166750DBB25D11F19B8E7C1780000782756558EDD5941E178C749525DB36DC17AC98C15A2C6D63E9EA71BACBAD0
D2DBDE8FC68795F983A9F439144D45563D5C7FDF23AAEA11ED48E1DBB47A7ADF69F9A85982840CA303969B817B4E1AEFA6AD
DBD19B6ECA000070461886304982A4716EEA722F874A5846BEE997E9FC4DFB709264FABA88983261DEF6A7738EC6AE571A71
B099BEC268D3DFC83ABC25680A4752EC68690B8BD10F9BC43B7B4BE59845671207E1A97F535778517038110080B70F79A45E
DF7C0C4BFA44D7B0CFB1668554DBFDA75AF7F1DCAB810EF15E03D74B2EA5EDF17D320D2904309BBE287FF1E1E8357BC2BD30
9DC8D1D25655679DBE4CB4CA0F53691C34FABC559926CF19AA28186BD396812D0100F0B6C1D78AABAB56A07984C9F17622C9
95D48E63F2317304F9A56647488F95B48457066E7E3CF17939049AE9551A649C6DA1BF9E4FD961B2191C2D6D561B73F2BA72
F42C4162A6D1610B9034D4298BA6C832BEB5C813C19C0000785B7761E22AF23D6499DF59958E7896B9946399BC48B87A3BD6
6BE193DE49B1200B8D50175377594833D84F6FD88030674BD0E4FB05E7698774E39F96671A399DBFFFACDC91CB9026D4EACA
A5D2043793F831581400006F8130D460295FD863285B848E99C2A0089DFB44DEEDC70EB4DB9CAF6A3E1CB3667FD472075CCD
6619E98D311B1F8C7DC171D0B10C93939B7FC3A72F17553738BACF2B6DFD0EFB3184D6B360540000BC05AB2A5312DB07D112
3B8F484A6ED88D792E11D5373B50C7AB3629AFA4FFE61B3283232D062BEA09188689AF0DF37E382EA7D57117BDC362F43F4E
E1DD78E80451870CBCCB58829BB6760B4D428C240000DED0F1D216A3F0B634DE4D53B5D291D3999A63FA692AFFB7330A874A
958DB406975CDF1A3C35B3390E6CA97BA16832B2EA219204A5C21C874D2457482CF5912EDE2C912B29C72F5233162ECB18A8
2CF995343680810100F05A6842A9ADF5C5923FD7B79E76E474EA8DF4B12BCA51B3F839458E75DA9A61DAE26BC2B63EF17CC1
7904E6D45D9034F1BCE29EF7A3F11C498923A733385237C293FFE0A9CE294AD5A62956E48F9265BB5B959960630000BC7E46
668F803055963EC0F11DA49754D9A3D56F3DE888EE02B25B127D43A607975C038BEA0E49408657DEDFF478623DEED0FE24F9
626285AF74C916A9C39E467CB5B15B7155F9422CE9AF66490898190000AFEF810D7578DA57F29CA1A4B1D1C1936AB132E7FD
552367F053731DD1395BA93067FBD33981F9E7C0A83E68E4B2BF3878B0E9D104C777311D91A0FF791A3FE089D6898A575BB7
1D361C0200F016ACCA8CF6AD86639D22B5C5959671F304BE87648E99BC2A49F1CE67F36E671F63DA1830ADF78066E8B89A50
EF87E36BA48EEE6947AEA2B61E94CD5D2BE68B09272A6103EF2A96F8B1A67A3D6D5380BD0100F00A0C6532F02E4B13DC3455
6B9C22C1660B73F6966AF46C47716AF467EAF0CA3DE18B6E651F01EB7A0FD29B62363D9A50C4CF74FCA4661498C6CC119CBA
AE72AE12B6C8A26599DF290B2711BA4AB03700005E9D9AD9E49AEA7578CA1706EE25674973419979FC3CC1FE738E3BD3A9C3
2B9030F0CF390E06D6250AF9E99B1F4FCA6C8E77FCA45274DB797FD5C405C2C2728B731532A1AF56164DC6D3FB5BE4706A06
008057214DADF2DC9F64997FB72AD39C25CD6A2DB5EFB47CDA5251ABC071576E6BB1F2DDCF17DECD3B0D36D6496AA465BE4F
A68557063A456A1B5A6CF3D74B36ED73BE38590C6DD370D64A623F320A6E83D50100F00A366D9934DE4D9EFB334D38D38EA9
F874E32FD3F977821DDA134B85A860D7F305C125D7C1CCDE8950D3BA3F6A5940EE496749F08B6483C74CC183673A672C6D7D
CB09699C9BB66E3B4319C0F600007879D660123F44FD83AA74AE73A5BC856FF3DA2459BD1D73F07416F0D2773C9D1B51F900
8CED2D280CD89944DFE371DECE92608B95397645E9B95454D36875C602374B9FE2E9039425D309431D981F00007F40DB949A
5A1F2CF9337DB393BD01A7E9B61B0F3563660B0A1CFEAD6E727DC4F6B059A98D2FC0DE5E3F4211C6474597B73F9D63759E38
531DAF0F7C0EC89CB4CC095DA5A270BC2C6320044F0400E0654833CF1E2AD11E5935C1E9129F59601A3B57702940EDF8497D
5E716FE7B3F965C25C30B95760DA18A49A363D9A2850B73A51B213338DC8F6FC839C359400439BD59CD5B0B500008057F8C7
A6821134A172BAC4F344C4725F6C85AF947106C700F7F2CEEE8F5ACE533581D5BD0C4752BCF58967768B9389D21B0F35E3E7
39EED9D8CEA06F3989DABEB6D69726B560870000B4CF17AC466100EA1954E5F39D31FD04C95CBEAB9EB0405859E704EF76CD
84E942CA8E53893E06AB0E6CAF038956702C76BDD39DDF54A829DF43B2C55BA41A1DEDBC856FC1A364997F57164E04AF0500
007440DB649AAA55764F05ADCEEAA13731D3386A96D304A691EA8487A3D7DCCE3E06B68730D90C0F0B2FED7AB680619CCC09
646D9375DE3AF1EE9372A72E7FD2D8A82CF644CDDF8C4582350200D0DE2D34CB32BF47974DEDACEFBBAB1BAC33578A771C73
9AFEB95C94BF37C22BAAEA21985F4E4BD2A6C74EB69DA083F43C93E7129173C53E782DDA5A1F49F447FAD633608D0000B4D9
C31FA449E3DC14F91E6D4EEBB11F57D81DD12FF296184C4EB3961BCD09FA2D7259A5B8D0956D8FAF6AFE2D6AA9B3382C7A85
E048DDD8B98234878CCFD5258C825B58F267EA8AC5A4990FFD2100B83834A9D3351DB20749A9D9ECC4B9683F9F3861BEB0A2
D6698E8D336DCCF5CC0327E23769CC2AD7B43D835517987FFE40D472674C3C4130A7AE2BA72F1735726D4E3F2F50E72AF27E
91650D7222BFA60000F4109459A0281887A7F73749439D3A23D1297617739189CEE4A24DA6971C895D179077CA356DAF8097
BEF9F1448956E08C899729299F03B2655BA57A23EDEC1541937A75E5F2F6F38977A04B04001787D095B79F49FC91B2624E9D
91FC52F3542FD1853B6AE74A766E6BF281172BB25AE25DCDF0245AFE8978EFC745579C34FDCD3CC2CB5BB2F5A08C1DD5D171
3E5153E30D519501C0956128B3817FDD79CF24BE4C5D936DCE1AF1B6C3CED74BDFCB3B7B3C6EA34C2F71212D4A112F388FB6
85CE609C360BE5D5F60DAE872FB2640CB5C813E4D943E439C3ADAA1CE81801C065A12C6265C90C3CF56F46FE0D67CF8B1823
D7EDC2973B892FA397519914C7E23604B85250C52679CDCE67F3F2B8A9CE9B85EC22B3E712D1D540353B6A84B6C9D5E58BA4
716E4621BC440000D785D05548133E96E70C254DCDCE9E17AD9EDE7F563177AD58A1A29C2EF1194DB14762D614F1335CC1EA
8C364350F1D5832F563A752E12338D93BD848FC3D9E38A4AD77404A9024DF506DA266B0300C0F56028938177591AEFA62C9D
C982ECD808E6C21D95E7625163AB53EE09BF9679E05CF236336164BDE1554B4B7D9F4C6B94719C3A1731A986295EC2F0783D
6BEAC5228B96650D92E70CB3AAB2A17B04001784B28894459E78EA97481BB023478161DAB17305255516674C7C93BCE648CC
DA179CC7ECB63A9D45135870FE4C92AFB367243211A9025142067B541C65C554E50BA4102909005C159BB6549AE026CF7627
8D0DECC8D1D318BDC74C7E6E89B386AA795272FD54A20FA613B1D8EA6AA4A5BE21D3B8CA4667CF4868B47EC27C617A9E894D
B5A36B3E663F89C059EBEC27920000E82A0C653470CFDB4F1F94CD614DA6A2920C63E60852B29D75FAA630E227E23705175F
63F142C183820BA7137D589097A048DDE859828C7C56A9028B22094D13E4D943AC2A97D8E20200C01F501691A270129EF695
51E0CF9A4CC5A51926CC17C4A41A9C370B51550F4F25FAB42A1B5869750DB2AAED4F673BFB8E820E82912A98CD365540935A
75E55249CC4746FE4DE82401C0A5B0698BEDCE8B72865156F61C934FC8308E9F2F0C8F776255A031AB4E266C765EDF3E6FC1
4A5AC22B037F8B5CCA8EEC8446EBC7CC112465B26D77A8817B1E4BFC585DB984340BDA0000700DDA631F1CC612DCD495CBD8
942FA40AC6CD133C8B73EE6DE191550F4E276E65C19BF757E0AB9A8FC66EC86E4960477622120C1317089D7A61EAF5F30575
BE22DF034FFBDA228B85AE12005C04D2C49567BBCBD20798B1676CCA577B2804FE8B24E7EEA85526C5C9842D21A56C5BC2CD
6949DAF278326BB2F322D93061BED0D925E86BD1D46C96447FA46B3ECA30541B0000AC87A1CDB21869AC9B22CFA38D21D994
B3C844C3A859FCE46CA75FD44592E05CB21FCE221FC82A93DC3FE7049BE240D9BD182D123E7AAE655F0F61123FC2D3BE5616
4D22F495D05F0200EBA16D724DD56A2CF9335DE33E9665ED799C7EE4743E0B3680F1948D67927C136A9FB2A66A1A64553B9E
CDE5A99A5893A3DC12B3E752D1C50035FBBA08D2CC5595FE2A8D7733891E40870900AC87D073B0A44F655983086D29CBB2F6
E8B96ECC6C4161B9850579B99975E46AC63E8A154BB83443A7D447F885CE6493B155D55967AD12FF76869D1106EDDE8FED8E
0BD6B069373200007F86A10C86D633763705A5B35996358A62AE06AAA72E1636716D2CC84E6E6BF2B9E4ED15E20216E44561
C46F651F7D5C74954DF6C615124B7CA49B7FC359D951585599F2BC11787A7F8B3CA10D0000F6429A79F29CE178FA372631DB
3CEBEA0DF4A10BF6E8484A351BE6D766C2743AD1971D47141BE5D5BBC317B42AEAD9646F7215E57B50B6D447AA33D0ACEC2B
34D51B25311FE99A0E3334013D2700B075A5C08C47B5EF331CC9D03696654E2A23BDF7DA2329B3264741C5D72EA7EDD19895
CE9E913C6EEA96A0496C13D82473E6A6CA7389A8B6D1CACADEC22409C6D30728F23D6C9A42E83B018095D8A39F94CDC5523E
D7351D625FEE1A5B6DF3D78BB7EC674F10D83261EEB964BF425EBA53E7C260D58594DEBC98BA8B7D261714A9F3F8951F97C6
CE309734A154972F92C47C64E0B3D6053700B83836759EDD9F61B63B69E6B12F77F9A5E6498B84A76FAA5893232B693993E4
FBB8D8B95FC78BB57CA46D521A22D96772A9B9A6090B84D71F68D8DA63187857B1E4CF54253308431DF49F00C03AE5AFD2D6
6EC112FFAAE1AC67650663520D3F4FE38746B3CAABCCFD8273D7330F5A48B3F366A15156BDFBF94291868542B4A6D13A7BB5
78C731395B3B0DD2D48A248134C1CD2409822E1400580652FB58CA17B2AC1F6CDA22F6E58E61DA6E3ED28C9B2728A9B2B029
5F99CD7157D27F6B9055396F164A04D91B1F8E65659B5269A8ED47E50B374A70056B7D00EA5B4E60096E9AAA959485CD01BE
01C0D5604883BEF904D2FCEAF285ACCCA0524DA129DBFCF5128D8E551BC2851AEEC5D45D4975CF9D558B52B6F8DAB043D1AB
D8DAB2EE8668C7CC16641698D89A419BB6545130164FFECC8C4740470A00AC813435CB3206C832069AF1185666B0A6D13A67
B5D8EFA88C7D593B9FB2FD7EC179274DBCCEA2092ABE763BE7185B5B565A9E69FC7CE1E5BBEA36F6A26BD8238D73D3D66CA2
6DF236000058B05040998CFCABD2783765D154B6E63125DB3876AEE051B88E85B3D1BCD337B20E9394539E19C7F5926B1907
62AA59FB569A2B24966D95AEDB8991245BB3D86655A6C9737FB2475194C741770A006C582830F3E4D9EE78DA376CF56A4EB6
7B359CB8505855CFC293E37135A157D3F70BD4ADCE987881BAE544BC77A92087B5929B69BB18A01A3F5FC80E3FDB6FC21E45
31E6236D9D2F4DA8DA000070EA5E8BB61885FED2383745FE28B6FA2813888995DBA46CF25FF432A5C2DC2B69FBCA84B9CE98
F86645DD6F914B5B95F52C6E622939A631B305EC7E896091C5B6CF2CBEB2C860B900009C7FA1206B306ACE46F6BA22C92A34
8D9923B8E0CFCE6E99AF6ABE94B627B5F1853326BE0EAFF00B9BA535B379C4144888557ED8B2AD528D8E62713635D5DE92D8
8FB4B55B6077010038F142016536F2AFB72F1478B0CFC5F1EFB28762AEDDD78C9D2360EB12AED1A6BF90BAF359798033269E
2329DEFC7812EB1BDAFD30EDC819FCE81403AB970B62ECCB05295F5A642FDA000070D211D3D42ACBFC8EDD0B05CD3CDBA28D
12345963713D5E48D9F1A0E08233A6BC4A5CE4FD683CEB1B5A59B5C57389889547605E9E64D8E325C57E847E5256BC0D0000
E75B283019B817FFB15060656B36ED2E0D3DF9F742B52CAECA4BA9BB6F673BE5E1BE7251DEBAFB1EAC6F6B14C59CF7574D98
2FCC2F35B3389B1679BC3C67289EF2B9198F6C0300C0E9160A8C2DB28C8178DAD746C12DB6E6115750BE876433578A851236
077BF5CF39815401D3C638DF241AA982C091AED0DC0ACBCD53BC8407CE29D89D4D6DAD2F9A6868386B29AB04FA580070AA85
02A3A1F58C34DE4D5138AECD0947934E929E671AE1C93B7155C9EEDA7C5878F166D66133E17C0EF4902A5873EF6757687136
1B73F29A72E2426161059B8F285A9599F2BC9FB1A44FCDD853E86601C0A9160A1AF1B4AFF1F4FE26F163B6E651A5A1F69D56
4C5A24ACACB5B2BB361F155DBE9975C46473BE88BD95E2C28D0FC6B848A3CB2D368F9B274036C9F2E582FA9D68BAA12A5F48
9A9AA1A70500A78026B5DA86BDD2043765F1143677C225668F5FF9FBCF2A585FA1A1A5B76E641ED2989DCF814CB5B4644BD0
142B697185764790CCD95B2A6493F1E9461667D3A62D53148E97C6B91905B7A1B305002769B6A558D227B28C6F2DF278B6E6
51A3A5F69F534C5C282CAF61FF88E3BCAAA01EAFDAF174AEC2E02A5BD6390DD6992B45ABB7635A3D9B7D17E85B4EA11E4651
30C6A62D81FE16001C7DA1C0A65057AEC0123F56572C667136D3728D233C79C72E2B5DA14E1F155DBA9975D819DF20B42AEB
0FBE58D128AB769D06F8385C377C32EF4E309B0FC55016A1AA749624E6237DF331169F6F02007688028B22451AEF26CB1A44
E82AD99A4991945CBF1B9FB15CC41512AE50A977F3CEDCCC3A4250CEE7874AACE19D4DF2CD6C8E779D1628C1C98D7BF0E9CB
459C7A360F9746E11D3CED6B59C6408B2219BA5D0070600D2F56148CC6923FD3D6F9B1389BC191BA619378014F342E52AD37
320FDDCA3EE28C29571871FFECE34F4A6FB85433CC2E348F9E2DD876584E10AC3DFEC3D01675E50ABB0F647B8465706A0400
0ED94EEDFE8D6F48E3DCE4B93FB138B459499565DC5CC1DA9D1841322E52B3E753B607E49E72C6949B09534465E0A9842DAE
D618EF856A7F9CC2BBC36AE16A9127C8738661497F354B43A1FB0500078434D4DB4F23A67D6D14B2766FB0042737FF864F5C
282CE5585CA45A09CA763E65C79312679D6EE77193373F9EE86A8D51A5A1761C938F9D2BC8C837B1389BDA865DD2043755E9
AF84A10E7A600070286842ADA9D9D47E1AD193B5B28764EE8668874EE4DE71997707084C27BA90B233AEC659A763B558B96F
C8749D45E36A4DB2B2D63A63B968E14649AB80B5BB5F486383B278AA34F6237DEB598621DA0000701018DA224FB26F32CCFC
CEA6CE656B2E330B4D3F4FE36F3D2873A9BAE5488A2FA5ED29E0A53B69FA85EAD693099BB39A135CB05D26661A7FF2E4FB1C
C00D469AAD79348AEEE1E9FDEDE7A01529D013038083405944F2DC9FB094CFB50DBBD89AC7EA7AABE712D1ACD562173977F0
0749F5E157D2F7B62AEB9D34FD3A8BE671D1D5ABE9BFB966DB7CF84C37782CF7D0050545B1751B0CA3A9F1B60747A85A8D3A
22E88D01A0EFDB2465D0B79CB2873C2818CDD0EC7CDB2E10136B766063E70AB20ACDAE56BF0F0A2E5CCF3C68269CD85D5E6E
6BF2E6C7935CB379DA08E662806AF018EED95B2A86A5BAC0A62D55148C455D10783B04008768929A82764F8603CDB2585666
50ADA50E5F540C9FCC7B12A573C1FA3D9BB4ED5EFE59A7CE42A3BC7AF7F305CDF21AD76CA106237DEC8A72F058EEE5BB6AD6
E69177054FF9529E3DD4AACE813E1900FA10CA2A55144EC4923ED5546F6065064D66FA4AA01AF5A8E76FAB5CB07E5B150D48
1524D7873B752E54264540EEC93BB9275DB69D6A75D481730A64C697EEAA59B962C05066BBFB82988F349C8D94450A3D3300
F4514B3419B817A4F16EF2DC11B48D85BE7F2D56C63F4833688CFDB5AC6B567162DDB3F3C97E4D72A7F7189CDD92B02568B2
2BB7568D8E3A785E31684CEBE91B4A9B8D5DCAA05DE8682A42B917BF6DDCF96FA2FBBE463E972649E8A201A097B1698AB0E4
4FD91A2ED96CB14B0234BDDA7F4EE1B2557C23EBF0B58CFD34E3F43BD8F9AAE6C3316B903670E506AB37D227AE29BF1FD5BA
E3985CA9664FF82445465AED8CF15503FB55F5EF5735C0FEB36240BF9A79D33425C5D04B0340CFAB7286A128DA6A33F22AB1
8451D2F8BFAACABCD8974D8391BE1AA81E3486BBEF8C8271D5BAC6F5E23349BE51558F5890170B690EAFB8B73F6AB9AB2B79
82B9FE408D84C12A3F8CCB0A3F06BC0B6739DFFF7FBFEB8197AFFEFD38DFFF45F224A88D61DA0000E81908B3B1353AAC6ED6
94AA81FFD6D10C3983FFA3D1778DA1A9894DD944D3A82317953F8C6E3D744149D3AE5BDD49F5E1175276B026E06083AC6A7B
D8EC3ABC121A72E80BDDF029BCE9CB45B925CE7DA64618709B33E8A3D748823F84C177FF894586438D03404F6095CBEB567A
557DF37A4D2EF0BFC98E6CF245C486DD98FB04EEB5FB1A17AFF18BA9BB6E641E664D760C567D50F1D5E3711BA12D23901EF0
5C22FA712A2FE089D649237AE81BEA6BA78D7DA324F847EF5437C7137D126A1C00BA179B4AD5B066D95B1A20E7BBFF16DCBC
E6ECD92CAE344F5B2AFA651A3F345AE7E2355E25293E93E89BCD2E9780F578E5B6D05FABA525D0A2ED136D29B1653FFE9D47
EBE67D38CF097D73F1AF5CE07CF75FEF5005EDC2401AFE0CAA1B00BA11DA6A15DEBFFB4E4DCE19FAA5BEAED649F368B6304F
A2F4C326F1D0042AAFD40C957E27E7E485941D469B9E4D99B210A6B032FFBD118BA17E3B2048E6FE537B74C531B30561D17A
E78ABCDCE8BDE6354B977FBEBEE9C73D7B8A26204A020074DF88299556797CFF6E4DFEEDBF37FA7A3B63060512C2F790ECFB
51AD3E07646CDA9BFDDE34C96B4EC46F4EA87DCABEAC71958DFB221627D53F875AFE83DA26EB0A5F29B2FF15BE58458DD529
D24C59ADF5CB17BEBB536A5705CDBBFD6C6A355434007417DACAF24E69F2FEFDAA3D7E70B2C9A395799162183D5B30C29377
27580375DDC1DDBCD3E792FDB4661676A44C1B93D6F8624BD064962D837CF8A24170A46ED42CC1D089F673378E1FEC436996
972C9DC9E9A42AD8BB8DD0A8A09601A07B7A5192C4229E775615FCF8B55924748E7C316DF52DB60DBBB11F46B52EF5915637
58A1AE3BA896961C8B5B9FDC10C1D60C2A4DB22BE9FBCE266F83BA7E055C4E1EBAA040C2E0C7A9BC6357944209D9D6E68807
FB1A659C7D89CB4F4E75AB1DD08FD30955D0727034A1AB80FA05806E812608E9F3A79D570526A1C0F13385C9C95337944327
7147CDE20745E8A0965FE672FADE4BA97B08CAC6E23C36C9AB773C9D93DA1005D5FD9AC2E1DAB61E940D19CBFD791AFFF045
4563AB8D661C451BA084BCE03CDEF4684258CDBD98737EA9DFFE7BCDBB3A25CE0FFD5A4EFD87BE651B6D5340E50240B7A02A
CCEFA42AE0FCF22D4339AEA351D4B329D4D4ADC79A513305EE13B8BB4EC83119B845FD17D0407924765D993097ED1965D29B
A2D1E022D2F0A0D25F4B438B6DDB61D9D049BC21E3B8ABFDB0EC2293C5CAD0745FCA039941723C6ED3B6D05901294FBDEF6E
98E4EB79D17DC83B960BFAF76BD834068B1F8C257D6CE05D66681BD42C007C3094BE21A9EAFB774802D4304B07F48BF47417
1252475B7244E941BD1926A3AEDC538F9E2540BD9CF73EBCBE19FA873FF5BA7AC99198B50F0B2FB94266CD84F171D195ED61
B3D9BD2AF28108A5E4E91BAAD1B305DF79B44E5E24BC14A06E6CB511244351BD270F98F6958AF4C6D80D0FC6AFB8E03B7BE7
0DCFDFC64CF4D9B4E584203920BF62FCC8CA6FDE28096A3DC7EA6B5B4C926B78EA9778CA9766E953A85300F84048630D9630
B07EFDBFBF7DBB6FF5807E6903FE326DC68285977EBC19FF1C8DC2BDD96FBC09940CAB8DA969B4ED3DADF8710ADF7D3C77ED
4EA7D965DDFBDCC83A742AC1476376957D597203763A71EBF1386FA8FA7728280B139B6A5CEA23759FC01D34A675FA32D1B9
DBAA528EC568A2C99E5108E886E8B624696FBFB955026F7F9FE907C78C58E8EFB166E7ECE3A34E84844870FB59A1C47C6AC5
AC84C7DF78D47DDBCF1E07E18F1EA9FDF7DA395375B5BF87CFD6D6FA62891FCB73DCADEA6CA8500078FF6502AB445138014B
FE441A3BB466F21BDD887106F42B1EF87F2E790E587A75FDF8D50FA71D18E1B973E3B1EB424EBD85444D9B62E8DE9507148D
1EDA86C9A907CFB4F3D78BBFF768FD653AFFD00565331766856F24AE26647FE432177877F0AF73614DEB6F914B035C38C872
9790CAC847E1BAC55BA44327F150B31A3983BF6E17EE1FA42DAAB0A834144132E842ED1D35C0CE6F45B02FE5311D6DD6FE75
8260B47ABAB8D272FBB166ED2EC99855773CF7FD3CC167DD2C9F78AFB34BF7452CE62A7FF7AF5EDD609BB952347335565EAE
C59E3EAA9D3696F3FD5FAA06FE5FCE0FFF5337D7138B784E99FFE978842654AAF245D2383765C90CC2500B550900EF33D126
549A1AEFF640C9C329B3C0D8C2AD1937FC35926060BF9281FFEFD2E0153E475BAE65EF5C1738FA5244F8C213ABA61FF8D97D
7AF8A485A22317956979269D9EEED185471AF52A2463B1323C21F1F09976D956CC7D3C77C838EEAC5562D48FE90C3454E85B
A892141F885A1ECD0972C1BC37CAAA7D436604155D0533E83C1A1D1D9766DC734A31799110B5B21F46B50E1ECB9DB040B871
0F7EEEB63AE4853EB3C054DF62536B2934D97FCD65B53755F40BBA4F33CF96536C0E8DD69FB9A9425F9FEC25422DF73B0FEE
90A96953762E5C7265E2ED98E4C48A6CBFA79E37B30EFE719E54A6A496FB4A3D66F2738A4DFFA231DEBCAF89D073944513A5
B11F6938EB288B082A1100BA044319F42DA790B49665FEDDA248E9F8236930704F1CA976FF1409F28E8BF3F7FFA89BFC4B4B
48F885FB26F709FCEB0F9559DCA875813F3FAB0888E5442CBE327CF66FBBDD27D6A34E63E844FE325FECC6234D7E9959A5A1
6D04D32E12DE73FF12D331B9689F59982D8C504A442519769D904F5C20445318D44D792E155D7FA0E60AC1A1D9BBC1F5E263
711B02F3CFBA6C09544B4B7D82A68495DD0663780F146A0A69805B8F345B0FCAA67A09874EE4218530688C5D2A0C1ADD3A6C
126FD42CFED4C5C2192B4433968B50C39CB8508846F3119EF67D8CA8B5FE30BA157D187D057D718A9768D33EFCDC5DDEA1D0
73EBEF7BDCCD3F69260CF1B5A19B1F4F78F1AFB13BF79FB59FA08C4C347429A956458A3C7704EAD6744D4768120E1F0140A7
C75CDA6692864913DCF0F4FE2651E09FFE97D65655E2B1D1F2B4943FCE216AF5B4EF41D94F53F999051681B6DEF7C9B40B29
3B441AEED1B8D51B1F8EBD1D93BAFF9CC2738908B57DD4090C1EC39DB450B8653F7EEDBE3A3AC550526591CA488BE5F77984
0D5D847DB8EFB86C1DD73F66190623CD1312F9A5E6B018FDD1CBCA255BA42367F051DF32684CEBCF9EBCD5DBB107CF74EC08
11DB3BE8ADDAD3895B2FA6EEB29216D75E2D2942C220A8F81A98C40742516D629C2CABB6C4A619EE8668CEDD561DBBACDC7B
5AE17754EEBD17DFBC0F47EAFDC039C5C9EBCA8B01AABB21DA17C9868232B3484AD2F6E90195CF4FDCFA64D2FEA865BCF637
05B7738E6D7DE259C4CF7CF911F7C3B4486C5C0C50DBBAEEA2D9240996657E8F3A3703EF2A43C3162300E8D4328145918225
7F8AA7FE4DD774B8F3DFC3E4D4F4E5A209F305B89C31129A3349BE7E6133C55A416A43D4BAC05FAE64EC34102AA5864CCD31
9DBFA346C3F7B879F6E300F669C5E85634AD709FC09BB0403073A56891B764A98F74E31EDC67BF6CCB7ED99A1D18FAE7C28D
92E9CB44A3660AD087BF1F65D700E88BC326F3D0EC63E771F9A3705D758315E2A877152B69BE9CBEE770CC1A5C2F86D2A8C5
2A7C82A7FAE79C80A2E8937E87232DD91DBE68F3E349D92DF176BD6AD19E4CD8BC37C28BAB6C78F9738515E6C95EF6571512
FC3D8F151B5ACFA2CE0D4BFECC247E6C77810000C05BDBA64D5B82A77D85257FAEE1ACE9EA9739F5D69F3C79681047BF13B4
F549C98DB581232BC4055A8BFA5CF2F6D5013F25D486316DFF0C3460B6D04D5C223DDFF4244A77FEB66AC731F986DDD82A3F
6CFE7AF1ECD5E2A98B859EE85A2A42BF2FD962D709FB4E2B2EDE510747EAB20A4D7C110132E0432028DBCDAC236852C65735
436974C05335EE7C36EF74E256175F39E9D51EA78DE1AB9AD02462C38351A1A5B72D847D8BA058C3DB17B9F864C2168511FF
9799BE99F1DE8B8F9B2F28AAF8A00AD235EC45AA004FFBDA82BF802A0080B7401A1A645983B0A44F546573DE434593241399
A8771FCFBD1A6877A1CF304C5673FC9A7B3FC7D584A27F56898B7C43A6F985CEACC32A1818D1FB7A95E07AE6C13DE18B5A14
75501A2F2337604763D7EF783A57AA134269F4AC1E6018B1967F2DF3C09A7B239035AA4DCA8EBF37CA38DB9FCE467F37D98C
AF7CE5FA03CDD089BC87CF741FFC685253E38D257E2CCB1A625566405D00C06BA1CC4245FE18D45214851368F23D63C7E80C
F4EE93F25FA6F339F5BFBFB3AB96966E7C38E641C105BB6CA0C9E7E57757050CBF90B203D38968068E06F40106ABEE52DAEE
7D118B79AA26288D3F63268CF7F2CF6C7830BA809706A5D1138B03A8E123D1E59F731C4D19CEA76C7FF90556B928CF3764FA
FD820BA8AF78E58B1DEF0EB61E90E18A6E086C4A136AFB59C57837459E874D5304F50200AF4A022BA62A9B2D4D7093E70C27
CD1FE406562825262D122E582FF96341008D3EA8A55F48DD49B5CB008D5979257DDFCA3BC31E145ED05934A00D7A13995E72
3CDEFBE08B1530177EFBC895D91CEFFD707C40EE29789BD08D8B03A8B10BD42DD7330FADB93BE254828F58FB2F5D4D3E3715
75144FCBEFBC46AA5918EF7DF8D8B98282727377A58734B5A84AA6DB9D18144D21F455504100F04FD96C53A82B57D85D1364
0FB6698B3FF46E745B72B6D17D02EFC21DF54B8391F4B7C825075FAC3411BFAF0A0AD5AD8763D6A0C94264D503A3CD403314
54444FD320ABDAFD7CC1E9445FBD05C246BF1B34873D1AB77E5BE88C3AAC124AE383FA04862669A2515E7D2ED96FCDBD1117
537761BA573D0664B724F886CC78C179FCDA3B3C7CA6FBC9937F2F54DBBD09230D7548122061A02A9B431AEBA1A600A0AD7D
254D5BEB8324812CF37BAB2AAB5BEEA9D5FFFE1EA1A6F19FEE0475D6DF0F2668CCCA3FFE588B95EF095FB436F09798EA6033
61A2401BF4D8342DBD317A4BD0E4BB796740807561D4A0C9A8AA4748BBDECB3BABB768613F4C57A168CA60D5653527EC7C36
7FCDDD1177724E288DB23F7F2CA329D62F6C5642DDEBE3140824C4FCF59295DBA40271F71F3AB6E92A148513A5B16EAA8A45
A40976DE02200934DA869DED92E03B8B3CB11BEFDC22203C66F057F9612FFFD144186F671FDB1AECA932FD4BCF50212EDCF1
74EEBAFB1E481B18AD7AD49340D574233A8BC63FE7C4A64713D21A61C7F5FB20D4700FBE58B5F1C15834BA219D00DAA0332A
94A249A94E18527A13B5EB0D0FC6BCA87A8C64FF6B3F9CD210E9173A33A93EFC4D773BEFAFFEC9931F9D62E8A1D4DA34858A
82314818A82B9691262E541FE0BA9280D4EA1AF6DB2541C6B7662CB27B6F6EB532F79FDA9D8DBC48FE97B66C268C77724FF9
044F951BB057BE52212AD8FD7CE1DA7BBF3C2FBFABB5A8FEBCDD08780F6AB0B25DCFE6EF8B582254B742697CC83097D51C8F
94D59EF04535D2329AA1401BBC16D46C6D94B5162B3F97E4B7F2CED0BD115E2582ECB71C684AAC7BE61B323DB93EE24D1FA8
6EB04E5F26DA7154AED2F4E064C1AACC50E48DB40B83AA951FB8B10A009C5612E8744D87ED3E8D33069A243DE2001F9393B3
5689D1F54AF789A60C8105E7B6044DC674AF719E538757EC8F5ABED27FE8BDBC33220D97A40906B623BE175A8BFA51D1E50D
0F4607975C835733DD33A924AD4FCB0290703D1ABBBE1EAF44D361D82BDB01323082B249B482F0CAC0AD4F3CD7068E0CC83D
F54EEF58290D51DB427F4DAA0B7FCB678E5D517AFCCA4FCB35F574162CF20479EE4FA84B5457AD026100B8E22A019204F16E
B2F4FE46A17F4F7514545B5C9AD17D02F7E1F357CF175B48331AB0D0CC4BA2E5BFF6BB220DEF7CCAF69577DC4F256CA91217
A10E8782A5834ECF6B098A28E465F885CD443335F048D0ED18ACFAC0FCF36BEE8E381CB3A6525C68374E57155D3483326F53
991499CD7107A357AE0A187E287A758920BB337B57B25B12FC426776383379139C7AEBB4A5A23DA7E45A7D6FA82F0BFE429E
33DC2E0C2A969226585B035CA62113BFBF38C0D3FB1B783D1B2A4EAEA496FA483D9788E83FB5692B69417358EF47E3D1FCE2
4D5FD75BB58F8A2EADBFEFE113EC19531D8C3A1FD403C3D2C11BD580FD2438D52CAF4533D9CD8F276534C54099F41C269B21
B8E43A32CE1D4FE7A437C5205B252982710D0FBA48A2DB48ABDAACCC6D4D3A99B019890134E58FAD7E820AA193776857ADB3
222B1FBCFD63676EAA3C7EE5A7E4187B2D6B1659B43CB7431878C1E643C03524815ADBB0BB431218F9D77BFC71745B729671
D8246E58F46BDC22B5AF185C41E397DC207DDB60C73005BCB47D918B57F8BB5F49DF572529428A024D87A1365FD603A8A36E
55D65F48D9B92ED02330FFDC9B767601DD0B52021DCBE0F6730AF9679124B351569266A171A2668834396A7A48C6C7D73E3D
12B36EA5FFD0AD4F3C9F57DC7BEDB182B760DFEBF27C4148E9CDB77FAC854FCC5923DE7658A654F7EA524CFBAB84111DC715
09032CB501EC9604CADF0F21A60FE8B91707AF802BC8C59B2573D78ADF34E1F2CF3DE9133C556356BDF3563283F45EDE99F5
F747793F1CF7A4E4BA40DD8A262C2EBE2991693F09DE2CAF3997BC6D75C08F17527740A8A33EA155D970296DF7EABB3FF985
CD8CAA7A28D1F2D1006ADF15E3CCAB0728F1280B28237AABB60EAF7C527263FBD3D9489CEF8B5C125F13A63129DFE39E5C65
C3EEF0859D89DC7D2758F3CB747E44A2BEF7336E51A428F27E6977703499D055807903EC94043699BA7279FB89838126D1FD
5E9C4C319189866193790919AF5F06441DCEE5B4BDA82F450AA193F7AC12179D4CD8B2D2DF7DD7B305D19C203479B1F7C02E
B37EDB01459366C2C891969C88DFB4FAEE8F17537783BBC23E87A0899CD6A483D1ABD0D089A6C3CF2AEE72958D16C28CECD3
59CEDBD20C4D50360B6936DAF4AD8A7AA4708EC56D4082736DE0C88BA9BB0A79E936CAF6DE375718B1832F565ECBD8FF4E8F
917A03BDD257BA6CAB942FEA9B85179B26DF7E5C31DE4D91EF61D314806D036C1B412C6255E92C690292047F3763CF7BF9E9
628CF45C2CEA88A5F85A3466E5A9041FD45D7469D7169AC564B7241E8E5EBDE2B6FBEEE70B9F95DF45F335D49BD9282B8B77
867774DA328334BA3A68EB13CF7581231F155D5299E460E40E8595B220E33C11BF7955C0B08D0FC75CCF3C98DD9280EB2566
C2F4FB1A82C31C6C6CB728FB820012994AA3BC5C94FFA4E4E6D1D8756BEEFEB42A60F8F178EFF8DA508501EF860938693E97
EC87AECE2C0CA249C4A85902FFA0BE74C249E8AA944593ED7EE0B3DDADCA54B06A802D30A4A9595130CE2E09B206591429BD
9F02B385B9F548F3B327BFEA1F21935E3789C00F45AF3E9DB8F57D1A2F65CD6A8E3B1EE7BDF2CEB04D8F27DECD3B5D252944
331DD40B39FB12EECB5DB78DB46A2DEA3251DE85941D6BEE8DD81DBEB0900FE1DE9C009EB231ACCC7F6FB8D70A7F779FA0A9
57D2F7C5D58436C96BF456ADC9664056DA71C4A6177442FBFE13AA637B00D200E8E932BDA44490FDACE2EEC5D45DBE4FA62F
BB35D8FBD1F8CB697B92EAC3BBFD55D49DDC13C7E3368ADF70ECE8150E9C534CF1127D60C4E40FC71E2BA1742696F0319EDE
DF8C85A3DE14EC197072454011BA4A2474B1C48FE539436DDA3E0B0D56D76C1B3E9977E492E22D9F91EA84A8E7BC9777E643
1E542D29BD917978C38331A807465D5064E58366799D89303AA94220691275E0686E552CC8BA94B667EDBD5F363C18FDB0F0
12069B079C937ABC2AACECF689F84D1D26EA17360B69BC90D29BE98DD17578A5C228436A160DD61D4B0A36CA8A4670D2EE15
81B247167A97F5A2CF20F58804063275F445F47564F668F437DA0C06AB4E6356B62A1B725A129F96DFB99671607FD4F27581
1E4B6F7EBF25680A9ABF875706564B4BD1577A28E3E8A1E88935D2B2CE7C982B24E6AE15EF3AD14B0712DF21C86D7275C512
2CE9133CE573A3F01E4399C18C0167550434615565C9D2FB237B56E47990C6BE0C9BAB5051EB77E193160909F26D3D5B8BA2
CE37647A62DDB30F7FA2DC208DE63C3E14BD7AE59DA1EB02479E4BDE1E53FDA4595E6BB0E9503F897A3F877580D0A104D0E8
C05735C7D6849E4EDC8AC4009AC13D2EBA2CD4803B56F64050441D561159F510CDCD773D5FB0EACE30AFAB03D6DC1DB1E3E9
DCE371DE57D2F7DDCB3FFBAC3C20A92EBC909781C653D43A849A564C2F529B143A8B1A5D5AB34A66908AB43C894EC05735D5
606525829CCCE6F8F8DAB0E71577EF179CBF98BAEB60F44ADFD019E8B68BAF7FBBF2CEB05DCF16A03F8696DDCE694D7ACBA1
E0EE25B3396E4F8417FAD9C9CF3F8BD57BCC148446EB1CA62FB5686A3663295F48133FD6379FA00915582FE0848B04269324
144BF90C4BFE4C59EC49DB947D9B1E926422120CC327F31233DF71F4B85498B3E9F1844A7161373E1D0DAFE11581C7E33622
79803AC6032F5604169C437D14FA3B9A43A1A999D5BE8A4BF449D8A0F66DDE761960228C7A8B162529A9FEF9A5D4DD1B1F8C
5976EB87BDE15ECF2BEE49B52230691701D78BCB45F949F5E1A1A5B76E661D3E95E0B33F6A995FE84C240BD7DCFB7985BFFB
921BDF2DBAF2F5C2CB7F43D7A2CB5F2DBE3670F9AD21CB6E0D4286BDF1E15824AAF7452C46A68E86FE80DC531195F7D1E85F
8F5775F5F0603782F4CCEEE70B9F97DFEBFC57F69C524C5F26AAA8B53A54D5E81AF7E3695F4BE3DDB4359B28B3006C157022
68526B683D2B4D70C353BE54977B31B4CD1152C5131163E608FC8EBCA377621826A32966D3A309224D8F381DD55934392D89
77724EA0D9D94AFFA1CB6F0F41B3A7ABE9BF45563E28E4A70BD42DE80306AB1E8DD1D6F68D8B244D50DDE1F4BEDDC510DDE1
25DE429A4DF6755DBDC6AC6C96D72219702BFBE8DE08AF9577507A06EF8B5C125C7C1DCD22C19B2EE0ECA84CF2A3B1EBFC73
8E77BE054965E4226FC9AE13729DC1E1ECDF28B825CBFAC1EECAA0E45742CF81FA059C4312D8649AEA8D763F45695F6BEB77
394EC2D45A6AE771F984F902FA5D6D1D4D9C434B6F6F0BFDD56833F474AAF4565DA930173DEE74A2AF6FE88C15FE4396DEFC
6E4BD094C3D1ABAF651C082ABE1A5B1392CF4D6D9071D0240E090674E9AD5A835587C674943C93FD32A2CB4C18DBB76FD977
70B5FF97FD421F43D37FF41534536B51D4970A7392EB239E94DCB89AB1FF50F4EA4D8F272EBDF9FDAA80E168327837EF7446
532C6E90800103AC0129816B9907CF24F976E98C4C6A8E69FC7CE183A73AC7CC9445166BF771146F3F98605124B7418411C0
A11B214D1A1B148513DA4F200E44B2D6A15247524C4482FEC7A9BCFCD2776FD7D19A5597D3F61E8BDBD0FBE92468A255519F
D91C1756EA7F33EBE8F178EFED4F67AF0BF4587CFD5BAF2BDFACF0775F7FDF636BB0E78EA77376872F3AF862E5C117AB0E47
AF3911BF09A516FDBE3F6AF9AEE70B763E9B8754CDA64713D6DCFB05C90CAF6B035707FCE8FB64DAA1E895486C8495DDCE6A
89E7A91A21B803C062222AEF1F8E59538775CD0BD0A500B5E712517691E36EEA237495CA624F2CF1633CE5CBF6FD8706A86B
C0211581D5A24C4762004BFCAB3C67689F9C407C270DADB69F3C79E76E776AAF8E44CBDF1BE1F5A4E4BA4365C14418650649
8BA2AE5A5A5226CCCBE326E7B626A1213EB93E22AD311AFD5EC04BAF14177224C50DB22AA1A6556D92C3D00FB8202582ECFD
51CB521ABA16A29DA6DBBCF7E176E7456287F61A4D136A35678D7DFF61829BAE610F6581553EC0C14C94D41985FE58D22758
F2E748C4521607DD9C86C949D4DE176CE86C0BAA9696FA044F2D13E5411503801381E94487A2573F28BCD0D52FD6355B672C
171DBFA22449273844AC6F3E8A2662D2783765C9AFED8E9119A87AC0111609904CD570D6B76F24F84A53B3B98D715CCB349A
E84B016A8F997C31D6A9E933CDD0A90D515B82A6A88CE0BB0F009C6492C25097D2F69C4DF2D559BAEC993021C33866B6E071
B8CE59326B963E93E7FED8FED2F65B3316C1D0163000A02F05014D10DA7279DE880E57C646E11DC74F7376A169D8A437C644
F83306ABFE5EDED9FD91CBA0BA01C02988AC7A78307A55B5B4F43DBE7BFD817AD222616AAE3345FC248DF5AAD25958D2A758
E25F754D47282B063600F48D20277526F1432CF90B2CE91379EECF366DB15324BB99478C9A25B814A0EEFC5730BDF850F4EA
BB79A7A0D201C0C1A9C32AF6472D7FC179F47E5FDF7B4A3E6B95B8BAC1EA5CB966688BB66E3B9EFE4DFBDB841984AEAC0D8E
1503BD6B829445A4E1ACB3BF3548FD5255368F26D4CE9276094E2ED9225DBFAB6BC1566AA4653EC1538AF99950F900E0B098
08E3E9449FAB19FBDF6F87ADD1C4ACD981ADDF8549654EB941D78C3DB31F5A4C7093A5F737891FA3891B9804D01B8A80325B
55D9F26CF78EE38786D6D3CE957EB596DE7D5231D9AB6BFB21499A8CAD0941C2C060858606000E4A70F1B5A3B1EB79AAF774
AECE151273D6880F5D50D86CCEBA6D8F32F3D5158BF1F6B3099AEA0DA4B119B620023DBB466053E89A8F63891F63C99FC9F3
46DAD4CEB739DF6265EE86683D7EE5F3BA18365D6D525C4EDB7B26C917EC00001C903261DEC1172B521A22DEFF0E1CCBE445
C21B0FD5CE5E1406DE2559D6A0DF3D1DC9A219CA04E601F48020B010BA327B40E476A7851ACE5A9A323A695E52734DC326F3
320ABADC525A15F5DBC36627D587833D00804361220CA713B7DEC83AFC213749CE324EF6128644B1613DD0A62D52164DC592
3E95267EACADDF4E9AB9B0680074EB1281D2C0BB8C257F6ADF5898ED6E968639757EAA1B6C4327F2DEEFF0516653ECE6C713
E506D8E50B000E4468D9EDA3B1EBB9CA860FBAC90B9DE712516CAA911D65C2D0666DFD4E59BADDA1813C6758FBA201784104
3ED8AE2833A12B55144EEC7047A02E5FC882632FCD3CDBB8B9C23337DF271AA9CEA2F1CF397E3C6E23D80600380835D2B203
2F56C6D63CF9C0FBDC7CA899B6549459C0AAF5768B22495130D6BE6890E0A6ADD962DF6900A11380F7159AB44DAE6F3963F7
5898F4892C6BB049749F1D39E38B89992BC5DB8FBDA763A25665FDCE67F392EB23C04600A0CF2169F242EACECB697BC90FF6
EC7DEABA72C672517125DB1C0131944157BF4B96F16D876319FBF1044205960374CD8A48BD5599DE11990B4FFB5A5DE14559
A5ACC95DC7E1C495DBDE73D183619876878793B51635980A00F42DF1B561C7E236D648CB3EFC56472E2967AF1657D5595959
50565586B2683296F239EAD555A5736DDA5286B682FD009D5821B091669EB676335295F6830639C34C926096E551A6A0D6EE
C4E7AC11BFF71DD426C5B58CFDE792FDC06000A02FDBB25E723CDE3BA8E45AB7DC6DDF19C5FC759286161B6BBB7786D2B79E
91670DB6C75B4CFE4CDF748832F3E18502F0164140DB144661009EF635B21959FA006DAD0F4DEAD99751858ADAB21F9FB1FC
834238D5E395DB427E2D166481E100405F71BFE0C2E9441F8511EF96BBED3A215FB841D2CCB3B1BBD048439DBAC20B4FB33B
4294650F31499EA09E1F4E2800AF40933AAB32435130B6C35DA1A270A24D95CDD6CCAA3494DF11D9142FE187DC84A06C9195
0FB63F9DC3406B0280BEA0525C742466ED87382878856D47644B7CA45DF564E2A458F028FB2EC4F6170ACA624FAB2A9361E3
1C10789F1502CA44186AD4552B7E7F6590ED6E14DC647796353A6AEF69F9A485C20FBC8F58CB3F16B73EB4F416581100F476
C7C53097D3F75E4ADB4375DF02B8EF21D9922D52AE90709532A42DF6170AD9EE58D227765F889CB5849EC35066B02ED76D56
B485B208754D073B826DC932066A6AB652F6A52496A3D5D3872F2A27CC1792D4874EF3735B93B7044D921BA4604E00D09B64
36C71D8BDD5022E8CE254DBFA3B2851B25CD3CC2A54A92B288B4359B6599DFB66F36F85C57BF8334D4833670493D203170CF
CBD2BE4112114FFF4655360FA94417C97E872A183F5F60FD606FE71AB3F276F6D173C9DBC1A800A0D73013A6B349DB6E651F
EDDEDB6E3F2A9BB756CCE2DD866FC1A62D5695CDC5D3FBDB4784D42F758DFB4863131A29C0D85C400F58292B66E45D95650C
B4D77EDA578AFC311659AC4B15025205072F28267EF01B840E6AA465BE21D36BB172B02E00E81D626B9E1C8FF7AE975575EF
6DF79D56CC5C25AE6974DDF37A56658AB26406DE315B4CFD4ADF7C8434B5C00146F6EA011B65C58D027F59E6F7ED5B0AFFA6
C8F73089835CB028D45A6AC731F9A445DDA30A4C84F149C9F5832F56828D01402FA032C94FC66F7E5070A1DBEF7CFCAA72FA
725129C7D527C81679BCB2D8134FFBDA3E52A4F7D7B79E26CD5C78A7C0B6F5011B6E12DD97670FE9386220CFFB85358E0ADF
03A586F23D24F35C2CEAAE1BB628EAF6842FCC6E49006303809E26A2F2FEA9441F81BAA5DBEF7C3550EDB9449455081106DB
B5011EA52C9A82A77D653FC0983140DF74883434B44752805357CEAC07281365111BF9D7E539EEBFEB81DC1146E16D172F16
5C41ADDD812D582FE9AE1B1214F182F378C7D339607200D0A3288CF8A9842D8F8BAEF4C4CD0343B553BC4489994628E73F30
4B9F2A0A26D8B541821B9EF2B9A666934D5B4A131A34DB84C2712E39C0907AD2D8A46B3A2CEB5805427A2067B8817F15DC58
21C418E9E52D59B713EFC67BA299CB9198B5B1352150BC00D0930B0581ED0B05AD3D71F3C844FDC4058267B1706CFF4FDA00
8FB4BF5348EF8F257E8CE481AA7CA155994ADBE40C4340E138FEEA004DA89196D3D46CEE089265DF4F98E76114F8333454DF
EFF045C48CE5E29DC7E5DD784F9AA1E26B43B785CE84E205809E5A283060A712B604155DEDA1FBE79698C7CF13FA0769A0A8
5F0B5202AAB279B28C6FED834B9C9BB2609C49F284B288180ADEB938A41CA02DB44D669127A15AB3FB234AFCAB2CBDBFA270
A2490AB3D75769E4DA46CD125CB8D3CDB18DD0FCE558EC86E8EA20286100E891B97CD5C3D3895B853DB350806868B5CD5821
3A715549D3F0EAFC8D10BA4A6DCD1679F6102CE50BA40D64997FD7351E20F4556842DAC690503E8EA0066852479AB806DE25
79EE4FA88EB0E4CF909653572CB12AD3A1785E4B458DD57D3C37E485AE7B6F4B337442ED53BF30582E0080EE4767D19C4DF6
0D2C38D7738F506BE9E5BE529FFDB842056F5ADF0165C5F5ADE714791E1D7173A4891FABCAE69AB170CA226428A37DD9DA62
B12A1566A9C42C165B30A94DA36648D00CDD31E6932469329206036531B731FF225FED3B09ADB84591AAA95E879440FBF1D2
BFC9B3DDB50DBB49632314DD1BC76EBA2D2EDD387C0A2FB7A4FB0FDAA059CCF1786F582E00806E27B1EED9A9449F0659CF3A
5BDB714C3E6F9DB8BED90605DE492CB21855F97C59E67758F2E7ED4B0703B575BB0CAD0582EB972B3D3D2ABEFBAFAA81FF87
33E493FA25F3B0A808AB5C0E25F67E5869ABC9A0368B45A8181B4E1C6ADEBB9D7FF5A22A3FD7AA90D384912694A4B14EDF72
CA7EB2C0BE38F0A92C63A0B268B249741FDC4DBC139399B9765F336A165F2AEB7EE14AD1546C4DC8F630388C0000DD898DB2
5D4CDDD5EDCE0CFFCC957BEAA95EC2CC027851DE35486393BEE5B822DFBE74C0BDF09F9C21FDAABEE957D5BF5FD580F6AB7F
FBF54DBF7AAFD9862698B4BE0F25B5C9517B977006FDEFEF05FB8F22AD99F0A3E4F92145FE0C2CD1CDEEB33AED2BFBE240FD
769BB6180AAD9328D5D4D683B2694B453D74FF5665C3A1E8D5194DB150D400D05D6435279C4CD8522ECAEFE907BD48364C98
2F087D01C710DE1369C449CEE0BFFC530FBC727DD3AFD6738C91C78382EA1284565BB76E59CD6B4BB55D1E346EFF773C6DA0
AA6C81197B06BB3BBA8A5042CC5E2DDEB40FEFA1FBDB286B78C5BD7D914BA0A801A0BBB89175F852DA9E5E78504DA3D57389
F0E43525EC377C0F2893A97AEC4F6F9404FF10062D87F6105A1C8AAB933034DDB46DCB3B4A75E0FF9525874159BD1FB54DB6
9FA7F1AF3FE8C1C3470DB2AADDE10BCB457950DA00D00D6D16AB389DB835A521B2179E65B1321B76E32BFD30347D8092EF2A
821B57393FFCE51DE317BABEED274BDD6A53675216114DA82106D39BF5808D26B5EAB29C7717E9807ED563863114EC92EDBA
94A5DAA2920CC327F3F2CB7AD0A7B7D1A60F2EB97E3261331438007C38C1C5D7CF24F96ACDAADE79DCE91BAA295EC2BC1270
FBDF65EA16CCA81AF8EEF10B8D71CD87FE4312FD912C6380BA628951E04FE8AB292B864640340EBABA1260089AD4D3361969
6C36498235D51BEBBCFEFBDD45DA7EA90AF2C108BB8A56471F38A798B04068B6F4ECFA609930CF376406A617439903C087A0
306267937C91CCEEB527C6A61AC6CC163C78A683C2EF2AD5A3867466566B7F0FEEF7139E3A5C9639104FF9429AE0268D7793
E70CD7547B9BC40F09432D6595D284A63D1E934BBCC861682B52449415274D2D66ECB9B67E87A2608CBD5812DCB094CF39BF
744A12203DC6BB70068CB0AB88A4E4ECD5E2D57E584F3F486E90DEC83A14987F0ECA1C003E84A4FAF03349BE4DF29A5E7BA2
504ACC5F2FD97E54A6D5C17A6C1755C1E8CEA9826FFA714F9C248D66425F69E05D5695CD9367BBDBDD29277FFEBB42C87657
57AD32F0AEDA3479945940DBE434A963D3093B86B631F6050105651113BA72A3F0AEA6668B226F6447F6B1E44FF1B4AF6559
3F288B3DF5CDC7AB877DDA4955D0BC6F1718615729A9B20C9DC8BB1FA6ED85676534C5F886CCA0218608007C00D7330F5E49
DFDBCB0FDD7F56316DA9A88C03EFBBBB46ED8CF19D1ABFBEE9277E1848DBFEF9B280A14C16659ABEE594B264269204B2F401
78CA1776E748B16E78CA978AC289DA5A5FA330C0A629A0CC7C34A7B6EF46A08C4E138581215106694243D964944584648049
F258DBB04759324396DE5F1AD7BE2690FC399EFE0D52028AA2C9BAC6FD66FC054D287FD75A1E833AA90AF817611EDA354C66
FADA7DF5C819FC56416FD8124FD5743C6E53727D24943C00BC1FF5B2AA73C97EBDB3CFF06522120CA3660A1E3D8797085DA3
F5E451CE77FFD999214C57FBC6C51F8632DBD4B946FE554DD52A79DECFB2CCEFECEE13933FC3D03CDA2E123E57E48F54572E
D3371F334BC36CDA52D2CCA32C92F6F5048D5D2AD0D6BE8BE748DB1701900020B5F67500AB94320B083DC78C47195ACF6A38
EB1485E3ECD1277F97019FA2DF6519DFCA7386A9CB17195A4F5B154976AFD17FA271FD8A4EAA0259421C18619790CAECA112
176E94F4CEE3ACA4F979F9DDC3316BA1E401E0FD785E71F74C92AFDC80F5F2732538B968A364CB7E5CAE8497085DC0D0DA82
26B69CFEEFD854D0B07A894DA5ECE43D4933DF2C8BD1B79C54572EB58B84AC1FEC2B09A95F62499F48DB75823DE84FD62065
F1344D8DB7BEE5844974DFA24822F435ED6A4164DFA2800403A1B26F6524F568C866688B7D4F2343B65F483FD06FDEBDC0D8
2FF499F60FDB477CDA6ABF0369B08FFB84BA7DE8C7EC4F31F34943BD55996E9204195ACF69EB7C55A5B3D1706F773B1CD7FE
5220E9AF78CA1778DA3748E4C87386ABCAE6EB1A0F98A44F0943DD3B4B409694D0993D9C9CA15F5266D822DB358A2BCD4327
72EF04F75E40B40A51815FE84C4C2782C20780AE62214C17537705E49EEC93A79FB9A91A374F909E0F4E0EBB06FFCAC5AAC1
6E1C3448BD411270867EA1ABA97EEFFBA351D8AACA3689EE691BF6AACAE6CA7347C8B206A319379EF6B5FDA543D2A7F6F70E
F1482D7C840662A41FE4B93F2A8B3DD5954BB4359B750DFB0CAD678C82DB26F12333166EC623907E404339BAA14D5B6A5F76
303692A616D2D84468CB096D994D9D6355A55B9429663C125D2671905170C7C03D8F06746DED5675D50A65C9AFF2BC5FEC41
24D1E81FDB3EFA272001F00996F279FB3AC0402457EC1AA06486FD0D88E086559184B4CA7BE4BA76D694771DF6FC77F18340
30BF2EA137D227AE2A47CF1288A4BDE7F709D78BAF671E0C2ABE06E50F005DA55498733EC5AF90D73751DEF24ACD48159CB9
09F113BB4C9CCFDCECEFFEFB5561D0BF9F3D20C2F0AF145919AFC4F4F94068524FE8ABCDB21803FFBAAE719F86B31AC90034
16CBB387C832BF4743339EFE0D9EFA373CE54BFB6B08BB6CF82BD6BEBDDF2E1ED015F78F2BD6AE25FE71BDF4F78E8F25B48F
F8891FDB07FDE4CFEC13FFD42FED737F34F4A3E97FD66079CE5065D12475E5326DFD4E03F722521D4866FCB137E003318B44
D51EDFBFB262D051C2D5E8FAEEFF35AC590A86D7555A05C4142FE1C63DBDEC508B49A87BBAE329844500802E13527AEB6CF2
36BD55D3274F475A60DB11D98CE5A28A1AD873D8350EA5ACD8B66675F090CF4B06B97106FD2FE787BF70067DC471FFBC61DD
7223B7B5D792C1D056CA2C2474E51645AA49126A14FA1B5ACFEA1AF7A369BEA66AA5BA62A9AA7496B2649AA270BCFDCA1F65
171239C3DA7FDA7F51E47B280AC7A10B690C55C9AFEA8AC5EAAA159A9ACD4878E85B4E1905374D92208B3CD1A62D264DDCF6
30913D8B49C0AF9B399933F8E3AA6FFFAD4365550EE85736F0FFA60EFCAFECB57E60755D85249967B1FAA113B96979BDBD1E
588F57EE8F5A562D2D815A00802E4C8E08D3A5B4DD77F34EF7611AE2D38DA366096E3DD24075741EA509DB123465E48A63F3
CF2ECA8E7B2179122A0CB88D453CB76052289C0F078F8B69F45E53337554F52FDFC68EFC2262E50CCF559EAB76EB09087DD0
453AF60ECD5F2FE9FD47AB4C72D4B3F9E71C875A0080CE5326CA3B9FB2BD8097D687693098E80DBBF1B96BC5B58D108BB6B3
3CAFF49F7F78CF44BF19B7D2AF605A23BC7DE939229A42525B73BDAE7D3F686C1D9AF64281741E9A6E4BCC34BA4FE08646F7
4DB9A53644F985CD828A0080CE1356E67F36699BCED2C7F3F4F078C3084FFECD87B05CD0597C83E68DF0F29F7178709534CF
4674EF0E02E05F87368649E424EF7AB670B6DFA35FA6F3A532582FE82CA8AC166F964E5F2EB2DAFAC640EBF08A0351CBABA5
A5501700D019289ABC9AFEDBADECA37D9E12A389F6DE8BFFBA42545103CB05EF2687D3BCE0EC842547AF6CBC374FAAE34381
F434CD52717A53CC8E9015233C05FF7F7BF7E1D554B6EF01FC4FBAEF3EA7BE7B67AE333AB61975C6366203C58A34054401C5
868880A0A288D27B27D45042EF35F41648213D2459E93DE16D742EA35CF5AA4349C2F7B39CB530C4CCC9EFEC73F66FEFB30B
29A8461392B08FEA2828AF53FFFD87D9943CC57A1D8354237ED11A9EB04E13AC001CCE846828BA2EA06EA2C41E0EA6AA41F3
CFDDACFB4FA526336EB91F62D02D1C0B0FDEE1131C98E399DCF940A9972126AB4D6FB04AB562B7844D8FD3B87FFB7E66BDFA
C31D0B87673A7086FBDB29AED9B29E5774F578E165CC4400F838D4B1FC073597D832869D1C4F508464CB014E7D1BD62EF890
3B318A5D41DF67D068D70ADC4704BD1B6463A3F565B1DAD40665549D7FD558F6495FE157DB99130C23C2F2A1DCD5687B99AD
D8B4955956AB5EDF2319E4750716BAD8CF5D0EC09EBD68BD175DE76F3FC733346ED87584733148289460A9C377A3500DDF1F
C93F78F7C8A8B06F4A3C6C342F3E70C1A082D546226CB559A7C4A3FEF9077942F30FFBD97BDDB8722536DF79AF9109C3777B
D867AF08D7FD48B80AD6A386A0227A0A4E0AC0874934A218DAD59CBE78BB3A2AD2BEF8721B333E5D8E9AEE3F0D4F187EDC27
DEE673A166B45CAE9F7FB3CE82D5CE0A1616871D5ACE25FF8BAF6237756ABFF86936F09E78BD06D1D939A9DCE27F47FCDD5E
1629B1EBDF6B61D6170E24DE2A3F8BF302F0610373EDD17501EBB5A4E1FB687556AF10D1D6839CC60E3C47780B5F64DE7594
F7E3E1FECB693E62CD5BABBB5BD1665D1366AB39BFFFE5EBF9EF7129B2BF7D3F9394ABB0589018BC1D258BADA042F53F9B67
635EDACB88970E26CD27FB1793150F7D003EA46C3833B2E6B248C5B3B703EB1FD1EF70E19CF613B0B9269CA6D7942AEB312F
FEF77B842199B103BC76B3F5ADC8908A09DD056B635E2B727FF97F46CB6213D83358F83F9B67A88D1A04FF4D43E3867FFDC6
3EEE6D47AB694D8A86EF545EE8E3B4E1EC007C40426B985D0D2A78535689EACBEDCC5B31F31A2D9AC10B3ABDCD3358F4D576
56519532BCEAD29C7C76D91BD057B096AE16BAD44D2E4EDBD11B6C472EF2BFDEC9EAECC762DD7FE08BCCE703859BF7B1EDE1
D9C112B19AFFBCE54E7A572C4E10C0FBC875F38F1A82327B9ED8ED11DE7D24FDE2276652AE628337C4484AE0775BFCC53666
6AAE9E3A42C9ED7FAAD0499115AC179BCDD6C1AC5F5A2E6F5E6ED9EDCAFDBF9F59765509AE17B5D61A9B28DBB46536BD5069
570766B2988AE9A9374A4FE11C01BCCFA86020A6FE6A33A3CA6E8F502AB7785C137EB79745A9516FD8D3445AA357C3175382
C7C9729B75219872B493D960B5599115AC17127CB3D5ECFA7CD3CCFCC4EB576639A69F7EE77CFF2B7B83270666F3E27082FF
FD61F67AA4C40E0FAF7DB6CE27FB178B15EB5202BC5BFD6459549D3F433266CF0739CD341DBEC0DBE1B241471E9296174909
BEDCC68C8C975ACC0B63A23E9FECDD4CE9D4BB1AB028D16BD557B0186C5B6AE7C3FB549FA5172718C6AD87389BF7B14736EA
461E24284D1DBA6F77B1DCFD04F6B9F0E3B8907EABFC1CF64F04789FECDE670F6A2EE94D3A3B3FCEBE21FDAEA373FB4FF3BA
E91BEBD1AD4C61F10A112EA604CFA5AFA7BFDDAE3897DDFB546B54232B58DFAC60617156AFF0F8B3BF4B34A2A5D747A74862
C0FE613F7B6054BF014F47D7806ECB41CEC1B33C89D44E571AE129D8714DA194A134946180777AD6742B9676D531BA35DAB5
E4667BE81CAF7768A324067C91F9E425C1573B98CF33E4E657EB3FCFC9673DD2B74E8947962A266405EBEB212D70D9062223
9386ED2E73FFF899D9D2ADB56CA455B8C885B9DD8543B2F7A959FB9D34A431AA72FAE2631CE4A607B0D6ED50ADE449C3F59C
DE678E72C075AD8B89016989740D38796240EAF7E109E33E77EEB7BB58B965AAA5010351B5FE2FDBEEBFDEDAD262B3D8B0CA
F17A63C9A65DE337690C6F6D88C09A33FD7A92FBE53666054DA3D36F88734452829D2E9CED87393D83F67E6DD64C14F9E71F
40D105F84FD3E2D147F54176B229D247A2B569B71EE2EC71E5D6B76B9DB5696C34DA6A9A17F39FCDFBD8752D7F0EA5E02BD8
17D2B68C0AFA5FFFD56431627982F5CCDCFE9D928596B967762F9FC52391595C2EF236FD38FB2C4D2E533873970129839DFD
BA9D4738DB0E73BA0674F67FC05DACC6CB79FBF4262C8F06B05C27AB2186767588D7ED5887DDD6ABDB7D62EEC7FDECC22A95
C9E9B6B2952BAD71A9B26F76B2F6B973C7A7DF5A842DAAD6FF456BF8EB8E8285C515F64CE82BB083ACC036261C704FFC566B
D42C7B83DE60F3BD29DAB465D63754C41538E7A077B3D956DFA6FD611F7BD7D1398748091616275EF5DF2A3F33211C441906
58A66A342FBA2E40A09C73B82327D5E5312FFE37BB584F171B624E321FCF62B14DCE18DDFD045F6E67FA5C17CDCBDE6A6072
152C8FB42DA402FAE3CD56335202FB115C723CAB27EE9DBF7A9222FF76176B8FEB5CDF90DEE85CDB25A8D4D6820AD5D73B99
07CEF01C68EF489674FA21ED0A6D9282720BB00CB98F3DA8B96CB13A64F7A6406CBE7C4BFCBF3FCEFADF114FCD3AFCC2E60A
95B584AAFAF100FBBB3DACD482772CD944CED4CBB6FB2ABD02E5D60E3B0DC685F4532FBFF9CFEE82D768ADDA1D2E9CAF7730
D38B94763B38FF537105E6C878E9DF37CF9EBC242017A3031DB944234C6CBB9FD61583A20BB04C426B5864AD9FE31EBFD5BA
F03455FECDABA64A6D8BD6418775198CB6B169A35788E8CBEDCCC3177843EF5A03872119F348DB3A26A4A3D0DAADEB14B78C
AEC7EFFBAD78DE72E1AA70D3D6D9B30182914983DEE0C09D0666B38D3E6A205F64D38FCC1B5192D7B3631C88CEA4CDEB4F78
73A109005878B538DBB3A65BF1CDB71DFD8BD4366BF7B872BFD9C57CF04CCA9C3339D0FA7E16EB82406C49CA517CFF2BEBBB
DDACC7C9F2F75516A165A7533BA3354615CAADDD76174C8A864EBDFC6669D4C73B25E62836EF637FB38BF5224BCE17991D71
ACA8586AC929557DB797F5C37E765A81D241CF17752C3F20FF10CA2DC09BE43A695CD38DEC9EA74EF05D0462F3B570C9173F
CDEE3FCD2DAF532F7B226F8F95886D718C7A394DF39B3BF7ABED4C773FC1B281856FEA61B77867EE5A5A5917ECD6DD4A8FF8
E63B1F7E0F495CCF5F157EB97DB183ABBE4D6BFF657589D1641B1A37F8848AFEF6AFD963DE7CF2B3E39EA96646954FD6CF28
B1006FE2C8661E3786548CE438CD37A2366A4856B0692BF35CA0B0A54BA754D963AF81D5BAD8994CAA03375F3EA91AF6B872
8BA9FF657F872B05870AE94906B30E85D6CE71152CD7E79B3E66F86E3155F5F3B1B94D5B673D83853D837A3B1F314B0A2D8B
6B4ACE536EDEC7FEFE57767482CC6C71EC6193DDECA62B05BFAB0C18A503F0A77121FD5143903DEF8BF419541AEBE364D9D6
839CAF7732AFDE1377D3F552B9BDB4C50C461B5F64AEACD7B8FAF0BFDACEDC7E98F33C43FE5FC742548EE65C2D72E12B3928
B10EE179CB9D7B555E1FF34E72EA635ECAB61C6093E430304CD23BA4B7C37E03AB6D81273497D5A90F9EE57DB18D79F68A70
70CC19B67818E2F584969E42FF1BC09B06E63A62EBAFF5B09B9DEFAB9176CD8D280969D47CBB8BE9734358DFAE1588CDEBD8
BA5128AD53B3C6947CE5C1B38BCF0B761D9D23F9805AF3DF9B871AA3CA33633B6D92E2A0F3443620854EE696F0C5A468E823
DF2F9AB78446CF9336F8E274D41BA2B65E1D5F68B68755384C661B876FAEA0A94F78F3376D611E3AC7AB6C709E4D4BA7C423
F7A93E4E79F703F86C1D4C5A0C2D7084DFE7AC5F90C13291FBEDD6439C2FB7318F7BF3F3CB55933346A9DCB26643BC946A2B
C94FEADBB4C11192CDBF2DA628BF9FE7A5152AF51F3D57E269D3CD0735BE52AD18C5D58114D213AF14FC6EB67EC25E007C91
39FCA974CB413629AB472EF2B24A94A4F4AED78A8872A5855C29D914E5E1F3BC2FB7CDFE7A929B51A474B273C4944E3DA405
364C95A1B8022C69982A27D785D3F7A191FB6D5CAA7C9F3BF7AB1DCC9F0EB1FDEF882B681A72D3236DB4D5981D6630D8C827
935B3AAD4D7B3B767ED7D139D206242D41F2FF6DEFFDB4810153E2D10BA93FD2B95D28AB8EC5643579A46DA58E157CEA3F24
694042A68294D56F76327FD8CFBE162EAE6DD1CCB017D303CBEA8F3B50912476CED4D0A10D8A90905CFAABED4C170F5E3155
6DB53AE139E22A58714D37CA8633515C019650C7F2A3EB0236C8036B9B6DA1BD4F17745FB2F3C862354D1A6517AE0A5F6429
3AFAF453B3469ED02C575A3F6FDD39925A907F4B3E817C4E5BAFEE65B69C7CF28F07D8E4A6FAAFDFD8670304A4DDF7798DBE
E09213E95DB1EF5B1807EC59FF5CDB99A4EF3E3C4BF143197BBBF6D24DD1D6838B8F1576B870AE868B4BAAD563D38B0555A9
B6AE6086404AAF4862999C35367668C3E3A47BDDE6366D9925A5D72B4454DFE6CCDB0488D4FC17ADF7B21D67633880355031
921D55E72F56F337D4B7D6EA6CA4FD151A3DBFFF34F79FBB595FEF64FEF43BE7B8373FF88124B54049EEC603230672FB250D
3436CFCC1598F922B3506226774EA1C4C2172EBEC2E29AA69926F29E81117D7593E645963C2452E2EAC327F7D22FB631FFF9
0BEBE76373FE77C445557F698264D970D69582435C391305D54185555E7CDC10F2573E41ADB19650D59EC142D272FF66E762
D13AE6C50F7B2C2DA6AAE8630606CB44920492706A7556CB478C992189B1C16053A9AD12A98594E1D129437397EE79A6E27C
A060CB41CEA62D8BBD13A418A7172AC9C73AFDD991E9E6533B1FBE6C0B47410578A3DEC98CAAF59FDFC0CFAC67D8C6BC7215
C9108E7AF2B71FE69024E18B9F985FEF60FEB08FBDC7957BF8028FBCEEE62B38E32F204DFEF38142F233C91F0E9DE3FD7C7C
6EF3BEC5DE00F2E6EFF6B2485EE1E2C10B7920C929558E4EADC0DACBA421E399B1AD9951852D111D97582D704BF8F2E3871D
7EB845DFD4A98B7826FDFD3C6FCB41F6B7BB4836CBFAF9389714C82B61E2B814594185AAAE55DB3BA41F1C338C4C1A488D3F
C130923FA43492BFD2470D5D747D054D9398A3088F9BF7B9212229F1F77B49699F25A577F78939BFDBA2C24AD54648069628
F5F2CCEE274F9B42515001965086D2226BFD48CE8C501026936D6CCA58D3ACCD2C513E7C2123375B920690A619C901F6BACD
913BE7C1B38B79829B2F9FB4DDAE474A625ECA328A16FB16665826CB4ACF6EB857E5F5ACF916B63C707485034997727F3598
F52BF89972A5A5BD579F98AD2045D4C583BFC385B3F937F63F7E61911C75D34FB3DFEE5A1CC4F2AFDFD8DB0E73C8AF7E3CC0
7EB5942273D3D6D9C50C760F8B241524A73D17288C4A9055D66BA699C68D796A540645564FDC93C6EB28A5004B0A06124956
A035AA110ABB429BA45CCAD9CB9C9F42281C9DD566F5CBDFF7BEBD1457CA1CDF3C3866A0B56A497B3F3147F13C4341D2DAA8
E73252EFC7A5CA5F642948A25B4C5537766A47A78CF6B37CC7FAD2993479FD09518EBC0B0CC08A2BA22747D65EFEEC0151B0
1AA45AB177D6AEDA89622C50E01C66E6274EBEF86A563A8950D815ED1F59813F4201B0A46C3833BA2E605E234228EC4738D5
FB7143B0522F43289C467AF7A3C04217A3C58050202B00B0F3AC20AACE1F5981FDA818C9F6CBDBCF924E2314CEC46431F964
FF92D7978050202B00B067556379E4A2C002FB76822D637866EC589C77B0807907CE8621193BF5F2EB15998F002B426354E5
F4C6C7D0AE2214004B1AA6CAA2EB02B03F883D305B4D2114D7E48E480CFE7456F9FD2F7D73F668707EED83522FCBE87EF4AC
F9164201B0A47DB6EE212D7054D08750ACBBD4CEE8E0921302D51C42E1AC6C365B50F1B1F8E63B08853D906925E4A24B6CBB
8F50002CE99F6B8FADBFD6CB694128D65707B3DE3B6BD720AF1BA1706E4215F774E23F3A98348462DDBD5AF1382C072B1E03
BC614234F4A83EA861AA1CA158477C25C727FBE78A916CB3D58C6838BDC6E9CAF3A93F6030CFBA5BDC1DA911BB2301BC85A7
603F6E0829194C4328D68BC1AC0F2D738F6B0A55199488C606F1BCE5CE758A9BDEA44528D611533A155D17D030558150002C
2135515CD38DF4AE188462BD24B486910AE2F38613D86C8BAB34238676E2E3CF84D96A0E2C7449ED8C5EC06493F5332E1CBC
5B797110DB9403BCED59F32DAC04BE5E6AC60B7DB2778F09E89FF7CF355A6B6DB3B6BE5DCB1799AD568473CDD300DB824EFF
67B5FE49BB1B8BD4FCF3299B9BA62B11C6F532C0ED0C2971E529D80805C09B523AA222AA7D1187B547E7767A66EC6898AAF8
ECD5092452CBFDA7D2BD6E73F9E52ABDE18F0F999759A4728B156DD0556332D96639A6DE217D69AD9ADAA8793349F824DDEC
A6B3C9DF3324E308E9BAE860D6FBE5ED37598C0805C09B4A8732226BFDB0BCE11A63CB183ED9BBF3FA138C2B7153222DD6A5
BE82820A555CAABCA45ADDD1AF536BD083B03294EA3F2329965A5E6D56C83BE2C94F2BF84BA34132BA6349C524DDC05B99AF
A3BAC912DF9CBD8803C0326DB3B50F695746050308C59A916BE7AF151F7DD218A25E851186DD74FDC317B2A39EFC7DEEDCB1
A93F17DE67724C5C8159A3459EF009A6668D5D03FAB25A4D729E72A927406FB056D03412E9CAEC5D758FEA154EF5C6C8C335
66341B0AE949772A2F201400CB4C8A86636857EB264A108A35BA1D590CA416B8557E46A211AEEAFF4822B398CD7F766ADF7B
22BD133B9F94ABA86DD12C3D6E802506A34D3C6F61B04CC63706700686897F3FCF2329D6BDB879EBEA3C95D199B401F98752
3AA22C9897BA86A45A71527B6452FB0384026019955EF1A83EE8D570685875569B25B13DE24AE161E6DAEE7F64B32DD0470D
1945CA2B77C52E177902F11F15904E6F6DE9D675F6EB46A78C3CA1D9B2613A122C56DB9BE3337B86F4D914E5C317B22B6162
81F8CF4E80C171836AF59FC288D5828B19DB4A87337081AC19B68CF1901658335E845000FCA717AD6151B57E88C31A281C48
F2CDD9332AE8B793E391482DEE7E02970BBC8B41C2B854F9D23804ADCE3A396364724CA2790BF9D9E62C5D0B737C73DF90BE
B1435754A5223F2FBD1E1421F1B9217A9222AB69D6ACCB608C29F1C8B9D4CD9DCC7A5C236B6388D71D4271C5163000EF441A
291135BE62351FA158555563F93E593F77B39AECEDC0D45A6BDFB081D4954B4F16182CD3B5FB92E0084954828CDAA83118FF
789DD498336C139B6B12DB5FB6F0E6B190C31B9E3074D3F54D9D5A16F7CF4702A9F90A375FC16FA7B82417A28F19ECEA2C34
31AA2EA4FE38CCEFC595B2065A18D44B39BF220E00EF34C8ED8AAE0B689BAD4528564FE3748577D6AE9619AA431CAD4E6F23
B56A394D13972ACF28526A757FD4ABE4C51BD1F3371FCE936CA1AE45B3F4085E22B5F40FEBE9A386D129A3506C5E7A024FD2
09A5DA4A7209926F982D9F9043904F7833E5209F235358485DCF139A95AA3FDBF2E4AF1D7DBAC60E6D599D7A82615CFA27E4
AF9E41A2E3DEFCBD6EDC0A9A66E975BEC82C9258EC36ECB97DCFBD32774E894771BDAC2A93C5544C4FB9517A12A1007827A5
5EF6A83E38BD2B16A158259DAC069212D44D143BFA5A76A46A6EEBD51555A99FA42C4E7E5CEA4368EED29D0B149EBD22F0BE
2ECA2B5319FFFDFAC8A421394F9194AB20A945CFA0DEF2EF0A99CD35511B35354DDABA56EDE48C71E9113FEBD5EB95F51AF2
BF181A372CBD3E3866789E298F79290B8F9B6F68D72EBD4E8EE1C845FEAF27B92E1EBC62AA7AE9750ECF3C366D94291C6F9C
C48BD67B3ED9BB596B3BEC64A311AA78CF9A6F65F53E452800DE27A9FDC1839A4B88C36A18E076F864FF5C369C69B1599CF5
3B9ACD36A9CC32C33292AA7F6AD6B854FB93DAFCFE5369C803894FA82897640BFFEE5BA869D2B8FAF07F3FC7FBE5F8DCD334
F9D2EB242538E1C32755BC9BAF20AB44B93487A29BAE2729417CBA3CBD503930A25FAAFDF5069BC9E946EE47D6FA05141CC2
9A7BAB674C3070A3F4D488DD0CEF01B0430DD3E51135BEE3C2418462650DF2BA7D73F6E6F7BFC0126AF0916C0BB6B0AA8B21
2527241A01A2B11A9AA62B2FE5625001C0870894730FEBAE90CA0BA158C95E82B90E72F3C9EA7962341B100DF878268BE966
F9E93B15E7B1E8E88AD318D5593D71F7AB7D100A800F7BD91A1E597319715829BDEC964B397B73FAE2C91D1ED1804FA53369
434A4E8453BD31396865CDCE4F8453BD1AA6CA110A800F6B6650EF51BD0679D857740574B11A7D73F6160C249AB1601D7C2E
BD497BA3F4D49DCA0B42150FD15829AD333597F37ED318550805C0872974B287B4C0948E2884E22F6A9FADF3CDD95D4C4FB1
582D8806FC1506B3FE56F959F287AFE4201A7F9DDAA04CEB8A09A77A2314001F23AF2F21ACEA227670FB6C16ABB97EAAD437
674FD970A6D58694005680C962BC5BE971A3F4144FC14234FEA27121FD3AC5AD97D38A50007C8C19C9F8FD6A9FD221ACC7FE
39F4665DC9609A4FF6EEEAF1429B0DDB0FC10A260626D2BC0DA1B8CECE4F221A9F8D5C9595A3B9FEF907100A808F97DC1149
12031DF675FD442A8322AD2BF652EEDE0E2C650FABC062B344D5FA5D29383CC2EF43343E0F57C17A507329AF3F01A100F878
A382FEB02ACFEAF10284E2E3CD6B45714DA101058746B12E0AACA6976DE1BE397B3A990D08C567689AAEF4CBDBA7D4CB100A
804F92D876FF7EB50FC6E87E24B68C41C2759DE2467E403460B5150E247B65ECAC9B28B6607ACBA710AB05B1F5D7125AC310
0A804F35291ABE5BE981D1051F6390D775A3F46444B5AF04ABCDC05A699CAEF0CCD85E444FD6993488C6476A6150FDF3F763
8E27C0E7C9EC7E72BBE29C08CBA7BC9FC1ACAF9F2C0D283898D2194D7E4640606DD3D16E9FAC5DC9ED91A4098C68FC574215
37B2E6F28BD67B0805C0E71128E7C2ABBDB176C1FB4835E2AC9EB8CB79FB6A274A100D58175C0533B0D025B2F6F28C641CD1
F8205BCD78E1E5DC7D4ABD1CB100F86CB4494A48896B0FBB19A158665A3C1655EB7FADE808C616C2FAD29B7591B57EC125C7
BBD94D980CFB3E33F3E3A165A78BE8C90805C05F61B15AE25BEE845579CAB41244E3358359DFCCA80A2A3916597B093DB760
2772FB9EFBE6EC2E1BCE54191488C6323A9336AD2B36B8E4044201F0D7B1658C5B1567933B22D10C5978F56832ABE7A95FFE
FE427AA205EB16823DE96635F964FF12DF729B29C532476F699DA9F1CF3F806E3D8095D2C2A05E2B3A5A3749D9C841B0582D
0373ED11D53ED74BDD4605584306EC114FC9BE4E710BA19C689FADC3369DAF911CE966F9198C8F025859B97D092114D70D9B
6C8BD582E2C114D2DC78DE72576D54A23C80DDB2D96CC91D915E99BBB27B9F8936FC143C954111DF7C3BB4D41D0503606519
CD868496BBB72BCE71E4B31BEA8B930657FF5C7B746DC09582DF9B19552809E010BA59CD97737F7B50736990D76D5BD8A0CF
FECC5673C960AA7FFEC18D76D702581B52AD38AAD62FA2DA47A4DA282B187015ACBCFE17E4AEF2B82158AAC12692E058CD64
E5435A6040C1A1E2C114818ABB0123D03C5D15907FA865868AC200B07AB56458D5C5877557246AA1737F53A55ED6345D1956
E579ADE848DB4C2D4E3D38A886A972DF9CBD11D5BE3DEC66A3C5B871BE38F9BE41C5C70AFA5FA20C00AC9AC57E48B68C71B7
F242745D804039E7945FD260D60F703B9EB7DCF5CBDB9FD211A535AA71E2C1D153DCD8FA6B9772F766F63C61CBA637C2571E
98EB20294162DB7D9C7D8035480C3872921878DCA7FA38D93EEF569B654A3C92D1FD3820FFE0FD6A9F99792C1607CEA38BD5
1858E872AFEA226DB2D4B9D7D9E89F6B0FA1B83E6BBE89930EB0668901B9AB44D6FADD2C3F43E7763AC7B762CB1845F4E4EB
14D7E092E3AD333538CDE07CCC567366CF13DF9CBD0FEB023B98F51A8313EE884ABE5750F1F1A74DA138DD006B9C18A80C8A
67CDB748EBA366BC506FD239EA37B1D998D229CA507A58D5C5AB452EE4076C4D0BCE8DA76047D75DB994FBEBCBD6F0617EAF
D9590ABCD96AAA9B282177A4A4F6089C6580F5AA520BE9C9570A0EA57646935B8D631DBCD5669D999F28A2A7DCA9B8105070
28A7375E6BC48EB4B051D0B95DB7CACF06151FCBEE7D362AE877F4258F147A596EDF73FFFC03F9185E08B0DE06B81D374A4F
DDABF26A9DA9D199B4F67FC05AA37A84DF57309078B7D283B42C72FAE29DB22B15E0BF6A9CAE2489C1758A5B46F7E3617EAF
D16C70C46FC1908CC5358506E41F6C9AAEC43905B0074ABD3CA1F51EA961933BA2C68403567BDD2040A4E291D425A923F2F5
F801CA509AC1ACC3E9830D8E3649B9567C34A4C435BDEB119DDBE940FB2B916648C354D9ADF273B7CACF32E7B1FB03807D21
F793DB15E743CBDC0BE949D39231ABCD6A2707A63628C784F48A91ECC70D21010587EE567A3433AAB0D913C05BFD06531537
4A4F5DA7B8BD6C0B6F99A1F2951C7B3E5ADB826D543090D81E71B5D025A523DA7EEE3600F01FED8ED21B9493772A2E140E24
8D08FAF4EBD71827ED8849D170CD4451527B04C904AE16B9BC680D674AA7708E00DEA787DD1251EDEB9F7F208676B5742863
42346830EBEDED20C9555C4C4FB9557E96FC19E163AB32007B476E23A4611E5AE67EA7E27C7A576C3383CA91CDACD96305B9
4E3A261CA04D52523B63C2A9DE01F9071FD45C6A66549936D2C26E007F055B3643726972EDDCABF24CEF7ED4F2EA12B687AD
C367E627CA86B3EE537D42286ED4B17C9C2900C7D2C1A4DDABF20A2A3E1A537F35AFEF79DB6C2DB9A835ABB05AA05C373F2D
196D9FA55186D212DB23EE57FB5C2D7289A8F1AD192F52EAE53811009FE1D55C3F0A490C020B5D62EBAFE5F4C6934B98259D
5EFB41896A836A90DB55444F8EACBD1C547C2CB7EFB9C131074602002150CE150E24DD2C3F1342718DA105A67646978F6477
32EBC78583E4579F3A6DC166B3A90D4AA18AC7908CF5715A6993A5C58329C91D51E493AF534E06151F8F6FBE4DEE5D986608
B052442A5EC14022B984834B8EC7D65FCDEC79429BA48C0AFA567B8D4472A54F8886EA264A92DA23EF545CB84E71AB18C971
88894E00F031780A16752C9F54DF24DBBF59763AAA3620BEF94E46F7239233908BBD61AA8C34F67B39AD0373ED43BCEE1141
DF30BF87CEEDEC65B774B21A5A676AC88DA86C3833AF2F21AD3326A1358C345EC2A95E24D3B853713EB13DA27EB2D4E1D64C
0070B44B984D194A5FEA8B4B68B94BAE47929993EB94239B5991C90B241360CB18BD9C16724F20D9FE839ACB241978480BEC
613723FE004E6C4E3EDB3A539BD317FF907625B8E4C4D54297EB14D7DB15E7EE5579DDA7FA3CA8B9B4B86B738D6F5895E7EB
694737CB4F93760AB917DDAE381FD7149ADDFBB47EAA8C2119B7ABB1C7468B49AC919AB13A22383BD260EF6235A4763E0C2D
2317E63172C1920C3FBD2B9664F8D5E3052D33D57D736DA4994F2E738946A8D0CB485DAF37EBCC5693C56AB158CD06B34E67
D290D7852A2E533A352AE827957EDD4449213D29B533FA49E3F5B0AA8BAFFA25829AA62BD1EF07B031916B9FAFE44C4B46C7
85F4416E572FA7B58FD3D63FD74EEE18CCF949B19A6FB4D8FB9344917AFE595B6272776617BB5FA896901B204E2BAC36DBC2
3ACFB13598F563427AD568DEF396BBA1A5EED78A8E901C3E9CEA4D1AF84F9B6FBE6CBB4F2AFAB4AE98CC9E2739BDCF72FBE2
737AE3C9CF19DD8F523AA2135AEF921C20AAD6FF751A40FE90BF960D6732246338B300E0E85406755CEB8BD329AEEE29AE77
AAEF739502C404568AC6A89D14CF0C0B26966705F6B7F28644239C100EB6CED45086D2D3BB1EBD680D7FD67C8B240911D5BE
245BB85FED437E7ED410FCA2F51EC910A863F95DAC468664DC0EA7400200FC4516AB85CE1F79DAFA32A4FCE68C948580C05F
6E891B19F34CDA74CBE3E6E7DEB93E372A6E232600008E981E58DE1EF3C0530A3B58BD63A229D2E8437CE063982CA6A4CE0C
DFFC4B27138F7A645EB85B1D51365A8DB000003881C69936AF1C2FCF6CCFD8A667D5930DE3A2699541838597E1037426BD7F
D1951B15B70B064B67A598620300E03C784A2165847AB7FA814796C7A9A4631732CF37CDB463C2022C2CEE2FA69E10331A67
DA99B2B9375F2749A34423457C00009C185BCE23E9C1F5F2DB6DAC6EF3DBF314D07BB0C19201D5207FB462AC36BAE18977AE
F7E914373C2000008025C95D99C9DD998D8CF649C98CCAA046409C1849FF7AE706CFA59F3D9978D42BC73BACE641D15005EB
EDBE020000D8B02C56CB035AAC7BF209F764579F3CDFA8FAC795E3756A0C4D747C569B55A81633E6596F65050B36C63C33B2
FE51F564234F8149AC0000B09CDE6418E28FE50C14DFAD8EF0CCBE78B9D07F4EC15F56BF204A8E42AC9E1F154ED2A65B52BA
B36E54DC09A40423260000F079B80AFEA86852FFF6DE7093E2993C3A853A513FC81FE52985260C56B4E3FE8107B45892DAB9
279F389BE61E5812F4A8391ECB5D0200C00A6A67F5F8E4F99C4A3EE191E5115C16FAAC2D91CE1B3159901BAC1B8D51CB91F3
0678232AC35BCBF55B6CD67BB5512431281EAE18164C18CC46C40A00005696D6A89B10332AC66A9FB4240495DD38937A2ABD
376FD9B248128D94AF129177225CAB87259B6B9EED2C1C2A8F697A7AA524E862D64592182C7B0F7A06000060CD586D569221
8835F3CB061B2477675DAFB815DBF42C7F90D2C068A5F3469506F5BAEFA1E3B87426DDB2E7382498A93D3967D24EB9279FF0
CEF5B95E7EEB696BE2B2C184000000F6A060A88CD453DE39DEEE29AE67524F9E4D731F128C2FCB0A147A95D684CE8477531B
34D31266CF1CBD76AA2967A0F851737CF564E3B2F7B4B37AC9AF3AD8BD3CA51011030000FBF4BAF627FF7195826E4E7FD150
F9035A0C47C15BF6B6F8B6A498C667997DF995E375ADCC6E3A7FD460366CB458992C66B94E21D72997BDDEC5E9F7C9BBE491
79FE74CAC9D329AEDEB93EF9F452142D0000705691F58F4865772695D47A6EE7D2CF78645E10A844CBDE33229C1C174F73E4
BC79AD4C6772925D6E157A25C982A813F579F492E76DC9E1B551899D19CBD6931CE48F86D544BEEC4C2F1FABE9E30E91FCCA
6831A1CC000080D33693AD26BE4A44AABFBAE9E6CCBE82D8A678525DBEF906ABCD769512E295E375A5E4DADDEA88474DF1E9
BDB9CB1E3AD8166C5C8580A413329D4263D4AEFB34C9579B0EB3860463DD9C81A6997652F537CF742CABF127C4D31732CE93
5CC83DC5CD33FBA27FD195E7ED2958651A0000E0038C66635A4FEE035A4C68E55D921878E7FA78E5784BB5F237DFA3366AC9
6F43CA6F9296F5C386B8676D893903C5CB06E6992CE66E0EBD8BDDDF3B37D8CF1B1E128C4F8A6796ED314DDEC398674E4966
C99F49C90C79C3AC94B36C0825F9D83666773DA3B566B29134E14B862BAB271AFE73AFEACB05FE9ED99E1E591EE7D3CF92AA
FF1AE5FAB28900E42B9083AC996C20C743FE2F0ABD0AE71A0000E0E3698C5A9E52402AEE653B3CAA0C9A07B4D8E0B2D080E2
ABBEF9972F665DBC54E0A77C7BFB06994E417EE591799E54D617B32F92ECE25AE9F5657D0E62F5BC5F6100C93ABC73BCC99B
C91F926C2C1BE720524BC8279CCB387336EDB47B8AEBA9A4E3FE4501CB3AF6A55A19399EC7CDCF93BB330B06CBA813F5A4EA
473F000000C09A218D7A921E905A9B29E530E659CBDAE624A3C8A397A474672776A6C7B725C5B5BE78D999BEACC627994364
FDE3B09AC87B359151F58FA31B9E2474A42EABF1C9E7908A9E32525535515F3FDDDAC6EC1E128C61FD670000000000000000
0000000000000000000000000000000000000000000000000000000000000000000000000000000000000000000000000000
0000000000000000000000000000000000000000000000F83CFF0F6DA801BB
      > EI Q}%
    \rule{0pt}{5.228985507cm}\hspace{5.5cm}%
  }%
}

\begin{document}

	\title[On the sharpness of the \(C^1\)-norm threshold for perturbations]{On the sharpness of the \(C^1\)-norm threshold for perturbations
		in the normally hyperbolic invariant manifold theorem\\---a toy model perspective}
	
	\author{Yujie Huang}
	\address{School of Mathematics and Statistics, Beijing Institute of Technology, Beijing 100081, China}
	\email{huangyujie@bit.edu.cn}
	
	\author{Junhao Li}
	\address{School of Mathematics and Statistics, Beijing Institute of Technology, Beijing 100081, China}
	\email{junhaoli@bit.edu.cn}
	
	\author{Lin Wang}
	\address{School of Mathematics and Statistics, Beijing Institute of Technology, Beijing 100081, China}
	\email{lwang@bit.edu.cn}

	

	\begin{abstract}
		
		The classical normally hyperbolic invariant manifold theorem asserts that a \(C^1\) normally hyperbolic invariant manifold persists under sufficiently small \(C^1\) perturbations. For a family of standard-like dissipative twist maps with Jacobian determinant \(\lambda\in(0,1)\), we prove that \(C^1\) perturbations of Lipschitz seminorm at most \((1-\sqrt{\lambda})^2\) preserve a unique \(C^1\) invariant graph, including at equality. Under the strict bound, this graph is normally hyperbolic. Conversely, for every prescribed \(a^*>(1-\sqrt{\lambda})^2\) and every rational frequency, we construct a \(C^\infty\) perturbation of \(C^1\)-norm exactly \(a^*\) whose unique invariant graph is Lipschitz but not differentiable everywhere. We also give an analytic perturbation at the critical norm whose invariant graph is \(C^1\) but not normally hyperbolic. Thus the closed bound for \(C^1\) regularity and the strict bound for normal hyperbolicity are both sharp.
	\end{abstract}

	\maketitle
	
	\vspace{2em}
	
	%
	\section{\sc Introduction}
	Let $\alpha:=(\alpha_1,\alpha_2)\in \R^2$. Fix $\lambda\in (0,1)$. The critical perturbation size for this fixed
    $\lambda$ is $(1-\sqrt{\lambda})^2$. We consider
	\begin{equation}\label{F1}
		T_{\lambda,\alpha}(x,y):=(x+\alpha_1+\lambda y,\alpha_2+\lambda y).
	\end{equation}
	The parameter $\lambda$ controls the dissipation, and $\alpha_1,\alpha_2$
	are constants that determine the translation in the $x$ and $y$ directions, respectively. For any parameter pair $(\alpha_1,\alpha_2) \in \mathbb{R}^2$, a direct calculation shows that the integrable system \eqref{F1} possesses a unique invariant graph
	\begin{equation}\label{ingm1}
		\Gamma_0 = \mathbb{R} \times \left\{\frac{\alpha_2}{1-\lambda}\right\}.
	\end{equation}

	The map $T_{\la,\alpha}$ restricted on $\Gamma_0$ reduces to a circle diffeomorphism $g_\al$ with rotation number $\al_1+\frac{\lambda}{1-\lambda}\al_2$. Note that the invariant graph $\Gamma_0$  is normally hyperbolic. More precisely, it is  so called immediately absolutely $r$-normally hyperbolic for any $r\in\N$ in terms of \cite[Definition 2]{HPS}. This notion is also referred to as $r$-bunching (see for instance \cite{KH}).

	\begin{Definition}
		Let $r\geq1$ be an integer and let $V$ be a $C^r$ compact submanifold of a smooth manifold $\mathcal M$.
		Suppose $f: \mathcal M \to \mathcal M$ is a $C^{1}$ diffeomorphism and $f(V) = V$.
		Let $T_V \mathcal M$, the tangent bundle of $\mathcal M$ over $V$, have a $Tf$-continuous invariant splitting
		\[
		T_V \mathcal M = N^u \oplus TV \oplus N^s .
		\]
		For any $z \in V$ put
		\[
		T_z f = N^u_z f \oplus V_z f \oplus N^s_z f .
		\]
		Thus,
		\[
		Tf|_{TV} = Vf,\quad
		Tf|_{N^u} = N^u f,\quad
		Tf|_{N^s} = N^s f .
		\]
		$f$ is \emph{immediately absolutely $r$-normally hyperbolic at $V$} iff
		$f$ is $C^r$ and there is a Riemann structure on $T\mathcal M$ such that for all $z \in V$,
		$k=0$ and $r$:
		\begin{enumerate}
			\item[(a)] $\inf_z m(N^u_z f) > \sup_z \|V_z f\|^k$ \quad and
			\item[(b)] $\sup_z \|N^s_z f\| < \inf_z m(V_z f)^k$ ,
		\end{enumerate}
		where the minimum norm $m(A)$ of a linear transformation $A$ is defined as
		\[
		m(A) = \inf \{|Ax| : |x| = 1\} .
		\]
		When $A$ is invertible, $m(A) = \|A^{-1}\|^{-1}$. An inequality involving a zero-dimensional bundle is omitted; in particular, condition (a) is omitted when $N^u=\{0\}$.
	\end{Definition}
	We use $r$-normal hyperbolicity in its usual uniform-iterate sense; the immediately absolute inequalities above are a sufficient condition. In the normally attracting case, these estimates require constants $C>0$, $0<\nu<1$ such that, for all $n\geq0$,
    \[
      \|Df^n|_{N^s_z}\|\leq C\nu^n,\qquad
      \|Df^n|_{N^s_z}\|\,
      \|(Df^n|_{T_zV})^{-1}\|^r\leq C\nu^n.
    \]
    Our sufficient criterion below gives the stronger immediately absolute property in an adapted metric. A pointwise derivative inequality in a fixed coordinate is not used as a necessary condition.

	For each $(x,y)\in\R^2$,
	\[DT_{\la,\alpha}(x,y)=\left(
	\begin{array}{cc}
		1 & \la \\
		0 & \la \\
	\end{array}
	\right)=\mathsf P\left(
	\begin{array}{cc}
		1 & 0 \\
		0 & \la \\
	\end{array}
	\right)\mathsf P^{-1},
	\]
	where $\mathsf P=\left(
	\begin{array}{cc}
		1 & \frac{-\la}{1-\la} \\
		0 & 1 \\
	\end{array}
	\right)$. Hence, we have a $TT_{\la,\al}$-invariant splitting
	\[T\Gamma_0\oplus N^s,\]
	where $TT_{\la,\al}$ denotes the tangent map of $T_{\la,\al}$.
	Moreover, there exists a constant $C$ such that for all $i\in\Z$ and $z\in \Gamma_0$, there hold
	\begin{equation}\label{hype}
		\frac{1}{C}\la^i|v|\leq |DT_{\la,\al}^i(z)v|\leq C\la^i|v|,\quad \mathrm{for}\ v\in N^s_z,
	\end{equation}
	\[\frac{1}{C}|v|\leq |DT^i_{\la,\al}(z)v|\leq C|v|,\quad \mathrm{for}\ v\in T_z\Gamma_0,\]
	where the constant $C$ can be estimated by the norms of $\mathsf P$ and $\mathsf P^{-1}$.
	
	We consider the persistence of invariant graphs under perturbations. More precisely, we study the three-parameter family
	\begin{equation}\label{Fp}
		F^\phi_{\lambda,\alpha}(x,y) = \big(x + \alpha_1 + \lambda y + \phi(x),\; \alpha_2 + \lambda y + \phi(x)\big),
	\end{equation}
	where \(\lambda\in(0,1)\), \(\alpha=(\alpha_1,\alpha_2)\in\mathbb{R}^2\), and \(\phi\) is a \(1\)-periodic \(C^1\) function satisfying \(\int_{0}^{1} \phi(x)\,dx = 0\).  The map \(F^\phi\colon\mathbb{R}^2\to\mathbb{R}^2\) is the lift to the universal cover of a dissipative twist map on \(\mathbb{T}\times\mathbb{R}\). If \(\alpha_2=0\), then \(F^\phi\) is called an exact conformally symplectic twist map. For related studies on dissipative twist maps and exact conformally symplectic twist maps, we refer the reader to \cite{Cas87,CCD13,CCD20,CCD22,Lec87,Mas23} and the references therein.
	
	The classical normally hyperbolic invariant manifold (NHIM) theorem asserts that for fixed \(\lambda\) and any \(\alpha_1,\alpha_2\in\mathbb{R}\), if \(\|\phi\|_{C^1}\) is sufficiently small, then \(F^\phi\) preserves a \(C^1\) invariant graph, where
	\[\|\phi\|_{C^1}:=\max\{\|\phi\|_{C^0},\|\phi'\|_{C^0}\}.\]
	For a zero-average periodic $C^1$ function, there is a point where $\phi$ vanishes. Integrating $\phi'$ along the shorter arc to that point gives
    \[
       \|\phi\|_{C^0}\leq\tfrac12\|\phi'\|_{C^0},\qquad
       \|\phi\|_{C^1}=\|\phi'\|_{C^0}=\|\phi\|_{\mathrm{Lip}}.
    \]
	The NHIM theorem provides a sufficient condition for the existence of a \(C^1\) invariant graph. In this paper, we investigate the necessity aspect from a quantitative viewpoint, i.e., we address the following question:
	
	\vspace{1ex}
	
	\begin{itemize}
		\item \textbf{Question:} {\it For any \(\lambda\in(0,1)\), what is  the critical size of the \(C^1\)-norm of the perturbation \(\phi\) that guarantees a \(C^1\) invariant graph?}
	\end{itemize}
	
	\vspace{1ex}
	\subsection{Main results}
	The regularity of invariant graphs is a classical theme in twist dynamics; see Birkhoff \cite{Bir20}. For the present dissipative model, a continuous periodic invariant graph is necessarily Lipschitz, as the direct argument following \eqref{eq:main} shows. Uniqueness follows from the area-contraction argument in Section~2.1. On the one hand, we give a quantitative version of the NHIM theorem for the model \eqref{Fp}.
	
	\begin{Theorem}\label{M2}
		Given \(\lambda \in (0,1)\), for any \(\alpha \in \mathbb{R}^2\) and any \(C^1\) perturbation \(\phi\) with zero average satisfying
		\begin{equation}\label{contrac}
			\|\phi\|_{C^1} \leq (1-\sqrt{\lambda})^2,
		\end{equation}
		the map \(F_{\lambda,\alpha}^\phi\) admits a unique \(C^1\) invariant graph. Moreover, if
		\begin{equation}\label{m22con}
			\|\phi\|_{C^1} < (1-\sqrt{\lambda})^2,
		\end{equation}
		then the invariant graph persisted by the map \(F_{\lambda,\alpha}^\phi\)  is $1$-normally hyperbolic.
	\end{Theorem}
	
	First, under condition \eqref{contrac}, the classical graph transform produces a Lipschitz invariant graph. At the critical value, the slope intervals of the iterated cone fields shrink uniformly to singletons; see \eqref{cone-width-iterate}. Together with continuity of the finite cone iterates, this yields $C^1$ regularity without a strict derivative contraction. Under the strict inequality \eqref{m22con}, the derivative transform contracts uniformly and gives normal hyperbolicity; compare the $1$-fiber contraction and $C^r$ section theorems in \cite[Pages 31--35]{HPS}. Our proof follows the geometric approach of \cite{BB}, with an explicit estimate that also covers equality.
	
	On the other hand, we establish the following result.
	
	\vspace{1ex}
	
	\begin{Theorem}\label{Th1}
		Given \(\lambda \in (0,1)\), for any rational number \(\beta\) and
\[{a}^*>(1-\sqrt{\lambda})^2,\]
 there exist \(\alpha \in \mathbb{R}^2\) and \(\phi_\lambda \in C^\infty(\mathbb{T})\) such that
		\[
		\|\phi_{\lambda}\|_{C^1} =a^*,
		\]
		and \(F_{\lambda,\alpha}^{\phi_{\lambda}}\) admits a unique Lipschitz invariant graph. Moreover, the restriction of \(F_{\lambda,\alpha}^{\phi_{\lambda}}\) to this invariant graph is a circle map of frequency \(\beta\), but the invariant graph possesses non-differentiable points.
	\end{Theorem}
	The proof of Theorem~\ref{Th1} is constructive. The construction is based on the theory of circle maps \cite{H79} and the dissipative versions of the Herman-Mather formula and its derivation (see \eqref{eq:main} and (A') below), and draws on ideas from the constructions of Herman \cite{H83} and Arnaud \cite{Ar11} in the area-preserving twist map setting. The difference here is that we require not only that the invariant graph possesses non-differentiable points, but also that the \(C^1\)-norm of the perturbation has a prescribed a priori magnitude. Compared with the conservative case, this imposes additional constraints.
	
	Together, Theorems~\ref{M2} and~\ref{Th1} show that $(1-\sqrt\lambda)^2$ is the optimal closed bound guaranteeing a $C^1$ invariant graph for all parameters and all zero-average $C^1$ perturbations. Nondifferentiable graphs occur at every prescribed larger norm, whereas $C^1$ regularity persists at equality. Normal hyperbolicity requires a separate distinction: it holds under the strict bound and may fail at equality, as Proposition~\ref{lexx33} shows.
	
	\begin{Remark}\label{R2}
		Fix \(\lambda \in (0,1)\), and let $C^1_0(\T)$ denote the zero-average functions in $C^1(\T)$. We define the following subsets of $\R^2\times C^1_0(\T)$:
		\[
		\begin{aligned}
			A &:= \bigl\{(\alpha,\phi):\|\phi\|_{C^1}\leq(1-\sqrt\lambda)^2\bigr\}, \\
			B &:= \bigl\{(\alpha,\phi):F_{\lambda,\alpha}^\phi\text{ has a }C^1\text{ invariant graph that is }1\text{-normally hyperbolic}\bigr\}, \\
			C &:= \bigl\{(\alpha,\phi):F_{\lambda,\alpha}^\phi\text{ has a }C^1\text{ invariant graph}\bigr\}.
		\end{aligned}
		\]
		The parameter is part of each element of these sets: failure of normal hyperbolicity for one parameter does not exclude its occurrence for another parameter with the same perturbation. The strict inclusions $A\subsetneq B\subsetneq C$ are illustrated schematically in Fig.~\ref{X}. For the detailed verification, see Section~\ref{R3}.
		
		\begin{figure}[htbp]
			\small \centering
			\IfFileExists{X.png}{\includegraphics[width=5.5cm]{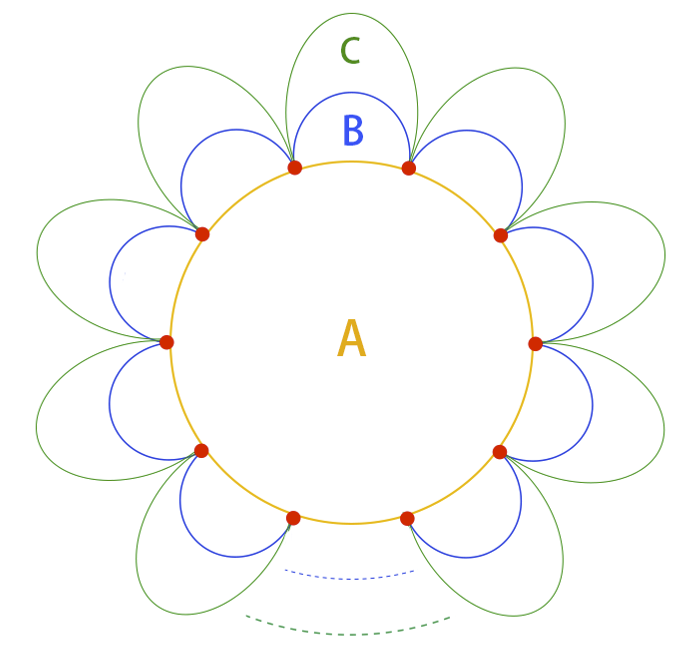}}{\SharpOriginalFigure}
			\caption{Schematic relationship among the sets \(A, B, C\). Red points represent critical examples in $C\setminus B$; the drawing specifies inclusions, not the topology of the boundaries.}
			\label{X}
		\end{figure}
		
		The orange, blue and green regions represent $A$, $B$ and $C$, respectively. The red points belong to $C$ and lie on $\partial A$ outside $B$. Countably many distinct examples are given by
        \[
        \alpha=(0,0),\qquad
        \phi_q(x)=-\frac{(1-\sqrt\lambda)^2}{2\pi q}\sin(2\pi qx),
        \qquad q\in\N,
        \]
        as follows from Theorem~\ref{M2} and Lemma~\ref{16p}. Their norms equal the critical value, and replacing $\phi_q$ by $t\phi_q$, $0\leq t<1$, gives elements of $A$ approaching them. The examples of Theorem~\ref{Th1} lie outside $C$ and have prescribed supercritical norms.
	\end{Remark}

	\begin{Remark}
		In \cite[Page 1]{Man78}, Ma\~{n}\'{e} introduced the ``persistent'' condition. To give a concrete definition of the ``persistent'' condition in our setting, we pass from the universal covering space back to the cylinder. Let \(\mathcal M = \mathbb{T} \times \mathbb{R}\), and let \(\mathrm{Diff}^1(\mathcal M)\) denote the set of diffeomorphisms from \(\mathcal M\) to itself. Denote by \(\gamma\) and \(f\) the projections of \(\Gamma\) and \(F_{\lambda,\alpha}^\phi\) onto \(\mathcal M\), respectively. The invariant graph \(\gamma\) is called {persistent} if there exist a neighborhood \(\mathcal{U}\) of \(\gamma\) in \(\mathcal M\) and a neighborhood \(\mathcal{V}\) of \(f\) in \(\mathrm{Diff}^1(\mathcal M)\) such that
		\begin{enumerate}
			\item[(a)] for every \(\tilde{f} \in \mathcal{V}\), the set \(\gamma_{\tilde{f}} = \bigcap_{n \in \mathbb{Z}} \tilde{f}^n(\mathcal{U})\) is a \(C^1\) submanifold of \(\mathcal M\) and \(\gamma_f = \gamma\);
			\item[(b)] \(\gamma_{\tilde{f}}\) is \(C^1\) close to \(\gamma\) whenever \(\tilde{f}\) is \(C^1\) close to \(f\).
		\end{enumerate}
		In our setting, we say that \(\Gamma\) is persistent if its projection \(\gamma\) is persistent. Moreover, we may define
		\[
		B' := \bigl\{(\alpha,\phi)\in\R^2\times C^1_0(\T):
        F_{\lambda,\alpha}^\phi\text{ has a persistent invariant graph}\bigr\}.
		\]
		According to \cite{Man78} and the NHIM theorem, the set \(B'\) coincides with the set \(B\) defined in Remark~\ref{R2}.
	\end{Remark}

	\subsection{A key ingredient: the Herman-Mather formula}
	We follow the notation of Herman \cite{H79}. The symbols \(\mathrm{Diff}^r_+(\mathbb{R})\) and \(\mathrm{Diff}^r_+(\mathbb{T})\) denote the orientation-preserving \(C^r\) diffeomorphisms of the line and the circle, respectively, where \(r\in\N\cup\{0,+\infty,\omega\}\). For $r=0$, they denote orientation-preserving homeomorphisms. The degree-one lifts of circle maps form
	\[
	D^r(\mathbb{T}) := \{f \in \mathrm{Diff}^r_+(\mathbb{R}) \mid f - \mathrm{Id} \in C^r(\mathbb{T})\}.
	\]
	For any \(f \in D^r(\mathbb{T})\), the rotation number \(\rho(f)\) is well-defined. By \cite[Proposition 2.1]{SoW},
	the map \(F_{\lambda,\alpha}^{\phi_{\lambda}}\) admits a continuous periodic invariant graph \(\Gamma := \{(x, \Psi_{\lambda,\alpha}(x)) \mid x \in \mathbb{R}\}\) if and only if there exists \(g_{\lambda,\alpha} \in D^0(\mathbb{T})\) satisfying the following equation
	\begin{equation}\tag{A}\label{eq:main}
		g_{\lambda,\alpha}(x) + \lambda g_{\lambda,\alpha}^{-1}(x) = (1+\lambda)x + (1-\lambda)\alpha_1 + \lambda\alpha_2 + \phi_{\lambda}(x) \quad \forall x \in \mathbb{R}.
	\end{equation}
	Note that
	\[
	g_{\lambda,\alpha}(x) = x + \alpha_1 + \lambda \Psi_{\lambda,\alpha}(x) + \phi_\lambda(x).
	\]
	Conversely, given a lift satisfying \eqref{eq:main}, define
    \[
       \Psi(x):=x-g^{-1}(x)-\alpha_1+\alpha_2.
    \]
    This is periodic, and substitution in \eqref{eq:main} gives
    $F(x,\Psi(x))=(g(x),\Psi(g(x)))$. This also verifies the formula without any differentiability assumption on the graph.
    Since both $g$ and $g^{-1}$ are increasing, subtracting
    \eqref{eq:main} at $x<y$ gives
    \[
    \begin{split}
    0<g(y)-g(x)&\leq(1+\lambda+\|\phi\|_{\mathrm{Lip}})(y-x),\\
    0<g^{-1}(y)-g^{-1}(x)&\leq
       \lambda^{-1}(1+\lambda+\|\phi\|_{\mathrm{Lip}})(y-x).
    \end{split}
    \]
    Thus both lifts are Lipschitz, and the reconstruction formula proves
    that $\Psi$ is Lipschitz as claimed.
	For $\lambda = 1$, formula (A) was given by Herman in \cite[Chapter II]{H83}, based on which he conducted a detailed study of the dynamics of area-preserving twist maps. Meanwhile, Mather established the same result in \cite{M84} independently. Recently, \cite{SoW} extended this formula to the dissipative setting. Henceforth, we shall refer to formula (A) as the {\it Herman--Mather formula}.

	At points where $g$ is differentiable at both $x$ and $g^{-1}(x)$, with $g'(g^{-1}(x))>0$, differentiation of \eqref{eq:main} yields
	\begin{equation}\tag{A'}\label{gdiff}
		g'_{\lambda,\alpha}(x) + \frac{\lambda}{g'_{\lambda,\alpha}(g_{\lambda,\alpha}^{-1}(x))} = 1+\lambda + \phi_\lambda'(x).
	\end{equation}
	We shall call formula (A') the {\it derived Herman--Mather formula}. Essentially, the Herman--Mather formula is derived from the invariance of the graph; its significance lies in the fact that {\it it directly relates the perturbation of the map to the dynamics of the map restricted to the invariant graph, while the invariant graph itself does not appear explicitly in the formula.}
	
	Most of the proofs and constructions in this paper rely closely on this formula. In particular, it is used in the verification of the normal hyperbolicity of the invariant graph in the proof of Theorem 1 (Lemma~\ref{lem:universal-g-lower}), in the construction of perturbations in Theorem 2, and in the proof of the strict inclusion relations between the sets in Remark~\ref{R2}.

	\vspace{1ex}
	\noindent\textbf{Notation convention.} For notational simplicity, we shall omit the subscripts and superscripts of the relevant notations whenever no confusion arises. For instance, we write \(F\), \(\phi\), \(\Psi\), \(g\), etc., instead of the more cumbersome \(F_{\lambda,\alpha}^{\phi_\lambda}\), \(\phi_\lambda\), \(\Psi_{\lambda,\alpha}\), \(g_{\lambda,\alpha}\), and so on. The dependence of these quantities on \(\lambda\) and \(\alpha\) should be clear from the context.

\section{\sc Proof of Theorem \ref{M2}}
	\subsection{Persistence of Lipschitz invariant graph}
	
	\subsubsection{Construction of the graph transform}
	Fix $\la\in (0,1)$. Let \begin{equation}\label{iI}
		I := \left[\frac{\al_2}{1-\la}-(1-\lambda), \frac{\al_2}{1-\la}+(1-\lambda)\right]
	\end{equation}and denote by $\Lip(\T, I)$ the space of 1-periodic Lipschitz maps $\Psi: \R \to I$ with Lipschitz constant $\mathcal{L}(\Psi)$. For $K > 0$, define
	\[
	\Lip_K := \{\Psi \in \Lip(\T, I) \mid \mathcal{L}(\Psi) \leq K\}.
	\]
	
	Consider the map ${F_{\lambda,\alpha}^\phi}(x,y) = (X(x,y), Y(x,y))$ where
	\begin{align*}
		X(x,y) &:= x + \alpha_1 + \lambda y + \phi(x), \\
		Y(x,y) &:= \al_2+\lambda y + \phi(x).
	\end{align*}
	For any $\Psi \in \Lip_K$, we have the compositions
	\begin{align*}
		X \circ (\mathrm{Id}, \Psi)(x) &= x + \alpha_1 + \lambda \Psi(x) + \phi(x), \\
		Y \circ (\mathrm{Id}, \Psi)(x) &= \al_2+\lambda \Psi(x) + \phi(x).
	\end{align*}
	Let $\varepsilon_\phi := \|\phi\|_{\mathrm{Lip}}\leq(1-\sqrt\lambda)^2$. It follows from $\int_\T\phi(x)dx=0$ that $\|\phi\|_{C^0}\leq \varepsilon_\phi$.
	
	\begin{Lemma}\label{invertib}
		Define the constant
		\[
		K_1 := \frac{2}{\sqrt{\lambda}} - 1.
		\]
		If $K < K_1$, then for each $\Psi \in \Lip_K$, the map $X \circ (\mathrm{Id}, \Psi)$ is invertible with Lipschitz inverse satisfying
		\[
		\mathcal{L}\left([X \circ (\mathrm{Id}, \Psi)]^{-1}\right) \leq \frac{1}{1 - \lambda K - \varepsilon_\phi}.
		\]
	\end{Lemma}
	
	\begin{proof}
		Define $\tilde{X}(x) := X \circ (\mathrm{Id}, \Psi)(x) - x=\alpha_1+\lambda\Psi(x)+\phi(x)$. Then $\tilde{X}$ is Lipschitz with
		\[
		|\tilde{X}(x_1) - \tilde{X}(x_2)| \leq (\lambda K + \varepsilon_\phi)|x_1 - x_2|.
		\]
		The condition $K < K_1$ implies $\lambda K + (1 - \sqrt{\lambda})^2 < 1$, and consequently
		\[
		\mathcal{L}(\tilde{X})\leq \lambda K + \varepsilon_\phi
\leq \lambda K + (1 - \sqrt{\lambda})^2 < 1.
		\]
		Since $X\circ(\mathrm{Id},\Psi)(x)-x$ is bounded and periodic, this map is onto $\R$. The preceding estimate makes it strictly increasing. Therefore, the map $X\circ(\mathrm{Id},\Psi)$ is invertible and the Lipschitz constant of $[X\circ(\mathrm{Id},\Psi)]^{-1}$ satisfies:
		\[
		\mathcal{L}([X\circ(\mathrm{Id},\Psi)]^{-1})\leq \frac{1}{1-\lambda K -\varepsilon_\phi}
		\]
	\end{proof}
	
	When $[X \circ (\mathrm{Id}, \Psi)]^{-1}$ exists, we define the graph transform by
	\[
	\mathcal{T}\Psi : x \mapsto Y \circ (\mathrm{Id}, \Psi) \circ [X \circ (\mathrm{Id}, \Psi)]^{-1}(x).
	\]
	For the graph $\tilde{\Gamma}(\Psi) := \{(x, \Psi(x)) \mid x \in \R\}$, we have the invariance property
	\[
	\tilde{\Gamma}(\mathcal{T}\Psi) = {F_{\lambda,\alpha}^\phi}(\tilde{\Gamma}(\Psi)).
	\]
	
	\begin{Lemma}\label{welld}
		Define the constant
		\[
		K_2 := \frac{1}{\sqrt{\lambda}} - 1.
		\]
		The graph transform $\mathcal{T}: \Lip_{K_2} \to \Lip_{K_2}$ is well-defined, i.e., $\mathcal{T}\Psi \in \Lip_{K_2}$ for all $\Psi \in \Lip_{K_2}$.
	\end{Lemma}
	
	\begin{proof}
		For $\Psi \in \Lip_{K_2}$, set $c:=\alpha_2/(1-\lambda)$. We first establish the uniform bound
		\[
		\|\mathcal{T}\Psi-c\|_{C^0}
        \leq\lambda\|\Psi-c\|_{C^0}+\|\phi\|_{C^0}
        \leq\lambda(1-\lambda)+(1-\sqrt\lambda)^2<1-\lambda.
		\]
        Here the last inequality follows from
        $(1-\sqrt\lambda)^2<(1-\lambda)^2$. Thus $\mathcal T\Psi$ takes values in $I$, for every sign of $\alpha_2$.
		
		The Lipschitz estimate follows from:
		\begin{align*}
			|\mathcal{T}\Psi(x_1) - \mathcal{T}\Psi(x_2)|
			&\leq \frac{\lambda \mathcal{L}(\Psi) + \varepsilon_\phi}{1 - \lambda \mathcal{L}(\Psi) - \varepsilon_\phi}|x_1 - x_2| \\
			&\leq \frac{\lambda K_2 + (1 - \sqrt{\lambda})^2}{1 - \lambda K_2 - (1 - \sqrt{\lambda})^2}|x_1 - x_2| \\
			&= K_2 |x_1 - x_2|,
		\end{align*}
		the first inequality above is obtained from the definition of $\mathcal{T}$ and Lemma~\ref{invertib}.
		
		Thus $\mathcal{T}\Psi \in \Lip_{K_2}$.
	\end{proof}

	\subsubsection{Contraction mapping}
	Note that $\Lip_K$ is a closed subset of the Banach space $C^0(\T,\R)$ equipped with the $C^0$-metric, and hence is complete. To obtain an invariant Lipschitz graph, it remains to show that the graph transform is a contraction mapping. Assuming the invertibility of $X \circ (\Id,\Psi)$ and the well-definedness of $\mathcal{T}(\Lip_K\to \Lip_K)$, we establish this through the following lemma.
	
	\begin{Lemma}\label{ctrmap}
		Define the constant
		\[
		K_3 := \frac{1}{\lambda} - 1.
		\]
		Assume that $\mathcal T$ maps $\Lip_K$ into itself. If $K < K_3$, then the graph transform $\mathcal{T}: \Lip_K \to \Lip_K$ is a contraction. Specifically, for any $\Psi_1, \Psi_2 \in \Lip_K$,
		\[
		\|\mathcal{T}\Psi_1 - \mathcal{T}\Psi_2\|_{C^0} \leq l \|\Psi_1 - \Psi_2\|_{C^0},
		\]
		where $0 < l < 1$.
	\end{Lemma}
	
	\begin{proof}
		Fix $z \in \T$ and $\Psi_1, \Psi_2 \in \Lip_K$. Let $(x,y)$ be the point on the graph of $\Psi_1$ determined by
		\[
		x := [X \circ (\Id, \Psi_1)]^{-1}(z), \quad y := \Psi_1(x).
		\]
		
		By definition of the graph transform, we have:
		\begin{align*}
			&\mathcal{T}\Psi_1(z) = Y\circ(\Id, \Psi_1)(x), \\
			&\mathcal{T}\Psi_2(z) = \mathcal{T}\Psi_2 \circ X\circ(\Id, \Psi_1)(x),\\
			&\mathcal{T}\Psi_2 \circ X(x, \Psi_2(x)) =Y\circ(\Id, \Psi_2)(x).
		\end{align*}
		It follows that
		\[|Y\circ(\Id, \Psi_2)(x) - \mathcal{T}\Psi_2 \circ X\circ(\Id, \Psi_1)(x)|\leq K\lambda |\Psi_1(x) - \Psi_2(x)|.\]
		Moreover, we have
		\begin{align*}
			|\mathcal{T}\Psi_1(z) - \mathcal{T}\Psi_2(z)|
			&\leq |Y\circ(\Id, \Psi_1)(x) - Y\circ(\Id, \Psi_2)(x)| \\
			&\quad + |Y\circ(\Id, \Psi_2)(x) - \mathcal{T}\Psi_2 \circ X\circ(\Id, \Psi_1)(x)| \\
			&\leq \lambda |\Psi_1(x) - \Psi_2(x)| + K\lambda |\Psi_1(x) - \Psi_2(x)| \\
			&= (\lambda + K\lambda) |\Psi_1(x) - \Psi_2(x)|.
		\end{align*}
		This yields the uniform estimate:
		\[
		\|\mathcal{T}\Psi_1 - \mathcal{T}\Psi_2\|_{C^0} \leq l \|\Psi_1 - \Psi_2\|_{C^0},
		\]
		where $l := \lambda(1 + K)$. When $K < K_3 = \frac{1}{\lambda} - 1$, we have $l < 1$, proving that $\mathcal{T}$ is indeed a contraction.
	\end{proof}

	Note that for each $\lambda\in (0,1)$, we have
	\[K_2<\min\{K_1,K_3\}.\]
	It follows that the map $\mathcal{T} \colon \mathrm{Lip}_{K_2} \to \mathrm{Lip}_{K_2}$ admits a unique fixed point $\tilde{\Psi}$. Moreover, the graph
	\begin{equation}\label{tGp}
		\tilde{\Gamma}(\tilde{\Psi}) = \bigl\{ \bigl(x, \tilde \Psi(x)\bigr) \bigm| x \in \mathbb{R} \bigr\}
	\end{equation}
	is an invariant graph for ${F_{\lambda,\alpha}^\phi}$, uniquely determined within $\Lip_{K_2}$. Notice that the contraction factor at $K=K_2$ is $\sqrt\lambda<1$, including the critical case.

    To prove uniqueness among all continuous periodic invariant graphs, let $\Gamma_1$ and $\Gamma_2$ be two such graphs on the cylinder. The map $F$ is a global diffeomorphism: for $(X,Y)=F(x,y)$,
    \[
    x=X-Y-\alpha_1+\alpha_2,\qquad
    y=\lambda^{-1}\bigl(Y-\alpha_2-\phi(x)\bigr).
    \]
    It preserves the lower end of the cylinder. Hence it preserves the lower component $U_i$ of the complement of each invariant graph $\Gamma_i$. Their symmetric difference $D:=U_1\mathbin{\triangle}U_2$ is invariant and has finite area
    \[
      |D|=\int_0^1|\Psi_1(x)-\Psi_2(x)|\,dx.
    \]
    Since $\det DF=\lambda$, the change-of-variables formula gives
    $|D|=|F(D)|=\lambda|D|$. Therefore $|D|=0$, and continuity gives $\Psi_1=\Psi_2$.

	    \subsection{Persistence of $C^1$ invariant graph}
    The proof is inspired by \cite[Proof of Theorem 3.1]{BB}; see also the cone-field approach in \cite{Ne}. At the critical value, the cone is only weakly invariant. The quantitative estimate below replaces strict cone contraction and proves regularity directly from difference quotients.

    \subsubsection{The cone condition}
    We write $F$, $\Psi$ and $\Gamma$ for the map, the invariant graph function constructed above and its graph, respectively, and retain $K_2=\lambda^{-1/2}-1$. In particular, $\Lip(\Psi)\leq K_2$. Let
    \[
       \mathcal C_{K_2}(x):=\{(v_1,v_2)\in\R^2:|v_2|\leq K_2|v_1|\}
    \]
    be the cone at $(x,\Psi(x))$. The restriction of $F$ to $\Gamma$ is represented by the increasing degree-one lift
    \[
       g(x)=x+\alpha_1+\lambda\Psi(x)+\phi(x),\qquad
       F(x,\Psi(x))=(g(x),\Psi(g(x))).
    \]
    For $x<y$, the Lipschitz bounds give
    \[
       \sqrt{\lambda}\,(y-x)
       \leq g(y)-g(x)
       \leq(2-\sqrt{\lambda})(y-x),
    \]
    because $\lambda K_2+(1-\sqrt{\lambda})^2=1-\sqrt{\lambda}$. Thus $g^{-1}$ is well defined and $\Lip(g^{-1})\leq\lambda^{-1/2}$, consistently with Lemma~\ref{invertib}.

    For clarity, define the set of limiting secant slopes at $x$ by
    \[
      \mathscr E(x):=\left\{\lim_{j\to\infty}
      \frac{\Psi(x+h_j)-\Psi(x)}{h_j}:
      h_j\ne0,\ h_j\to0,\ \text{the limit exists}\right\}.
    \]
    This is a nonempty compact subset of $[-K_2,K_2]$. Indeed, it equals
    \[
      \bigcap_{m\geq1}\overline{\left\{
      \frac{\Psi(x+h)-\Psi(x)}{h}:0<|h|<\frac1m\right\}},
    \]
    an intersection of nested nonempty compact sets; the reverse inclusion in this representation follows by choosing, for each $m$, a quotient within $1/m$ of the prescribed slope. We use the unoriented tangent cone
    \[
      TC_\Gamma(x):=\bigcup_{t\in\mathscr E(x)}\R(1,t)
      \subseteq\mathcal C_{K_2}(x).
    \]

    The slope of the image of $(1,t)$ under $DF(x,\Psi(x))$ is
    \[
      \mathcal S_x(t):=
      \frac{\phi'(x)+\lambda t}{1+\phi'(x)+\lambda t}.
    \]
    For $|t|\leq K_2$, condition~\eqref{contrac} implies
    \[
      |\phi'(x)+\lambda t|
      \leq(1-\sqrt{\lambda})^2+\lambda K_2
      =1-\sqrt{\lambda}.
    \]
    In particular, its denominator is at least $\sqrt{\lambda}>0$, and
    \begin{equation}\label{test2}
      |\mathcal S_x(t)|\leq
      \frac{1-\sqrt{\lambda}}{\sqrt{\lambda}}=K_2.
    \end{equation}
    Thus \[DF(x,\Psi(x))\mathcal C_{K_2}(x)
    \subseteq\mathcal C_{K_2}(g(x)).\]

    We next justify precisely how the limiting secants are transported. If
    \[
      q_h:=\frac{\Psi(x+h)-\Psi(x)}{h},\qquad
      e_h:=\frac{\phi(x+h)-\phi(x)}{h}-\phi'(x),
    \]
    then $e_h\to0$, and the invariance equation gives the exact identity
    \[
      \frac{\Psi(g(x+h))-\Psi(g(x))}{g(x+h)-g(x)}
      =\frac{\phi'(x)+\lambda q_h+e_h}
      {1+\phi'(x)+\lambda q_h+e_h}.
    \]
    The denominator equals $(g(x+h)-g(x))/h\geq \sqrt{\lambda}$. Hence every sequence with $q_h\to t\in\mathscr E(x)$ yields the limiting slope $\mathcal S_x(t)$ at $g(x)$. Conversely, given a converging sequence of secant slopes at $g(x)$, pull its base points back by $g^{-1}$ and extract a convergent subsequence of the resulting bounded slopes at $x$. The same identity then shows that its limit is the image under $\mathcal S_x$ of an element of $\mathscr E(x)$. Consequently,
    \[
      \mathscr E(g(x))=\mathcal S_x(\mathscr E(x)),
    \]
    and, since the horizontal component of $DF(x,\Psi(x))(1,t)$ is positive,
    \begin{equation}\label{cooncd}
       DF(x,\Psi(x))TC_\Gamma(x)=TC_\Gamma(g(x)).
    \end{equation}

    \subsubsection{Differentiability}
    We use the following elementary graph version of the tangent-cone criterion; compare \cite[Lemma 4.2]{BB}.
    \begin{Lemma}\label{difli}
      Let $u:\R\to\R$ be Lipschitz. If the set of limiting secant slopes at each $x$ is contained in a singleton $\{\ell_*(x)\}$, then $u$ is differentiable everywhere and $u'(x)=\ell_*(x)$. If $\ell_*$ is continuous, then $u$ is of class $C^1$.
    \end{Lemma}
    \begin{proof}
      The difference quotients at a fixed $x$ are bounded. If they did not tend to $\ell_*(x)$, there would be a sequence tending to $x$ whose quotients stay a fixed positive distance from $\ell_*(x)$. A convergent subsequence would then give a different limiting secant slope, a contradiction. The continuity assertion follows from $u'=\ell_*$.
    \end{proof}

    We verify this criterion quantitatively, including the critical case.
    The map $\mathcal S_x$ is strictly increasing on $[-K_2,K_2]$.
    Fix $t_1\leq t_2$ in this interval and put $d:=t_2-t_1$.
    The perturbation bound and the definition of $K_2$ give
    \[
      1+\phi'(x)+\lambda t_1
      \geq 1-(1-\sqrt{\lambda})^2-\lambda K_2=\sqrt{\lambda}.
    \]
    Since $t_2=t_1+d$, the second denominator satisfies the sharper bound
    \[
      1+\phi'(x)+\lambda t_2
      =1+\phi'(x)+\lambda t_1+\lambda d
      \geq \sqrt{\lambda}+\lambda d.
    \]
    Subtracting the two fractions in the definition of $\mathcal S_x$
    therefore yields
    \begin{align}
      0\leq\mathcal S_x(t_2)-\mathcal S_x(t_1)
      &=\frac{\lambda d}
      {(1+\phi'(x)+\lambda t_1)(1+\phi'(x)+\lambda t_2)}\notag\\
      &\leq\frac{\lambda d}
      {\sqrt{\lambda}(\sqrt{\lambda}+\lambda d)}
       =\frac{d}{1+\sqrt{\lambda}\,d}.
       \label{cone-width-critical}
    \end{align}
    The last equality follows by factoring $\lambda$ from the denominator:
    \[
      \sqrt{\lambda}(\sqrt{\lambda}+\lambda d)
      =\lambda(1+\sqrt{\lambda}d).
    \]
    Retaining the term $\lambda d$ in the second denominator is essential:
    bounding both denominators only by $\sqrt{\lambda}$ would give merely
    $\mathcal S_x(t_2)-\mathcal S_x(t_1)\leq d$.

    Set $J_0(x):=[-K_2,K_2]$ and, for $n\geq0$, define
    \[
       J_{n+1}(x):=
       \mathcal S_{g^{-1}(x)}\bigl(J_n(g^{-1}(x))\bigr).
    \]
    Here $n$ counts the applications of the slope maps along the backward
    orbit of $x$; explicitly, for $n\geq1$,
    \[
      J_n(x)=\mathcal S_{g^{-1}(x)}\circ\mathcal S_{g^{-2}(x)}
      \circ\cdots\circ\mathcal S_{g^{-n}(x)}([-K_2,K_2]).
    \]
    Each $J_n(x)$ is a nonempty closed interval. By \eqref{test2}, $J_1(x)\subseteq J_0(x)$ for every $x$; applying the defining recursion to an inclusion $J_n(y)\subseteq J_{n-1}(y)$ proves $J_{n+1}(x)\subseteq J_n(x)$. Thus the intervals are nested. The slope identity above also proves inductively that
    \[
       \mathscr E(x)\subseteq J_n(x)\qquad(n\geq0).
    \]

    To estimate their lengths, write $d_n(x):=|J_n(x)|$.
    If $y=g^{-1}(x)$ and $J_n(y)=[a,b]$, strict monotonicity of
    $\mathcal S_y$ gives
    \[
      \begin{aligned}
      d_{n+1}(x)
      &=\mathcal S_y(b)-\mathcal S_y(a)\\
      &\leq\frac{b-a}{1+\sqrt{\lambda}(b-a)}
       =\frac{d_n(g^{-1}(x))}{1+\sqrt{\lambda} d_n(g^{-1}(x))},
      \end{aligned}
    \]
    by \eqref{cone-width-critical}. The right-hand side involves the
    predecessor $g^{-1}(x)$, so we use a bound uniform in the base point.
    Define
    \[
      W_0:=2K_2,\qquad
      W_{n+1}:=\frac{W_n}{1+\sqrt{\lambda} W_n}.
    \]
    The function $d\mapsto d/(1+\sqrt{\lambda} d)$ is increasing on
    $[0,\infty)$, since
    \[
      \frac{d}{dr}\left(\frac{r}{1+\sqrt{\lambda} r}\right)
      =\frac1{(1+\sqrt{\lambda} r)^2}>0\qquad(r\geq0).
    \]
    Initially, $d_0(x)=2K_2=W_0$ for every $x$. If $d_n(y)\leq W_n$
    holds for every $y$, then the preceding estimate implies
    \[
      d_{n+1}(x)
      \leq\frac{d_n(g^{-1}(x))}{1+\sqrt{\lambda} d_n(g^{-1}(x))}
      \leq\frac{W_n}{1+\sqrt{\lambda} W_n}=W_{n+1}.
    \]
    Thus $|J_n(x)|\leq W_n$ for all $x$ and $n\geq0$.

    Because $0<\lambda<1$, we have $K_2>0$, and the recursion gives
    $W_n>0$. Taking reciprocals turns it into an additive recursion:
    \[
      \frac1{W_{n+1}}
      =\frac{1+\sqrt{\lambda} W_n}{W_n}
      =\frac1{W_n}+\sqrt{\lambda}.
    \]
    Summing from $0$ to $n-1$ yields
    \[
      \frac1{W_n}
      =\frac1{W_0}+\sum_{j=0}^{n-1}\sqrt{\lambda}
      =\frac1{2K_2}+n\sqrt{\lambda}.
    \]
    Consequently,
    \begin{equation}\label{cone-width-iterate}
       |J_n(x)|\leq W_n
       =\frac{2K_2}{1+2n\sqrt{\lambda} K_2}
       \longrightarrow0\qquad(n\to\infty).
    \end{equation}
    The bound is independent of $x$, so convergence is uniform along the
    graph. More explicitly, for each fixed $\lambda\in(0,1)$,
    \[
      \sup_{x\in\T}|J_n(x)|\leq\frac1{n\sqrt{\lambda}}
      \qquad(n\geq1).
    \]
    This gives an $O(n^{-1})$ upper bound on the slope widths, without
    requiring a uniform strict contraction factor for the derivatives
    of the slope maps.

    Nested compactness now gives
    $\bigcap_{n\geq0}J_n(x)=\{\ell_*(x)\}$ for a single slope $\ell_*(x)$. Equivalently,
    \[
      L(x):=\bigcap_{n\geq0}
      DF^n(g^{-n}(x),\Psi(g^{-n}(x)))
      \mathcal C_{K_2}(g^{-n}(x))
      =\R(1,\ell_*(x)).
    \]
    Indeed, the slopes in the $n$th cone in this intersection are exactly $J_n(x)$, and its only vector with zero horizontal component is the zero vector. Since $\mathscr E(x)\subseteq\{\ell_*(x)\}$, Lemma~\ref{difli} proves
    \[
       \Psi'(x)=\ell_*(x)\qquad\text{for every }x.
    \]

    \subsubsection{Continuous differentiability}
    Write $J_n(x)=[a_n(x),b_n(x)]$. Their endpoint functions are continuous and periodic. To see this, start from the constant functions $a_0=-K_2$, $b_0=K_2$ and use the strict monotonicity of $\mathcal S_x$ to write
    \[
      a_{n+1}(x)=\mathcal S_{g^{-1}(x)}(a_n(g^{-1}(x))),\qquad
      b_{n+1}(x)=\mathcal S_{g^{-1}(x)}(b_n(g^{-1}(x))).
    \]
    The map $(x,t)\mapsto\mathcal S_x(t)$ is continuous on $\R\times[-K_2,K_2]$, because $\phi'$ is continuous and its denominator is bounded below by $\sqrt{\lambda}$. Continuity of $g^{-1}$ therefore completes the induction; periodicity follows from $g^{-1}(x+1)=g^{-1}(x)+1$. Finally, $a_n(x)\leq\ell_*(x)\leq b_n(x)$ and \eqref{cone-width-iterate} imply
    \[
      \max\{\|a_n-\ell_*\|_{C^0},\|b_n-\ell_*\|_{C^0}\}
      \leq\frac{2K_2}{1+2n\sqrt{\lambda}\,K_2}\longrightarrow0.
    \]
    Thus $\ell_*$ is a uniform limit of continuous periodic functions, so $\Psi\in C^1(\T)$. This proves the first assertion of Theorem~\ref{M2} under the non-strict inequality \eqref{contrac}.

    \subsection{Normal hyperbolicity of the invariant graph}
    In general, for a given \( r \geq 1 \), we consider the \( r \)-normal hyperbolicity of the invariant graph under the assumption that a \( C^r \) invariant graph exists. To this end, we first establish the following two lemmas.

	\begin{Lemma}\label{lem:universal-g-lower}
		Let $\lambda\in(0,1)$, $\tau\in(\lambda,1)$.
		Suppose $F$ admits a $C^1$ invariant graph $\Gamma=\{(x,\Psi(x))\mid x\in\mathbb{T}\}$.
		If the perturbation satisfies
		\[
		\|\phi'\|_\infty < (1-\tau)\Bigl(1-\frac{\lambda}{\tau}\Bigr),
		\]
		then
		\[
		g'(x) > \tau \qquad \text{for all } x\in\mathbb{T}.
		\]
	\end{Lemma}
	
	\begin{proof}
		The invariance equation gives
        \[
        DF(x,\Psi(x))(1,\Psi'(x))^{\mathsf T}
        =g'(x)(1,\Psi'(g(x)))^{\mathsf T}.
        \]
        Since $DF$ is invertible, $g'$ never vanishes. The map $g$ is an increasing degree-one lift, so $g'>0$ and $g$ is a $C^1$ diffeomorphism. Moreover, $\int_0^1g'(x)\,dx=1$.

		Differentiating \eqref{eq:main} now gives the Herman--Mather identity:
		\begin{equation}\tag{A'}
			g'(x) + \frac{\lambda}{g'(g^{-1}(x))} = 1 + \lambda + \phi'(x), \quad \forall x \in \mathbb{T}.
		\end{equation}

		Set
		$$ b_{\lambda,\tau} = (1 - \tau)\left(1 - \frac{\lambda}{\tau}\right). $$
		The assumption $\|\phi'\|_\infty < b_{\lambda,\tau}$ implies $\phi'(x) > -b_{\lambda,\tau}$ for all $x$. Thus,
		\begin{equation}\label{12}
			g'(x) + \frac{\lambda}{g'(g^{-1}(x))} > \tau + \frac{\lambda}{\tau}, \quad \forall x \in \mathbb{T}.
		\end{equation}
		
		Assume for contradiction that the set $S = \{x \in \mathbb{T} \mid g'(x) \le \tau\} \neq \emptyset$.
		For any $x \in S$, applying $g'(x) \le \tau$ to (\ref{12}) yields:
		$$
		\tau + \frac{\lambda}{g'(g^{-1}(x))} \ge g'(x) + \frac{\lambda}{g'(g^{-1}(x))} > \tau + \frac{\lambda}{\tau},
		$$
		which implies $g'(g^{-1}(x)) < \tau$. Therefore, $g^{-1}(x) \in S$, yielding $g^{-1}(S) \subseteq S$.
		
		If $x\in S$, the strict inequality just proved shows that $g^{-1}(x)$ lies in the open set $\{g'<\tau\}\subseteq S$. By continuity, this open set is nonempty and has positive Lebesgue measure. Hence $\mu(S)>0$.
		Using the substitution $x = g(y)$, we have:
		$$
		\mu(g^{-1}(S)) = \int_{g^{-1}(S)} 1 dy = \int_S \frac{1}{g'(g^{-1}(x))} dx.
		$$
		Since $g'(g^{-1}(x)) < \tau$ on $S$, it follows that:
		$$
		\mu(g^{-1}(S)) > \frac{1}{\tau} \mu(S).
		$$
		Because $\tau < 1$ and $\mu(S) > 0$, we obtain $\mu(g^{-1}(S)) > \mu(S)$, contradicting $g^{-1}(S) \subseteq S$.
		Thus, $S$ is empty, proving $g'(x) > \tau$ for all $x \in \mathbb{T}$.
	\end{proof}
	
	\begin{Lemma}\label{lem:nh-iff}
        Fix an integer $r\geq1$. Let $\lambda\in(0,1)$ and suppose that $F$ is $C^r$ and admits a $C^r$ invariant graph $\Gamma=\{(x,\Psi(x)):x\in\T\}$. If
        \[
          g'(x)>\lambda^{1/(r+1)}\qquad\text{for every }x\in\T,
        \]
        then $\Gamma$ is $r$-normally hyperbolic for $F$.
    \end{Lemma}
    \begin{proof}
        Set $m_g:=\min_xg'(x)>\lambda^{1/(r+1)}$ and $\theta(x):=\lambda/g'(x)$. We first compute the normal quotient map. Define
        \[
          \pi_x(u,v):=v-u\Psi'(x).
        \]
        Its kernel is $T_{(x,\Psi(x))}\Gamma$, so it identifies the one-dimensional quotient $T_{(x,\Psi(x))}(\T\times\R)/T_{(x,\Psi(x))}\Gamma$ with $\R$. Differentiating graph invariance gives
        \[
          g'=1+\phi'+\lambda\Psi',\qquad
          \Psi'(g(x))=\frac{\phi'(x)+\lambda\Psi'(x)}{g'(x)}.
        \]
        Therefore, for every $(u,v)\in\R^2$,
        \begin{align*}
          \pi_{g(x)}DF(x,\Psi(x))(u,v)
          &=\phi'(x)u+\lambda v
          -\bigl((1+\phi'(x))u+\lambda v\bigr)\Psi'(g(x))\\
          &=\frac{\lambda}{g'(x)}\bigl(v-u\Psi'(x)\bigr)
           =\theta(x)\pi_x(u,v).
        \end{align*}

        Take the basis $e_1(x):=(1,\Psi'(x))^{\mathsf T}$, $e_2:=(0,1)^{\mathsf T}$. Direct computation gives
        \[
          DF(x)e_1(x)=g'(x)e_1(g(x)),\qquad
          DF(x)e_2=\lambda e_1(g(x))+\theta(x)e_2,
        \]
        where $DF(x)$ abbreviates $DF(x,\Psi(x))$. Thus the derivative in this frame is triangular. Look for a stable vector
        $w_{\mathrm s}(x):=c_{\mathrm s}(x)e_1(x)+e_2$ satisfying $DF(x)w_{\mathrm s}(x)=\theta(x)w_{\mathrm s}(g(x))$. This is equivalent to
        \[
          c_{\mathrm s}(x)=\frac{\lambda}{g'(x)^2}c_{\mathrm s}(g(x))-\frac{\lambda}{g'(x)}.
        \]
        Since $r\geq1$ and $0<\lambda<1$,
        \[
          \kappa_{\mathrm{nh}}:=\frac{\lambda}{m_g^2}
          <\lambda^{1-2/(r+1)}\leq1.
        \]
        The right-hand side defines a strict contraction on $C^0(\T)$ and hence has a unique continuous fixed point $c_{\mathrm s}$. The line field $E^s_x:=\R w_{\mathrm s}(x)$ is continuous, transverse to $T\Gamma$, and $DF$-invariant.

        Define a continuous metric along $\Gamma$ by making the frame $\{e_1(x),w_{\mathrm s}(x)\}$ orthonormal. In this metric,
        \[
          \|DF|_{E^s_x}\|=\frac{\lambda}{g'(x)},\qquad
          m(DF|_{T_x\Gamma})=g'(x).
        \]
        Consequently,
        \[
          \sup_x\|DF|_{E^s_x}\|=\frac{\lambda}{m_g}<1,
          \qquad
          \frac{\lambda}{m_g}<m_g^r
          =\inf_xm(DF|_{T_x\Gamma})^r.
        \]
        These are the required uniform contraction and domination inequalities. If a smooth ambient metric is desired, approximate this continuous metric by a smooth metric near the compact graph; the strict inequalities persist. This proves $r$-normal hyperbolicity.
    \end{proof}
    It suffices to set  \( \tau = \sqrt{\lambda} \) in Lemma~\ref{lem:universal-g-lower} and \( r = 1 \) in Lemma~\ref{lem:nh-iff} to obtain the desired conclusion.
	\begin{Corollary}\label{lexx2}
		Given \(\lambda \in (0,1)\), for any \(\alpha \in \mathbb{R}^2\) and any \(C^1\) perturbation \(\phi\) with zero average satisfying
		\begin{equation}\label{m22conx}
			\|\phi\|_{C^1} < (1-\sqrt{\lambda})^2,
		\end{equation}
		the  \(C^1\) invariant graph persisted by the map \(F\)  is \(1\)-normally hyperbolic.
	\end{Corollary}

	\vspace{1em}

\section{\sc Proof of Theorem \ref{Th1}}
Recall the Herman--Mather formula \eqref{eq:main}.  The construction
below uses a rational partition and cancellation of the two
one-sided derivatives, but controls both signs of the derivative of
the perturbation.

\begin{Proposition}\label{hf1}
For any $\lambda\in(0,1)$, any rational number $\beta=p/q$ in lowest
terms with $q\geq1$, any $a^*>(1-\sqrt\lambda)^2$, and any $s>1$,
there exists a bi-Lipschitz lift $g\in D^0(\mathbb T)$ with
$\rho(g)=\beta$ such that $g$ is non-differentiable precisely at
$q$ points modulo $1$, whereas
\[
G(x):=g(x)+\lambda g^{-1}(x)
\]
belongs to the Gevrey class $G^s$ and satisfies
\begin{equation}\label{Gcond}
\max_{x\in[0,1]}|G'(x)-(1+\lambda)|=a^*.
\end{equation}
\end{Proposition}

We first construct the circle map and establish the norm estimate.
We then verify the Gevrey regularity at the partition points and
recover the zero-average perturbation and its invariant graph.

\subsection{The template and the $C^1$ construction}

\medskip
\noindent\textit{Step 1. Constants and notation.}
Put
\[
a:=\frac{1+\sqrt{\lambda}}{2},\qquad b:=\frac{\lambda}{a}.
\]
Then
\begin{equation}\label{eq:template-endpoints}
0<b<\sqrt{\lambda}<a<1,\qquad ab=\lambda,\qquad
a+\frac\lambda a=b+\frac\lambda b=a+b<1+\lambda.
\end{equation}
The equality in this display will make the one-sided derivatives of
$G$ agree.  Unlike a monotone template decreasing from a value above
$1$, the template used here has both endpoint values below $1$ and
has average $1$ because its interior values exceed $1$.

\medskip
\noindent\textit{Step 2. Partition induced by a rational rotation.}
Set $y_k:=k/q$ for all $k\in\mathbb Z$ and $L_q:=1/q$.
Modulo $1$, these are precisely the points in the orbit of $0$
under rotation by $p/q$.  Thus rotation by $p/q$ sends $y_k$ to $y_{k+p}$ and
permutes the $q$ partition intervals.  This notation also covers
$q=1$ and arbitrary integer parts of $\beta$.

\medskip
\noindent\textit{Step 3. A smooth template near the critical norm.}
We use the following elementary flat cutoff; its Gevrey estimate is
proved in the next subsection.

\begin{Lemma}\label{lemma1}
For every $s>1$ there exists $\zeta\in G^s(\mathbb R)$ such that
$\zeta=0$ on $(-\infty,0]$, $\zeta=1$ on $[1,\infty)$,
$\zeta'>0$ on $(0,1)$, and every positive-order derivative of
$\zeta$ vanishes at $0$ and $1$.
\end{Lemma}

Fix such a function and put $M:=\|\zeta'\|_{C^0}$.  For
$0<h<1/4$ and $\ell>0$, define on $[0,1]$
\[
\chi(x):=
\begin{cases}
\zeta\bigl(\ell^{-1}\log(h/x)\bigr),&x>0,\\
1,&x=0.
\end{cases}
\]
Thus $\chi=1$ on $[0,he^{-\ell}]$, $\chi=0$ on $[h,1]$,
and
\begin{equation}\label{eq:log-cutoff}
|x\chi'(x)|\leq M/\ell\qquad(x>0).
\end{equation}
The cutoff is smooth and belongs to $G^s$: near $0$ it is constant,
and on the remaining compact interval the logarithm is analytic.
Let
\[
J:=\int_0^1\chi(x)\,dx,\qquad
\delta:=\frac{J(2-a-b)}{1-2J},
\]
and define
\begin{equation}\label{eq:low-template}
\psi_0(x):=1+\delta-(1+\delta-a)\chi(x)
                         -(1+\delta-b)\chi(1-x).
\end{equation}
The supports of the two cutoffs are disjoint.  Consequently,
\begin{equation}\label{eq:template-properties}
\psi_0(0)=a,\quad \psi_0(1)=b,\quad
b\leq\psi_0\leq1+\delta,\quad
\int_0^1\psi_0=1.
\end{equation}
The function is constant near each endpoint, hence all its
positive-order endpoint derivatives vanish.  Also
$0<J\leq h$ and $0<\delta<4h$.

\medskip
\noindent\textit{Step 4. Construction of $g$ and its rotation number.}
For the moment let $\psi$ be any positive smooth function on $[0,1]$
such that
\begin{equation}\label{eq:admissible-template}
\psi(0)=a,\quad \psi(1)=b,\quad \int_0^1\psi=1,\quad
\psi^{(j)}(0)=\psi^{(j)}(1)=0\quad(j\geq1).
\end{equation}
Put $P(u):=\int_0^u\psi(t)\,dt$ and, for $0\leq u\leq1$, set
\begin{equation}\label{eq:rational-template-map}
g(y_k+L_q u):=y_{k+p}+L_q P(u),\qquad k\in\mathbb Z.
\end{equation}
The formulas agree at the endpoints.  The resulting map is a
bi-Lipschitz increasing homeomorphism with $g(x+1)=g(x)+1$.
It is smooth inside every partition interval, and its right and
left derivatives at $y_k$ are respectively $a$ and $b$.
Thus these, and only these, are its non-differentiability points.
Moreover $g(y_k)=y_{k+p}$, so $g^q(y_k)=y_k+p$ and
$\rho(g)=p/q$.

For $x=g(y_k+L_q u)$ with $0<u<1$, direct differentiation gives
\begin{equation}\label{eq:template-Gprime}
G'(x)=\psi(P(u))+\frac{\lambda}{\psi(u)}.
\end{equation}
At every partition point the two limits of this expression are
$a+\lambda/a$ and $b+\lambda/b$.  They agree by
\eqref{eq:template-endpoints}; hence $G\in C^1(\mathbb R)$.
In particular, if
\begin{equation}\label{eq:template-norm}
\mathcal N(\psi):=
\max_{0\leq u\leq1}
\left|\psi(P(u))+\frac\lambda{\psi(u)}-(1+\lambda)\right|,
\end{equation}
then $\mathcal N(\psi)=\|G'-(1+\lambda)\|_{C^0}$.

\medskip
\noindent\textit{Step 5. Two-sided control near the critical value.}

\begin{Lemma}\label{lemma2}
For every $\eta>0$, the parameters $h$ and $\ell$ in
\eqref{eq:low-template} can be chosen so that
\begin{equation}\label{eq:near-critical-template}
(1-\sqrt{\lambda})^2<\mathcal N(\psi_0)<(1-\sqrt{\lambda})^2+\eta.
\end{equation}
\end{Lemma}

\begin{proof}
Write $P_0(u):=\int_0^u\psi_0(t)\,dt$.  By
\eqref{eq:template-properties}, for $0<u<1$,
\[
b\leq\frac{P_0(u)}u\leq1+\delta,\qquad
b\leq\frac{1-P_0(u)}{1-u}\leq1+\delta.
\]
Integrating \eqref{eq:log-cutoff} between any two positive points
gives
\[
|\chi(v)-\chi(u)|\leq\frac M\ell|\log(v/u)|.
\]
Apply this inequality to $u,P_0(u)$ and then to
$1-u,1-P_0(u)$.  We obtain
\begin{equation}\label{eq:slow-variation}
|\psi_0(P_0(u))-\psi_0(u)|\leq \mathcal E,
\quad
\mathcal E:=\frac M\ell\bigl(2(1+\delta)-a-b\bigr)
                  \max\{-\log b,\log(1+\delta)\}.
\end{equation}
The same bound holds at the endpoints by continuity.
The function $z\mapsto z+\lambda/z$ has minimum $2\sqrt{\lambda}$ and,
on $[b,1+\delta]$, has maximum at an endpoint.  Since
$b+\lambda/b=a+b<1+\lambda$, it follows that
\[
2\sqrt{\lambda}-\mathcal E\leq
\psi_0(P_0(u))+\frac\lambda{\psi_0(u)}
\leq1+\lambda+\delta+\mathcal E.
\]
Choose $h$ so small that $4h<\eta/4$.  Then choose $\ell$ large
enough that $\mathcal E<\eta/4$.  This is possible uniformly in $J$ because
$\delta<4h<1$ after decreasing $h$ if necessary.
We conclude that
\[
\mathcal N(\psi_0)
\leq\max\{(1-\sqrt{\lambda})^2+\mathcal E,\delta+\mathcal E\}<(1-\sqrt{\lambda})^2+\eta.
\]

For strictness of the lower bound, on the right transition interval
$\psi_0$ decreases strictly from $1+\delta$ to $b<\sqrt{\lambda}$.
There is therefore a point $u_*$ in that interval with
$\psi_0(u_*)=\sqrt{\lambda}$.  On $[u_*,1]$ we have $\psi_0\leq \sqrt{\lambda}<1$,
and hence
\[
1-P_0(u_*)=\int_{u_*}^1\psi_0(t)\,dt
\leq \sqrt{\lambda}(1-u_*)<1-u_*.
\]
Thus $P_0(u_*)>u_*$.  Strict decrease up to the final constant
segment gives $\psi_0(P_0(u_*))<\sqrt{\lambda}$, whence
\[
\psi_0(P_0(u_*))+\frac\lambda{\psi_0(u_*)}<2\sqrt{\lambda}.
\]
Subtracting $1+\lambda$ yields a value strictly below $-(1-\sqrt{\lambda})^2$,
and proves $\mathcal N(\psi_0)>(1-\sqrt{\lambda})^2$.
\end{proof}

\medskip
\noindent\textit{Step 6. Attaining the prescribed norm exactly.}
Apply Lemma~\ref{lemma2} with $\eta:=a^*-(1-\sqrt{\lambda})^2$ to obtain
$\mathcal N(\psi_0)<a^*$.  We next construct a second admissible
template whose norm exceeds $a^*$, keeping the endpoint values
$a,b$ unchanged.
Let
\[
w(u):=a+(b-a)\zeta(u),\qquad
d_0:=1-\int_0^1w(u)\,du>0.
\]
Here $b\leq w\leq a<1$, and $w$ has the endpoint jets specified
in \eqref{eq:admissible-template}.  Choose a nonnegative
$G^s$ function $\mathfrak b$ supported in $[-1,1]$ with $\int \mathfrak b=1$ and
$\mathfrak b(0)>0$, and set, for $0<\varepsilon<1/4$,
\[
\mathfrak b_\varepsilon(u):=\varepsilon^{-1}
\mathfrak b\bigl((u-1/2)/\varepsilon\bigr),\qquad
\psi_1(u):=w(u)+d_0 \mathfrak b_\varepsilon(u).
\]
Then $\psi_1$ is positive, satisfies
\eqref{eq:admissible-template}, and belongs to $G^s$.
Moreover $\max\psi_1\to\infty$ as $\varepsilon\to0$.
Since $P_1(u):=\int_0^u\psi_1$ maps $[0,1]$ onto itself,
\[
\mathcal N(\psi_1)
\geq\max\psi_1-(1+\lambda)\longrightarrow\infty.
\]
Fix $\varepsilon$ so that $\mathcal N(\psi_1)>a^*$ and put
\[
\psi_t:=(1-t)\psi_0+t\psi_1,\qquad 0\leq t\leq1.
\]
Every $\psi_t$ is bounded below by $b$ and satisfies
\eqref{eq:admissible-template}.  The maps
$P_t(u):=\int_0^u\psi_t$ depend continuously on $t$ in the
uniform norm.  The expression inside the maximum in
\eqref{eq:template-norm} is therefore jointly continuous in
$(t,u)\in[0,1]^2$, so $t\mapsto\mathcal N(\psi_t)$ is continuous.
The intermediate value theorem gives $t_*\in(0,1)$ such that
\[
\mathcal N(\psi_{t_*})=a^*.
\]
Use this template in \eqref{eq:rational-template-map}.  This proves
\eqref{Gcond} and all the assertions of Proposition~\ref{hf1}
except the claimed Gevrey regularity of $G$.

\subsection{From $C^1$ to Gevrey regularity}

Recall that $f\in G^s$ on a compact interval if there are constants
$C,R>0$ such that
\[
\|f^{(k)}\|_{C^0}\leq C R^k(k!)^s\qquad(k\geq0).
\]
For $s=1$ this is the analytic class.  For $s>1$ it admits
nonzero flat functions and is closed under composition and under
inversion of diffeomorphisms with derivative bounded away from
zero; see, for example, \cite[Appendix B]{BF}.

\subsubsection{Gevrey flat template}
\begin{proof}[Proof of Lemma~\ref{lemma1}]
Put $\kappa:=1/(s-1)$ and define
\[
H(x):=
\begin{cases}e^{-x^{-\kappa}},&x>0,\\0,&x\leq0.
\end{cases}
\qquad
\xi_s(x):=H(x)H(1-x).
\]
Every derivative of $H$ on $(0,\infty)$ is a finite sum of
real powers of $x$ times $e^{-x^{-\kappa}}$.  Since the
exponential decays faster than every inverse power as $x\downarrow0$,
all these derivatives extend continuously by zero at $0$.
Thus $H$ is smooth and flat at $0$, and $\xi_s$ is smooth,
positive on $(0,1)$, and flat at both endpoints.

For completeness, choose $c_s\in(0,1)$ sufficiently small that
on $|z-x|\leq c_s x$, with $x>0$, one has
\[
\Re(z^{-\kappa})\geq c'_s x^{-\kappa}>0
\]
for a fixed $c'_s>0$, using the branch that is positive on the
positive real axis.  Cauchy's estimate yields
\[
|H^{(k)}(x)|\leq k!(c_sx)^{-k}e^{-c'_s x^{-\kappa}}.
\]
With $v=c'_s x^{-\kappa}$ and
$\sup_{v>0}v^{k/\kappa}e^{-v}=(k/(e\kappa))^{k/\kappa}$,
we obtain
\[
\sup_{0<x\leq1}|H^{(k)}(x)|
\leq C R^k k!\,(k!)^{1/\kappa}
=C R^k(k!)^s.
\]
The constants may be enlarged to include $k=0$.
Leibniz's rule gives the same type of estimate for $\xi_s$.
Now define
\[
\zeta(x):=\frac{\int_{-\infty}^x\xi_s(t)\,dt}
                 {\int_0^1\xi_s(t)\,dt}.
\]
Integration preserves the stated Gevrey estimates, and the
positivity and flatness properties give all the assertions.
The bump $\mathfrak b$ used in Step~6 can likewise be obtained by translating,
rescaling, and normalizing $\xi_s$.
\end{proof}

\subsubsection{Gevrey regularity of $G$ at the break points}
For the chosen template $\psi=\psi_{t_*}$, both $P$ and $P^{-1}$
belong to $G^s([0,1])$, since $P'=\psi\geq b>0$.
For $x=y_k+L_q v$ inside a partition interval, differentiation of
\eqref{eq:rational-template-map} gives the equivalent formula
\begin{equation}\label{eq:Gprime-target-coordinate}
G'(y_k+L_q v)=\psi(v)+\frac\lambda{\psi(P^{-1}(v))}.
\end{equation}
The right-hand side belongs to $G^s([0,1])$.  At $v=0$ its value is
$a+\lambda/a$, and at $v=1$ its value is $b+\lambda/b$;
these values agree.  Every positive-order derivative of this
right-hand side vanishes at each endpoint: this follows from
\eqref{eq:admissible-template} and the chain rule, also for the
reciprocal term because $\psi\geq b$.
Consequently all one-sided jets of $G'$ match at every $y_k$.
There are only $q$ pieces on the circle, and their Gevrey bounds
are uniform after the fixed rescaling by $L_q$.  Their matching
jets therefore give a global $G^s$ function $G'$, and integration
gives $G\in G^s$ on every compact interval.  This completes the
proof of Proposition~\ref{hf1}.

\subsubsection{Recovery of the perturbation and the invariant graph}
The function $G(x)-(1+\lambda)x$ is periodic.  Define
\[
c:=\int_0^1\bigl(G(x)-(1+\lambda)x\bigr)\,dx,\qquad
\alpha_1:=\frac{c}{1-\lambda},\qquad\alpha_2:=0,
\]
and set
\begin{equation}\label{eq:reconstruct-supercritical}
\phi(x):=G(x)-(1+\lambda)x-c,\qquad
\Psi(x):=x-g^{-1}(x)-\alpha_1.
\end{equation}
Then $\phi\in G^s(\mathbb T)$ has zero average, while $\Psi$ is
periodic and Lipschitz.  Direct substitution, using
$c=(1-\lambda)\alpha_1$, gives
\[
\lambda\Psi(x)+\phi(x)=g(x)-x-\alpha_1,
\qquad
F^\phi_{\lambda,(\alpha_1,0)}(x,\Psi(x))
       =(g(x),\Psi(g(x))).
\]
Thus the graph is invariant, and its induced dynamics has rotation
number $\beta$.  At each $y_k$, the one-sided derivatives of
$g^{-1}$ are $1/a$ and $1/b$.  Therefore the one-sided derivatives
of $\Psi$ are $1-1/a$ and $1-1/b$, which are distinct.  These are
exactly its $q$ non-differentiability points modulo $1$.
The uniqueness argument for continuous invariant graphs in
Section~2 applies without a smallness assumption on $\phi$.

Finally, a continuous zero-average periodic function has a zero
$x_0$.  Integrating its derivative along an arc of length at most
$1/2$ from $x_0$ gives
\[
\|\phi\|_{C^0}\leq\tfrac12\|\phi'\|_{C^0}.
\]
Consequently \eqref{Gcond} and the norm convention of the
introduction imply
\[
\|\phi\|_{C^1}=\|\phi'\|_{C^0}=a^*.
\]
This proves Theorem~\ref{Th1}, in fact with the perturbation in any
prescribed Gevrey class $G^s$, $s>1$.

\section{\sc Verification of Remark \ref{R2}}\label{R3}
	We regard $A$, $B$, and $C$ as sets of parameter pairs $(\alpha,\phi)$, as in Remark~\ref{R2}. Theorem~\ref{M2} yields $A\subseteq B$, and the definition of $B$ yields $B\subseteq C$. We show that both inclusions are strict. Fixing the parameter pair is essential: failure of normal hyperbolicity for one value of $\alpha$ does not exclude normal hyperbolicity for another value of $\alpha$.
	
	\subsection{\(B \neq A\)}
	
	First, we construct an example showing that \(B \neq A\).
	\begin{Proposition}\label{lem:counterexample}
		For every $\lambda\in(0,1)$ there exists $\alpha\in\mathbb{R}^2$ and
		$\phi_*\in C^1(\mathbb{T})$ with zero average such that
		\begin{itemize}
			\item the map $F_*$ admits a $C^1$ invariant graph
			$\Gamma_*=\{(x,\Psi_*(x))\}$ which is $1$-normally hyperbolic,
			\item $\|\phi_*\|_{C^1} > (1-\sqrt{\lambda})^2$.
		\end{itemize}
	\end{Proposition}
	
	\begin{proof}
		Define $f_*(x):=x+\frac{\lambda}{1-x}-\lambda.$
		
		Since $f_*(x)$ is continuous on $[0,1-\sqrt{\lambda}]$, and
		\[
		f_*(1-\sqrt{\lambda})=1-\lambda>(1-\sqrt{\lambda})^2,\, f_*(0)=0.
		\]
		There exists a $t\in(0,1-\sqrt{\lambda})$ such that
		\[
		f_*(t)=t+\frac{\lambda}{1-t}-\lambda > (1-\sqrt{\lambda})^2.
		\]
		Set
		\[
		a' = 1+t,\qquad b' = 1-t,
		\]
		so that \(a'>1>b'>\sqrt{\lambda}\) and \(a'+b'=2\).
		
		Define a 1-periodic continuous function $p_*:\mathbb{R}\to (0,\infty)$ on $[0,1]$ by:
		\[
		p_*(x):=\begin{cases}
			a', & 0\leq x\leq \dfrac{1}{8}\\[4pt]
			1-8(a'-1)(x-\dfrac{1}{4}), & \dfrac{1}{8}\le x\le \dfrac{1}{4},\\[4pt]
			1, & \dfrac{1}{4}\le x\le \dfrac{3}{8},\\[4pt]
			1-8(1-b')(x-\dfrac{3}{8}), & \dfrac{3}{8}\le x\le \dfrac{1}{2},\\[4pt]
			b', & \dfrac{1}{2}\le x\le \dfrac{5}{8},\\[4pt]
			1+8(1-b')(x-\dfrac{3}{4}), & \dfrac{5}{8}\le x\le \dfrac{3}{4},\\[4pt]
			1, & \dfrac{3}{4}\le x\le \dfrac{7}{8},\\[4pt]
			1+8(a'-1)(x-\dfrac{7}{8}), & \dfrac{7}{8}\le x\le 1.
		\end{cases}
		\]
		The formulas agree at the endpoints of the subintervals, including $0$ and $1$, so $p_*$ is continuous and periodic. Moreover, $p_*(x+1/2)=2-p_*(x)$, and hence $\int_0^1p_*(x)\,dx=1$. Its values lie in $[b',a']$.
		
		Define
		\[
		g_*(x) = d + \int_{0}^{x} p_*(t)\,dt,\qquad x\in\mathbb{R},
		\]
		where the constant \(d\) is chosen so that \(g_*(1/2)=0\).  Since
		\(\int_0^1 p_* = 1\), we have \(g_*(x+1)=g_*(x)+1\); thus \(g_*\) is a \(C^1\)
		diffeomorphism of the line covering a circle diffeomorphism.  Its
		derivative satisfies \(g_*'(x)=p_*(x)\ge b'>\sqrt{\lambda}>0\) for all \(x\), so $g_*$ belongs to $C^1$ and has a $C^1$ inverse $g_*^{-1}$.

		By construction, \(g_*(1/2)=0\), then the preimage of $\zeta_0=0$ under \(g_*\) is
		\(g_*^{-1}(\zeta_0)=1/2\).  Therefore
		\[
		g_*'(\zeta_0)=p_*(0)=a',\qquad
		(g_*^{-1})'(\zeta_0) = \frac{1}{g_*'(g_*^{-1}(\zeta_0))} = \frac{1}{g_*'(1/2)} = \frac{1}{b'}.
		\]

		Set
		\[
		\widetilde\phi(x) = g_*(x) + \lambda g_*^{-1}(x) - (1+\lambda)x.
		\]
		This is a \(1\)-periodic function.  Let \(\bar\phi = \int_0^1 \widetilde\phi\)
		and define \(\phi_* = \widetilde\phi - \bar\phi\); then \(\int_0^1 \phi_*=0\) and
		\(\phi_*'=\widetilde\phi'\).
		At the point \(\zeta_0=0\) we have
		\[
		\phi_*'(0) = a' + \frac{\lambda}{b'} - (1+\lambda) = t + \frac{\lambda}{1-t} - \lambda = f_*(t).
		\]
		By the choice of \(t\), \(f_*(t) > (1-\sqrt{\lambda})^2\), so
		\[\|\phi_*\|_{C^1}= \|\phi_*'\|_\infty > (1-\sqrt{\lambda})^2.\]

		Choose \(\alpha_1\) arbitrarily (e.g. \(\alpha_1=0\)) and pick \(\alpha_2\) so that
		\(\bar\phi = (1-\lambda)\alpha_1 + \lambda\alpha_2\).  Then the Herman-Mather formula
		\[
		g_*(x) + \lambda g_*^{-1}(x) = (1+\lambda)x + (1-\lambda)\alpha_1 + \lambda\alpha_2 + \phi_*(x)
		\]
		holds.  Defining \(\Psi_*(x) = \frac{1}{\lambda}\bigl(g_*(x)-x-\alpha_1-\phi_*(x)\bigr)\)
		yields a $C^1$ invariant graph $\Gamma_*=\{(x,\Psi_*(x))\}$. Indeed, the Herman--Mather identity gives
\[
\Psi_*(x)=x-g_*^{-1}(x)-\alpha_1+\alpha_2,\qquad
\Psi_*(g_*(x))=g_*(x)-x-\alpha_1+\alpha_2,
\]
and therefore $F_*(x,\Psi_*(x))=(g_*(x),\Psi_*(g_*(x)))$.

		Because $g_*'(x)\ge b'>\sqrt{\lambda}$, the sufficient criterion in Lemma~\ref{lem:nh-iff} with $r=1$ shows that $\Gamma_*$ is $1$-normally hyperbolic. Thus $(\alpha,\phi_*)\in B\setminus A$.
	\end{proof}
	\subsection{\(C\neq B\)}

For the family considered here, the loss of normal hyperbolicity can already occur at the critical norm even though the invariant graph remains $C^1$. The obstruction is detected at a fixed point, where it is independent of the choice of metric.

\begin{Proposition}\label{lexx33}
For every $\lambda\in(0,1)$, one can choose $\alpha\in\mathbb R^2$ and a $C^\infty$ perturbation $\phi^*$ with zero average satisfying
\begin{equation}\label{m22conx5}
\|\phi^*\|_{C^1}=(1-\sqrt\lambda)^2,
\end{equation}
such that the map $F^*$ admits a unique $C^1$ invariant graph, but this invariant graph is not $1$-normally hyperbolic.
\end{Proposition}

We use the following fixed-point obstruction.

\begin{Lemma}\label{16p}
Suppose that $\alpha=(0,0)$, that $\phi$ is odd and $C^1$, and that $F=F_{\lambda,0}^{\phi}$ admits a unique $C^1$ invariant graph $\Gamma=\{(x,\Psi(x))\}$. If $\phi'(0)=-(1-\sqrt{\lambda})^2$, then
\[
\Psi(0)=0,\qquad g'(0)=\sqrt{\lambda},\qquad \Psi'(0)=1-\lambda^{-1/2},
\]
where
\begin{equation}\label{dfg}
g(x)=x+\lambda\Psi(x)+\phi(x).
\end{equation}
In particular, $\Gamma$ is not normally hyperbolic.
\end{Lemma}

\begin{proof}[Proof of Proposition~\ref{lexx33}]
Choose
\[
\alpha=(0,0),\qquad
\phi^*(x):=-\frac{(1-\sqrt{\lambda})^2}{2\pi}\sin(2\pi x).
\]
This function is odd, analytic, and has zero average. Moreover,
\[
\|\phi^*\|_{C^0}=\frac{(1-\sqrt{\lambda})^2}{2\pi},\qquad
\|(\phi^*)'\|_{C^0}=(1-\sqrt{\lambda})^2,
\]
so the norm convention used in the introduction gives $\|\phi^*\|_{C^1}=(1-\sqrt{\lambda})^2$. This proves \eqref{m22conx5}. Theorem~\ref{M2}, including its critical endpoint, provides a unique $C^1$ invariant graph. Finally, $(\phi^*)'(0)=-(1-\sqrt{\lambda})^2$, so Lemma~\ref{16p} excludes normal hyperbolicity. Thus $(0,\phi^*)\in C\setminus B$.
\end{proof}

\begin{proof}[Proof of Lemma~\ref{16p}]
Let $S(x,y):=(-x,-y)$. Oddness of $\phi$ gives $F\circ S=S\circ F$. Consequently, the graph of $x\mapsto-\Psi(-x)$ is invariant whenever the graph of $\Psi$ is invariant. Uniqueness implies $\Psi(-x)=-\Psi(x)$ and hence $\Psi(0)=0$. Since $\phi(0)=0$, the point $(0,0)$ is fixed by $F$.

At this fixed point the derivative is
\[
DF(0,0)=
\begin{pmatrix}
1-(1-\sqrt{\lambda})^2&\lambda\\
-(1-\sqrt{\lambda})^2&\lambda
\end{pmatrix}.
\]
Its determinant is $\lambda$ and its trace is
\[
1+\lambda-(1-\sqrt{\lambda})^2=2\sqrt{\lambda}.
\]
Thus its characteristic polynomial is $(\mu-\sqrt{\lambda})^2$. The matrix is not $\sqrt{\lambda}\,\mathrm{Id}$, because its upper-right entry is $\lambda>0$; it therefore has only one eigendirection. The tangent line to the $C^1$ invariant graph is invariant under $DF(0,0)$, so it is precisely this eigendirection. Differentiating the invariance identity at $0$ gives
\[
DF(0,0)\binom{1}{\Psi'(0)}
=g'(0)\binom{1}{\Psi'(0)}.
\]
Hence $g'(0)=\sqrt{\lambda}$, and the first row gives
\[
\Psi'(0)=\frac{\sqrt{\lambda}-1+(1-\sqrt{\lambda})^2}{\lambda}=1-\lambda^{-1/2}.
\]
Normal hyperbolicity would require an invariant normal line complementary to the tangent line at the fixed point. Such a line does not exist, because the derivative has only one eigendirection. Equivalently, the tangent and normal quotient multipliers are both $\sqrt{\lambda}$, so there is no strict domination. This proves the lemma.
\end{proof}

\noindent\textbf{Data Availability Statement.}
		The authors state that this manuscript has no associated data and there is no conflict of interest.

	\end{document}